\documentclass{article}
\usepackage[a4paper,margin=0.64in,top=2.5cm, bottom=2.5cm]{geometry}
\usepackage[utf8]{inputenc}
\usepackage{amsmath,amsfonts,amssymb,amsthm,bbm,mathrsfs}
\usepackage[shortlabels]{enumitem}
\usepackage{xcolor}
\usepackage{graphicx}
\usepackage{appendix}
\usepackage{setspace}
\usepackage{fancyhdr}

\usepackage{hyperref}
\hypersetup{
    colorlinks=true,
    linkcolor=black,
    filecolor=magenta,      
    urlcolor=black,
    pdftitle={Carleman with GDBC and Applications},
    citecolor=black,
    pdfpagemode=FullScreen,
}

\newtheorem{thm}{Theorem}[section]
\newtheorem{df}{Definition}[section]

\newtheorem{rmk}{Remark}[section]

\newtheorem{prop}{Proposition}[section]
\newtheorem{lm}{Lemma}[section]
\newtheorem{cor}{Corollary}[section]
\numberwithin{equation}{section}

\makeatletter
\def\@setauthors{%
  \begingroup
  \def\thanks{\protect\thanks@warning}%
  \trivlist
  \centering\footnotesize \@topsep30\p@\relax
  \advance\@topsep by -\baselineskip
  \item\relax
  \author@andify\authors
  \def\\{\protect\linebreak}%
  \authors%
  \ifx\@empty\contribs
  \else
    ,\penalty-3 \space \@setcontribs
    \@closetoccontribs
  \fi
  \endtrivlist
  \endgroup
}
\makeatother
\fancypagestyle{firstpage}{
\fancyhf{} % clear headers and footers

\fancyfoot[L]{%
\footnotesize
\rule{6cm}{0.7pt}\\ 
This article is based upon work from the project PRIN2022 D53D23005580006 ``Elliptic and parabolic problems, heat kernel estimates and spectral theory''. The fourth author is member of the ``Gruppo Nazionale per l’Analisi Matematica, la Probabilità e le loro Applicazioni (GNAMPA)'' of the Istituto Nazionale di Alta Matematica (INdAM).
  }
}

\title{\textbf{Carleman Estimates for Backward Anisotropic  Stochastic Parabolic Equations with General Dynamic Boundary Conditions and Applications}}

\author{Said Boulite$^\dagger$, \,Abdellatif Elgrou$^\ddagger$,\, Lahcen Maniar$^{\S,\top}$ \,and\, Abdelaziz Rhandi$^\ddagger$}

\date{}

\begin{document}

% Title and Author Information
\maketitle
\vspace{-2em}

\thispagestyle{firstpage}
\begin{center}
\end{center}

\begin{abstract}
We investigate a backward anisotropic stochastic parabolic equation with general dynamic boundary conditions, where the drift involves both $\mathbb{L}^2$ and $\mathbb{H}^{-1}$ bulk--surface terms. We first establish the well-posedness of this equation. Subsequently, we derive a new Carleman estimate through a two-step approach. In the first step, using a weighted identity method together with a careful treatment of the boundary integral terms arising from  the dynamic boundary conditions, we obtain an intermediate Carleman estimate for backward anisotropic  stochastic parabolic equations without weak divergence source terms. In the second step, a duality method combined with suitable optimization techniques is employed to incorporate the weak divergence source terms. As applications of the derived Carleman estimate, we address two control problems. First, we establish null controllability for forward anisotropic  stochastic parabolic equations with general dynamic boundary conditions. These equations involve both reaction and convection terms, with adapted, bounded stochastic bulk--surface coefficients. Moreover, we provide an explicit estimate of the null controllability cost, i.e., a bound on the minimal norm of controls required to drive the system to zero at the terminal time $T$. Second, we study an insensitizing control problem for this class of equations. The goal is to determine controls for systems with partially unknown initial data such that a given energy functional remains insensitive to small perturbations of these data. In this work, the functional involves the norm of the state over a localized bulk--surface region, together with the norm of its tangential gradient over a localized boundary region.
\end{abstract}

\vspace{1em}
%\tableofcontents
\section{Introduction and Main Results}
Fix a final time $T > 0$, and let $G \subset \mathbb{R}^N$ be a nonempty bounded open domain with boundary $\Gamma = \partial G$ of class $C^4$, where $N \geq 2$. Let $G_0, \mathcal{O}\Subset G$ be nonempty open subsets strictly contained in $G$, and let $\mathcal{O}^1_\Gamma$ and $\mathcal{O}^2_\Gamma$ be any nonempty open subsets of $\Gamma$. Denote by $\mathbbm{1}_S$ the characteristic function of a set $S \subset \overline{G}$, and by $\mathbb{E}[\cdot]$ the expectation operator with respect to the underlying probability space. The following time-space domains are considered:
\[
Q := (0,T) \times G, \qquad 
\Sigma := (0,T) \times \Gamma, \qquad 
Q_0 := (0,T) \times G_0.
\]

Let \( (\Omega, \mathcal{F}, \{\mathcal{F}_t\}_{t \geq 0}, \mathbb{P}) \) be a fixed complete filtered probability space on which a one-dimensional standard Brownian motion \( W(\cdot) \) is defined, such that \( \{\mathcal{F}_t\}_{t \geq 0} \) is the natural filtration generated by \( W(\cdot) \) and augmented by all the \( \mathbb{P} \)-null sets in \( \mathcal{F} \). Let \( \mathcal{X} \) be a Banach space, and let \( C([0,T]; \mathcal{X}) \) be the Banach space of all \( \mathcal{X} \)-valued continuous functions defined on \( [0,T] \). For some sub-sigma algebra \( \mathcal{G} \subset \mathcal{F} \), we denote by \( L^2_{\mathcal{G}}(\Omega; \mathcal{X}) \) the Banach space of all \( \mathcal{X} \)-valued \( \mathcal{G} \)-measurable random variables \( X \) such that \( \mathbb{E} \left[ \|X\|_\mathcal{X}^2 \right] < \infty \), with the canonical norm. We also denote by \( L^2_{\mathcal{F}}(0,T; \mathcal{X}) \) the Banach space consisting of all \( \mathcal{X} \)-valued \( \{\mathcal{F}_t\}_{t \geq 0} \)-adapted processes \( X(\cdot) \) such that \( \mathbb{E} \left[ \int_0^T\|X(t)\|^2_{\mathcal{X}}\,dt \right] < \infty \), with the canonical norm. Moreover, we denote by \( L^\infty_{\mathcal{F}}(0,T; \mathcal{X}) \) the Banach space consisting of all \( \mathcal{X} \)-valued \( \{\mathcal{F}_t\}_{t \geq 0} \)-adapted essentially bounded processes. Finally, we denote by \( L^2_{\mathcal{F}}(\Omega; C([0,T]; \mathcal{X})) \) the Banach space consisting of all \( \mathcal{X} \)-valued \( \{\mathcal{F}_t\}_{t \geq 0} \)-adapted continuous processes \( X(\cdot) \) such that \( \mathbb{E} \left[ \max_{0\leq t\leq T}\|X(t)\|^2_{\mathcal{X}} \right] < \infty \), with the canonical norm.
	
Consider the following product space:
\[
\mathbb{L}^2 := L^2(G, dx) \times L^2(\Gamma, d\sigma) = L^2(G) \times L^2(\Gamma),
\]
where \(dx\) denotes the Lebesgue measure on \(G\) and \(d\sigma\) denotes the surface measure on \(\Gamma\). One can easily verify that \(\mathbb{L}^2\) is a Hilbert space with the usual inner product
\[
\langle (y, y_\Gamma), (z, z_\Gamma) \rangle_{\mathbb{L}^2}
:= \langle y, z \rangle_{L^2(G)} + \langle y_\Gamma, z_\Gamma \rangle_{L^2(\Gamma)}.
\]

We assume that the functions \(a^{jk} : \overline{G} \to \mathbb{R}\) and \(a_\Gamma^{jk} : \Gamma \to \mathbb{R}\) (\(j,k = 1,2,\cdots,N\)), which represent the diffusion coefficients in the domain and on the boundary, respectively, satisfy the following assumptions:
\begin{enumerate}
    \item \( a^{jk} \in C^{2}(\overline{G}) \), and \( a^{jk} = a^{kj} \) for all \( j,k = 1, 2, \cdots, N \).
    \item \( a_\Gamma^{jk} \in C^{1}(\Gamma) \), and \( a_\Gamma^{jk} = a_\Gamma^{kj} \) for all \( j,k = 1, 2, \cdots, N \).
    \item There exists a constant \( \beta_0 > 0 \) such that
\begin{align}\label{asummAandAgama}
    \begin{aligned}
        \sum_{j,k=1}^N a^{jk}(x)\eta_j \eta_k &\geq \beta_0 |\eta|^2
\quad \textnormal{for all } (x,\eta)\equiv(x_1,\cdots,x_N,\eta_1,\cdots,\eta_N) \in \overline{G} \times \mathbb{R}^N, \\
\sum_{j,k=1}^N a_\Gamma^{jk}(x)\eta_j \eta_k  &\geq \beta_0 |\eta|^2
\quad \textnormal{for all } (x,\eta)\equiv(x_1,\cdots,x_N,\eta_1,\cdots,\eta_N) \in \Gamma \times \mathbb{R}^N,
    \end{aligned}
    \end{align}
  where $|\cdot|$ stands for the Euclidean norm in $\mathbb{R}^N$.
\end{enumerate}

Throughout this paper, $C$ denotes a generic positive constant depending only on $G$, $G_0$, $\beta_0$, $M$, and $M_\Gamma$, which may change from line to line, with
\[
M = \max_{1 \le j,k \le N}  \| a^{jk} \|_{C^2(\overline{G})}, \qquad 
M_\Gamma = \max_{1 \le j,k \le N}  \| a_\Gamma^{jk} \|_{C^1(\Gamma)}.
\]

The main goal of this paper is to establish the well-posedness and to derive a new Carleman estimate for the following nonhomogeneous linear backward stochastic parabolic equation subject to general dynamic boundary conditions:
\begin{equation}\label{1.1gen} 
\begin{cases} 
dz + \displaystyle\sum_{j,k=1}^N 
\frac{\partial}{\partial x_j}\Big(a^{jk}(x) \frac{\partial z}{\partial x_k}\Big) \, dt
= \Big(F_1 + \sum_{j=1}^N \frac{\partial F_j}{\partial x_j}\Big) \, dt + Z \, dW(t), & \textnormal{in } Q, \\[1mm]

dz_\Gamma + \displaystyle\sum_{j,k=1}^N 
D_j\Big(a_\Gamma^{jk}(x) D_k z_\Gamma\Big) \, dt
- \displaystyle\sum_{j,k=1}^N a^{jk}(x) \frac{\partial z}{\partial x_j} \nu^k \, dt  = \Big(F_2 - \displaystyle\sum_{j=1}^N F_j \nu^j + \sum_{j=1}^N D_j {F_\Gamma}_j\Big) \, dt 
+ \widehat{Z} \, dW(t), & \textnormal{on } \Sigma, \\[1mm]

z_\Gamma = z|_\Gamma, & \textnormal{on } \Sigma, \\[1mm]

(z, z_\Gamma)|_{t=T} = (z_T, z_{\Gamma,T}), & \textnormal{in } G \times \Gamma.
\end{cases} 
\end{equation}
Here, $(z, z_\Gamma; Z, \widehat{Z})$ denotes the state variable, where $z_\Gamma = z|_\Gamma$ represents the trace of $z$ on the boundary $\Gamma$, and $Z$ and $\widehat{Z}$ are the stochastic correction terms corresponding to $z$ and $z_\Gamma$, respectively. The terminal state $(z_T, z_{\Gamma,T})$ belongs to $L^2_{\mathcal{F}_T}(\Omega; \mathbb{L}^2)$. We emphasize that $z_{\Gamma,T}$ is not necessarily the trace of $z_T$, since no trace of $z_T$ is assumed to exist. However, if $z_T$ admits a well-defined trace on $\Gamma$, then this trace necessarily coincides with $z_{\Gamma,T}$. The terms $F_1$, $(F_j)_{1 \le j \le N}$, $F_2$, and $({F_\Gamma}_j)_{1 \le j \le N}$ represent the external source terms, which are assumed to satisfy
\begin{align*}
F_1 &\in L^2_{\mathcal{F}}(0,T; L^2(G)), 
\qquad
F_2 \in L^2_{\mathcal{F}}(0,T; L^2(\Gamma)), \\
F_j &\in L^2_{\mathcal{F}}(0,T; L^2(G)), 
\qquad
{F_\Gamma}_j \in L^2_{\mathcal{F}}(0,T; L^2(\Gamma)),\qquad j=1,2,\cdots,N.
\end{align*}
In \eqref{1.1gen}, $\frac{\partial}{\partial x_j}$ and $D_j$ denote, respectively, the standard partial derivative with respect to $x_j$ and the \(j\)-th tangential derivative operator on $\Gamma$, defined by
\[
D_j := \frac{\partial}{\partial x_j} - \nu^j \sum_{k=1}^N \nu^k \frac{\partial}{\partial x_k},\qquad j=1,2,\cdots,N,
\]
where $\nu^j$ is the $j$-th component of the unit outward normal vector \( \nu = (\nu^1, \dots, \nu^N) \) to the boundary \( \Gamma \). The tangential derivative represents the part of the standard partial derivative that is tangent to the surface, obtained by subtracting the normal component from the standard derivative.  For further details on tangential derivative operators, we refer the reader to \cite{BonnGuu98}. Throughout this work, the boundary $\Gamma$ of $G$ is regarded as a  $(N-1)$-dimensional compact Riemannian submanifold induced by the natural embedding $\Gamma \hookrightarrow \mathbb{R}^N$, and the associated surface measure $d\sigma$.

It is straightforward to see that, by introducing the notations
\[
\mathcal{A} = (a^{jk})_{1 \le j,k \le N}, \qquad 
\mathcal{A}_\Gamma = (a_\Gamma^{jk})_{1 \le j,k \le N}, \qquad 
F = (F_j)_{1 \le j \le N}, \qquad 
F_\Gamma = ({F_\Gamma}_j)_{1 \le j \le N},
\]
the system \eqref{1.1gen} can be equivalently rewritten in the following compact form:
\begin{equation}\label{1.1}
\begin{cases}
\begin{array}{ll}
dz + \nabla \cdot (\mathcal{A}(x) \nabla z) \, dt = (F_1 + \nabla \cdot F) \, dt + Z \, dW(t) & \textnormal{in } Q, \\[0.3em]
dz_\Gamma + \nabla_\Gamma \cdot (\mathcal{A}_\Gamma(x) \nabla_\Gamma z_\Gamma) \, dt - \partial_\nu^\mathcal{A} z \, dt = (F_2 - F \cdot \nu + \nabla_\Gamma \cdot F_\Gamma) \, dt + \widehat{Z} \, dW(t) & \textnormal{on } \Sigma, \\[0.3em]
z_\Gamma = z |_\Gamma & \textnormal{on } \Sigma, \\[0.3em]
(z, z_\Gamma) |_{t=T} = (z_T, z_{\Gamma,T}) & \textnormal{in } G \times \Gamma,
\end{array}
\end{cases}
\end{equation}
where the boundary term
\[
\partial_\nu^\mathcal{A} z := (\mathcal{A} \nabla z \cdot \nu)|_\Sigma = \sum_{j,k=1}^N a^{jk} \frac{\partial z}{\partial x_j} \nu^k
\]
represents the conormal derivative of \( z \) with respect to \( \mathcal{A} \). This conormal derivative serves as a coupling term between the bulk equation in \( Q \) and the surface equation on \( \Sigma \), so that the bulk equation is governed directly, while the surface equation is influenced through \( \partial_\nu^\mathcal{A} z \).  The operators $\nabla \cdot F$ (in $G$) and $\nabla_\Gamma \cdot F_\Gamma$ (on $\Gamma$) denote, respectively, the standard divergence in $G$ and the tangential divergence on $\Gamma$, defined by
\[
\nabla \cdot F := \sum_{j=1}^N \frac{\partial F_j}{\partial x_j}, \qquad
\nabla_\Gamma \cdot F_\Gamma := \sum_{j=1}^N D_j {F_\Gamma}_{j} 
= \sum_{j=1}^N \left( \frac{\partial {F_\Gamma}_{j}}{\partial x_j} - \nu^j \sum_{k=1}^N \nu^k \frac{\partial {F_\Gamma}_{j}}{\partial x_k} \right).
\]
The surface term $\nabla_\Gamma \cdot (\mathcal{A}_\Gamma \nabla_\Gamma z_\Gamma)$ models diffusion on $\Gamma$, where the tangential gradient is defined by
\[
\nabla_\Gamma z_\Gamma := (D_j z_\Gamma)_{1\leq j\leq N}=\nabla z - (\partial_\nu z) \nu,
\]
which is the projection of the gradient $\nabla z$ onto the tangent plane at $x\in\Gamma$.  For further details on these geometric tools and on dynamic boundary conditions, we refer the reader to \cite{Golde, khoutaibi2020null, maniar2017null, taylor2011}.

Similarly to the standard Sobolev spaces \( H^k(G) \) (\( k = 1,2 \)), we introduce the Sobolev spaces \( H^k(\Gamma) \) (\( k = 1,2 \)) on the boundary \( \Gamma \) by
\[
H^k(\Gamma)
= \bigl\{ z_\Gamma \in L^2(\Gamma) \,\big|\, \nabla_\Gamma z_\Gamma \in H^{k-1}(\Gamma; \mathbb{R}^N) \bigr\}, 
\qquad k = 1,2.
\]
These spaces are endowed with the following norms:
\[
\|z_\Gamma\|_{H^1(\Gamma)}
:= \langle z_\Gamma, z_\Gamma \rangle_{H^1(\Gamma)}^{1/2},
\quad \textnormal{where} \quad
\langle y_\Gamma, z_\Gamma \rangle_{H^1(\Gamma)}
= \int_{\Gamma} y_\Gamma z_\Gamma \, d\sigma
+ \int_{\Gamma} \nabla_\Gamma y_\Gamma \cdot \nabla_\Gamma z_\Gamma \, d\sigma,
\]
and
\[
\|z_\Gamma\|_{H^2(\Gamma)}
:= \langle z_\Gamma, z_\Gamma \rangle_{H^2(\Gamma)}^{1/2},
\quad \textnormal{where} \quad
\langle y_\Gamma, z_\Gamma \rangle_{H^2(\Gamma)}
= \int_{\Gamma} y_\Gamma z_\Gamma \, d\sigma
+ \int_{\Gamma} \Delta_\Gamma y_\Gamma \, \Delta_\Gamma z_\Gamma \, d\sigma.
\]

In the sequel, we introduce the following Hilbert spaces:
\[
\mathbb{H}^k
:= \bigl\{ (z, z_\Gamma) \in H^k(G) \times H^k(\Gamma) \,\big|\, z_\Gamma = z|_\Gamma \bigr\},
\qquad k = 1,2,
\]
which are endowed with the natural inner product
\[
\langle (y, y_\Gamma), (z, z_\Gamma) \rangle_{\mathbb{H}^k}
:= \langle y, z \rangle_{H^k(G)} + \langle y_\Gamma, z_\Gamma \rangle_{H^k(\Gamma)}.
\]
Moreover, throughout this paper we denote by $(H^k(G))'$, $H^{-k}(\Gamma)$,
and $\mathbb{H}^{-k}$ the dual spaces of $H^k(G)$, $H^k(\Gamma)$, and
$\mathbb{H}^k$, respectively, for $k = 1,2$, where the duality is taken with respect to the pivot spaces $L^2(G)$, $L^2(\Gamma)$, and $\mathbb{L}^2$.  For notational convenience, when we write
\[
(z, z_\Gamma) = (z_1, z_2) \quad \text{in } G \times \Gamma, \quad \text{a.s.},
\]
we mean that
\[
z = z_1 \quad \text{in } G \quad \text{and} \quad z_\Gamma = z_2 \quad \text{on } \Gamma,
\quad \text{a.s.}
\]
Throughout this paper,  the boundary vector fields
\[
\mathcal{A}_\Gamma \nabla_\Gamma z_\Gamma
:= \left( \sum_{k=1}^N a^{jk}_\Gamma(x)\, D_k z_\Gamma \right)_{1 \le j \le N}
\quad \text{and} \quad
F_\Gamma := \bigl( {F_\Gamma}_j \bigr)_{1 \le j \le N},
\]
appearing in \eqref{1.1gen}, are understood as \emph{tangential vector fields}
on the manifold $\Gamma$. More precisely, these vectors are regarded as functions
\( f_\Gamma : \Gamma \to \mathbb{R}^N \) such that, for each \( x \in \Gamma \),
\begin{align}\label{tangcond}
f_\Gamma(x) \in T_x \Gamma
:= \bigl\{ y \in \mathbb{R}^N \,\big|\, y \cdot \nu(x) = 0 \bigr\}.
\end{align}
This tangentiality assumption plays a crucial role in the derivation of the Carleman estimate \eqref{1.3} for \eqref{1.1gen}, as it allows for significant simplifications in the treatment of boundary integral terms. In particular, it permits the use of the standard tangential divergence formula
\[
\int_\Gamma (\nabla_\Gamma \cdot f_\Gamma)\, z_\Gamma \, d\sigma
= -\int_\Gamma f_\Gamma \cdot \nabla_\Gamma z_\Gamma \, d\sigma,
\]
which holds for \( f_\Gamma \in H^1(\Gamma; T_x\Gamma) \) and
\( z_\Gamma \in H^1(\Gamma) \). In the case where the tangentiality condition~\eqref{tangcond} is not satisfied,
we refer to Remark~\ref{remark1.2} for further details.

In Section~\ref{sec2SEC}, we prove that equation \eqref{1.1gen} is well posed. More precisely, for any given terminal data 
$(z_T, z_{\Gamma,T}) \in L^2_{\mathcal{F}_T}(\Omega; \mathbb{L}^2)$
and source terms
$F_1 \in L^2_\mathcal{F}(0,T; L^2(G))$, $F_2 \in L^2_\mathcal{F}(0,T; L^2(\Gamma))$, 
$F_j \in L^2_\mathcal{F}(0,T; L^2(G))$, and
${F_\Gamma}_j \in L^2_\mathcal{F}(0,T; L^2(\Gamma))$, $j=1,2,\cdots,N$, 
the system \eqref{1.1gen} admits a unique weak solution
\[
(z, z_\Gamma; Z, \widehat{Z})
\in
\Bigl(
L^2_\mathcal{F}\bigl(\Omega; C([0,T]; \mathbb{L}^2)\bigr)
\bigcap
L^2_\mathcal{F}(0,T; \mathbb{H}^1)
\Bigr)
\times
L^2_\mathcal{F}(0,T; \mathbb{L}^2).
\]
Moreover, there exists a constant $C>0$, depending only on $G$, $T$, $M$, and $M_\Gamma$, such that
\begin{align*}
& \max_{0 \le t \le T} \mathbb{E} \|(z(t), z_\Gamma(t))\|_{\mathbb{L}^2}^2
+ \|(z, z_\Gamma)\|_{L^2_\mathcal{F}(0,T; \mathbb{H}^1)}^2
+ \|(Z, \widehat{Z})\|_{L^2_\mathcal{F}(0,T; \mathbb{L}^2)}^2 \\
& \quad \le C \bigg(
\mathbb{E} \|(z_T, z_{\Gamma,T})\|_{\mathbb{L}^2}^2
+ \|F_1\|_{L^2_\mathcal{F}(0,T; L^2(G))}^2
+ \sum_{j=1}^N \|F_j\|_{L^2_\mathcal{F}(0,T; L^2(G))}^2 \\
& \qquad\qquad
+ \|F_2\|_{L^2_\mathcal{F}(0,T; L^2(\Gamma))}^2
+ \sum_{j=1}^N \|{F_\Gamma}_j\|_{L^2_\mathcal{F}(0,T; L^2(\Gamma))}^2
\bigg).
\end{align*}

In order to establish our main Carleman estimate for equation \eqref{1.1gen}, we first recall the following classical technical result due to Fursikov and Imanuvilov (see \cite{fursikov1996controllability}).
\begin{lm}\label{lm3.1}
    For any nonempty open subset \( G_1 \Subset G \), there exists a function \( \psi \in C^4(\overline{G}) \) such that
    \[
    \psi > 0 \quad \textnormal{in} \;\, G; \qquad \psi = 0 \quad \textnormal{on} \;\, \Gamma; \qquad |\nabla \psi| > 0 \quad \textnormal{in} \;\, \overline{G \setminus G_1}.
    \]
\end{lm}
Since \(|\nabla \psi|^2 = |\nabla_\Gamma \psi|^2 + |\partial_\nu \psi|^2\) on \( \Gamma\), the function \( \psi \) in Lemma~\ref{lm3.1} also satisfies
\begin{equation}\label{3.20101}
\nabla_\Gamma \psi = 0, \qquad |\nabla \psi| = |\partial_\nu \psi|, \qquad \partial_\nu \psi \leq -C < 0 \quad \textnormal{on} \;\, \Gamma, \quad \textnormal{for some constant } \, C > 0.
\end{equation}

For any parameters $\lambda, \mu \geq 1$, we introduce the following weight
functions:
\begin{align}\label{3.2}
\begin{aligned}
\gamma \equiv \gamma(t) &:= (t(T-t))^{-1}, \qquad
\alpha \equiv \alpha(t,x) := \bigl(e^{\mu \psi(x)} - e^{2\mu \|\psi\|_\infty}\bigr)\gamma(t), \\
\ell &:= \lambda \alpha, \qquad
\theta := e^{\ell}, \qquad
\varphi \equiv \varphi(t,x) := e^{\mu \psi(x)}\,\gamma(t),
\end{aligned}
\end{align}
for all $(t,x) \in (0,T) \times \overline{G}$. Note that \(\alpha\) and \(\varphi\) are constant on the boundary \(\Gamma\), which implies
\begin{equation*}
\nabla_\Gamma \alpha = \nabla_\Gamma \varphi = 0 \quad \textnormal{on} \;\, \Gamma.
\end{equation*}
Moreover, one can readily verify that there exists a constant $C = C(G) > 0$ such that
\begin{align}\label{3.3}
\begin{aligned}
& \varphi \geq 4 T^{-2}, \qquad |\varphi_t| \leq C T \varphi^2, \qquad |\varphi_{tt}| \leq C T^2 \varphi^3, \\
& |\alpha_t| \leq C T e^{2 \mu \|\psi\|_\infty} \varphi^2, \qquad |\alpha_{tt}| \leq C T^2 e^{2 \mu \|\psi\|_\infty} \varphi^3.
\end{aligned}
\end{align}

In what follows, we choose the weight functions \(\theta\), \(\ell\), and \(\varphi\) as in \eqref{3.2}, and the function \(\psi\) as in Lemma \ref{lm3.1}, with \(G_1\) being any fixed nonempty open subset such that \(G_1 \Subset \mathcal{B}\), for any nonempty open subset $\mathcal{B}\Subset G$. The first main result of this paper is the following fundamental Carleman estimate for equation \eqref{1.1gen}, stated as follows:
\begin{thm}\label{thm1.1}
There exist large parameters  \( \mu_0, \lambda_0\geq1 \), and a constant \( C>0 \), depending only on \( G \), \( \mathcal{B} \), $\beta_0$, $M$, and $M_\Gamma$, such that for any data \( F_1 \in L^2_\mathcal{F}(0,T;L^2(G)) \), \( F_2 \in L^2_\mathcal{F}(0,T;L^2(\Gamma)) \), \( F_j \in L^2_\mathcal{F}(0,T;L^2(G)) \), \( {F_\Gamma}_j \in L^2_\mathcal{F}(0,T;L^2(\Gamma)) \), $j=1,2,\cdots,N$ and any terminal state \( (z_T, z_{\Gamma,T}) \in L^2_{\mathcal{F}_T}(\Omega; \mathbb{L}^2) \), the weak solution \( (z, z_\Gamma; Z, \widehat{Z}) \) of equation \eqref{1.1gen} satisfies the estimate
\begin{align}\label{1.3}
\begin{aligned}
&\, \lambda^3 \mu^4 \mathbb{E} \iint_Q \theta^2 \varphi^3 |z|^2 \, dx  dt + \lambda^3 \mu^3 \mathbb{E} \iint_\Sigma \theta^2 \varphi^3 |z_\Gamma|^2 \, d\sigma  dt \\
&+ \lambda \mu \mathbb{E} \iint_Q \theta^2 \varphi \vert \nabla z \vert^2 \, dx  dt + \lambda \mu \mathbb{E} \iint_\Sigma \theta^2 \varphi \vert \nabla_\Gamma z_\Gamma \vert^2 \, d\sigma  dt \\
&\leq C \bigg[ \lambda^3 \mu^4 \mathbb{E} \int_0^T\int_\mathcal{B} \theta^2 \varphi^3 |z|^2 \, dx  dt + \mathbb{E} \iint_Q \theta^2 |F_1|^2 \, dx  dt +  \mathbb{E} \iint_\Sigma \theta^2 |F_2|^2 \, d\sigma  dt \\
&\hspace{1cm} + \lambda^2 \mu^2 \sum_{j=1}^N\mathbb{E} \iint_Q \theta^2 \varphi^2 \vert F_j \vert^2 \, dx  dt + \lambda^2 \mu^2 \sum_{j=1}^N\mathbb{E} \iint_\Sigma \theta^2 \varphi^2 \vert {F_\Gamma}_j \vert^2 \, d\sigma  dt \\
&\hspace{1cm} + \lambda^2 \mu^2 \mathbb{E} \iint_Q \theta^2 \varphi^2 |Z|^2 \, dx  dt + \lambda^2 \mu^2 \mathbb{E} \iint_\Sigma \theta^2 \varphi^2 |\widehat{Z}|^2 \, d\sigma  dt \bigg],
\end{aligned}
\end{align}
for all \( \mu \geq \mu_0 \) and \( \lambda \geq \lambda_0(e^{2 \mu \|\psi\|_\infty} T + T^2) \).
\end{thm}

To the best of our knowledge, this paper is the first to establish the Carleman estimate \eqref{1.3} for backward anisotropic  stochastic parabolic equations with $\mathbb{L}^2$ bulk--surface source terms and $\mathbb{H}^{-1}$ bulk--surface source terms appearing as weak divergence-type sources, coupled with the general dynamic boundary condition
\[
\begin{cases} 
dz_\Gamma + \displaystyle\sum_{j,k=1}^N 
D_j\Big(a_\Gamma^{jk}(x) D_k z_\Gamma\Big) \, dt
- \displaystyle\sum_{j,k=1}^N a^{jk}(x) \frac{\partial z}{\partial x_j} \nu^k \, dt  
= \Big(F_2 - \displaystyle\sum_{j=1}^N F_j \nu^j + \sum_{j=1}^N D_j {F_\Gamma}_j\Big) \, dt 
+ \widehat{Z} \, dW(t),
& \textnormal{on } \Sigma, \\[1mm]
z_\Gamma = z|_\Gamma, & \textnormal{on } \Sigma.
\end{cases}
\]
Such systems are substantially more challenging to analyze than classical stochastic parabolic equations with standard Dirichlet boundary conditions (see \cite{Preprintelgrou23, liu2014global, tang2009null} and \cite[Chapter~9]{lu2021mathematical}), as well as Neumann or Robin boundary conditions (see \cite{RobinBCelg, yan2018carleman}). The additional difficulties arise from the combined presence of bulk--surface interactions, weak divergence-type source terms, and dynamic boundary effects, which induce strong coupling mechanisms and reduced regularity.

In the existing literature, Carleman estimates for systems with dynamic boundary conditions are available only in limited and simplified settings (see \cite{elgroubacconvdbc, Preprintelgrou231d, elgrou22SPEwithDBC}) or for certain coupled forward--backward stochastic systems (see \cite{barboelouk, mulobjdynbc25}). However, a comprehensive Carleman estimate for backward anisotropic  stochastic parabolic equations with general $\mathbb{H}^{-1}$ bulk--surface couplings and dynamic boundary conditions, together with an explicit Carleman constant independent of the final time $T$, has not yet been established. Addressing this gap is the main objective of the present paper.

In addition, we apply the derived Carleman estimate \eqref{1.3} to study the
null and approximate controllability, as well as an insensitizing control
problem, for forward reaction--convection--diffusion stochastic systems with
general dynamic boundary conditions. The explicit dependence of our Carleman constant on $T$ is exploited, for the first time, to obtain a sharp estimate on the problem data for the cost of null controllability of the associated controlled adjoint forward stochastic problem. In particular, this allows us to estimate the minimal norm of the controls required to drive the system to zero at time~$T$. These results highlight both the novelty and the significance of the present work.

By establishing the Carleman estimate \eqref{1.3}, we address three fundamental challenges:
\begin{enumerate}
\item We consider a fully general class of stochastic parabolic operators involving anisotropic diffusion matrices $(a^{jk})_{1\le j,k\le N}$ in the bulk and $(a_\Gamma^{jk})_{1\le j,k\le N}$ on the boundary, thus extending beyond the isotropic diffusion framework commonly studied in the literature. This
generality introduces many additional terms that must be treated with care.
\item We derive an explicit dependence of the Carleman constant on the final time $T$, which is essential for quantitative estimates and plays a key role in controllability problems.
\item We allow source terms in the drift components to belong to negative Sobolev spaces, which significantly relaxes the regularity assumptions and enables the treatment of weak divergence-type sources both in the bulk and on the boundary.
\end{enumerate}

To accomplish these objectives, we combine two complementary analytical methods.

\begin{itemize}
\item First, we apply the weighted identity method to the following general stochastic parabolic-type operators:
\[
\textnormal{``} dz + \sum_{j,k=1}^N 
\frac{\partial}{\partial x_j}\Big(a^{jk}(x) \frac{\partial z}{\partial x_k}\Big)\, dt \textnormal{''}
\quad \textnormal{and} \quad
\textnormal{``} dz_\Gamma + \sum_{j,k=1}^N 
D_j\Big(a_\Gamma^{jk}(x) D_k z_\Gamma\Big)\, dt
- \sum_{j,k=1}^N a^{jk}(x) \frac{\partial z}{\partial x_j} \nu^k\, dt \textnormal{''}.
\]
As a result, our Carleman estimate extends Theorem~4.1 of \cite{elgrou22SPEwithDBC}, which corresponds to the particular case
\[
(a^{jk})_{N\times N} = (a_\Gamma^{jk})_{N\times N} = (\delta^{jk})_{N\times N} = I_N,
\qquad
F_j = {F_\Gamma}_j = 0, \quad j = 1, 2,\cdots, N,
\]
where $I_N$ denotes the $N \times N$ identity matrix and $\delta^{jk}$ is the Kronecker delta. In addition, we obtain an explicit dependence of the Carleman constant on the final time $T$. This generalization allows for anisotropic diffusion effects through the matrices $(a^{jk})_{N\times N}$ and $(a_\Gamma^{jk})_{N\times N}$, thus covering a substantially broader class of diffusion processes. A major difficulty in our analysis arises from the treatment of new boundary integral terms generated by the general dynamic boundary condition
\[
\begin{cases} 
dz_\Gamma + \displaystyle\sum_{j,k=1}^N 
D_j\Big(a_\Gamma^{jk}(x) D_k z_\Gamma\Big) \, dt
- \displaystyle\sum_{j,k=1}^N a^{jk}(x) \frac{\partial z}{\partial x_j} \nu^k \, dt  
= F_2 \, dt 
+ \widehat{Z} \, dW(t),
& \textnormal{on } \Sigma, \\[1mm]
z_\Gamma = z|_\Gamma, & \textnormal{on } \Sigma.
\end{cases}
\]
These terms require a careful and adapted analysis, especially in the proof of Lemma~\ref{lm2.2}.

\item Second, to handle the weak divergence-type source terms $F_j$ and ${F_\Gamma}_j$, $j=1,2,\cdots,N$, we employ a duality method combined with suitable optimization arguments (see \cite{liu2014global} and the references therein for the case of Dirichlet boundary conditions). This strategy allows us to absorb the divergence-type contributions into the Carleman inequality and thus derive the desired estimate \eqref{1.3}. In contrast to \cite{liu2014global}, where Carleman estimates are mainly used to relax regularity assumptions on diffusion source terms, our approach is specifically \emph{designed} to address weak divergence-type sources of the form
\[
\sum_{j=1}^N \frac{\partial F_j}{\partial x_j} \quad \textnormal{in } Q,
\qquad \textnormal{and} \qquad
\sum_{j=1}^N D_j {F_\Gamma}_j \quad \textnormal{on } \Sigma.
\]
The main difficulty lies in the fact that the functions $F_j$ and ${F_\Gamma}_j$ belong only to $L^2$-spaces. Consequently, the corresponding divergence terms are only defined in the distributional sense and belong to the negative Sobolev space $H^{-1}$, which prevents the direct application of standard Carleman techniques and necessitates the duality-based framework developed in this work.
\end{itemize}

In this paper, as an application of our Carleman estimate \eqref{1.3}, we
investigate two control problems: the \emph{null controllability} and an
\emph{insensitizing control problem} for the following forward anisotropic 
stochastic reaction--convection--diffusion system with general dynamic boundary conditions:
\begin{equation}\label{1.4}
\begin{cases}
\begin{array}{ll}
dy - \nabla\cdot(\mathcal{A}\nabla y)\,dt
= \bigl(\xi+a_1 y + B_1 \cdot \nabla y + \mathbbm{1}_{G_0} u\bigr)\,dt + v_1\,dW(t)
& \textnormal{in } Q, \\[0.3em]
dy_\Gamma - \nabla_\Gamma \cdot (\mathcal{A}_\Gamma \nabla_\Gamma y_\Gamma)\,dt
+ \partial_\nu^\mathcal{A} y\,dt
= \bigl(\xi_\Gamma+a_2 y_\Gamma + B_2 \cdot \nabla_\Gamma y_\Gamma\bigr)\,dt + v_2\,dW(t)
& \textnormal{on } \Sigma, \\[0.3em]
y_\Gamma = y|_\Gamma
& \textnormal{on } \Sigma, \\[0.3em]
(y, y_\Gamma)|_{t=0} = (y_0+\tau_1\widehat{y}_0, y_{\Gamma,0}+\tau_2\widehat{y}_{\Gamma,0})
& \textnormal{in } G \times \Gamma.
\end{array}
\end{cases}
\end{equation}
Here, $(y, y_\Gamma)$ denotes the state variable of the system, with $y$
representing the bulk state and $y_\Gamma$ its trace on the boundary. The
initial datum $(y_0, y_{\Gamma,0})$ is assumed to belong to
$L^2_{\mathcal{F}_0}(\Omega; \mathbb{L}^2)$. We further
assume that $(\widehat{y}_0, \widehat{y}_{\Gamma,0}) \in
L^2_{\mathcal{F}_0}(\Omega; \mathbb{L}^2)$ is an unknown perturbation of $(y_0, y_{\Gamma,0})$ satisfying
\[
\mathbb{E}\,\|(\widehat{y}_0, \widehat{y}_{\Gamma,0})\|_{\mathbb{L}^2}^2 = 1,
\]
and that $\tau_1, \tau_2 \in \mathbb{R}$ are unknown small parameters. The functions
$\xi \in L^2_\mathcal{F}(0,T; L^2(G))$ and $\xi_\Gamma \in
L^2_\mathcal{F}(0,T; L^2(\Gamma))$ denote the given source terms in the bulk
and on the boundary, respectively. Accordingly, the term $(\tau_1\widehat{y}_0, \tau_2\widehat{y}_{\Gamma,0})$ may be regarded as a small perturbation of the nominal initial data $(y_0, y_{\Gamma,0})$. The coefficients $a_1$ and $B_1$ describe the reaction and convection effects in
the bulk, while $a_2$ and $B_2$ account for the corresponding effects on the
boundary. They are assumed to be bounded and $\{\mathcal{F}_t\}_{t \ge 0}$-adapted stochastic processes. More precisely, we assume that
\begin{align*}
a_1 &\in L_\mathcal{F}^\infty(0,T; L^\infty(G)), 
\qquad\quad B_1 \in L_\mathcal{F}^\infty(0,T; L^\infty(G; \mathbb{R}^N)), \\
a_2 &\in L_\mathcal{F}^\infty(0,T; L^\infty(\Gamma)), 
\qquad\quad B_2 \in L_\mathcal{F}^\infty(0,T; L^\infty(\Gamma; \mathbb{R}^N)).
\end{align*}

The control mechanism in \eqref{1.4} consists of a distributed control $u$ acting
on the interior subdomain $G_0 \Subset G$, together with stochastic controls
$v_1$ and $v_2$ acting in the bulk and on the boundary, respectively. Accordingly, the admissible control triple satisfies
\[
(u, v_1, v_2) \in L^2_\mathcal{F}(0,T; L^2(G_0))
\times L^2_\mathcal{F}(0,T; L^2(G))
\times L^2_\mathcal{F}(0,T; L^2(\Gamma)).
\]

Similarly to the proof of \cite[Theorem 2.1]{elgroubacconvdbc}, one can show that system \eqref{1.4} is well posed.  More precisely,  for any initial datum $(y(0),y_\Gamma(0)) \in L^2_{\mathcal{F}_0}(\Omega; \mathbb{L}^2)$, source terms
$\xi \in L^2_\mathcal{F}(0,T; L^2(G))$ and $\xi_\Gamma \in
L^2_\mathcal{F}(0,T; L^2(\Gamma))$, and any control triple
\[
(u, v_1, v_2) \in L^2_\mathcal{F}(0,T; L^2(G_0)) \times
L^2_\mathcal{F}(0,T; L^2(G)) \times L^2_\mathcal{F}(0,T; L^2(\Gamma)),
\]
system \eqref{1.4} admits a unique weak solution
\[
(y, y_\Gamma) \in L^2_\mathcal{F}(\Omega; C([0,T]; \mathbb{L}^2)) \bigcap
L^2_\mathcal{F}(0,T; \mathbb{H}^1).
\]
Moreover, there exists a constant
$C = C(G,T,\beta_0, M,M_\Gamma) > 0$ such that
\begin{align*}
& \max_{0 \le t \le T} \mathbb{E}\, \| (y(t), y_\Gamma(t)) \|^2_{\mathbb{L}^2}
+ \| (y, y_\Gamma) \|^2_{L^2_\mathcal{F}(0,T; \mathbb{H}^1)} \\
& \qquad \le C \Big( \mathbb{E}\, \| (y(0),y_\Gamma(0)) \|^2_{\mathbb{L}^2} 
+ \|\xi\|^2_{L^2_\mathcal{F}(0,T; L^2(G))} 
+ \|\xi_\Gamma\|^2_{L^2_\mathcal{F}(0,T; L^2(\Gamma))} \\
& \hspace{2cm} + \| u \|^2_{L^2_\mathcal{F}(0,T; L^2(G_0))} 
+ \| v_1 \|^2_{L^2_\mathcal{F}(0,T; L^2(G))} 
+ \| v_2 \|^2_{L^2_\mathcal{F}(0,T; L^2(\Gamma))} \Big).
\end{align*}

The system \eqref{1.4} with $\xi \equiv \xi_\Gamma \equiv u \equiv v_1 \equiv v_2 \equiv 0$
models diffusion phenomena, such as the evolution of heat, bacterial populations,
or chemical concentrations, under stochastic perturbations and in interaction
with the boundary through the dynamic boundary conditions
\eqref{1.4}$_2$--\eqref{1.4}$_3$, which involve the time derivative of the boundary state.  
The convection terms $B_1 \cdot \nabla y$ and $B_2 \cdot \nabla_\Gamma y_\Gamma$ 
describe the transport of these quantities in a moving medium, such as a fluid flow with a prescribed velocity field. For further details on stochastic diffusion
models, we refer to \cite{DapratoZabc,lu2021mathematical} and the references therein.

We define the following energy functional, referred to as the
\emph{sentinel}, associated with equation \eqref{1.4} and localized on the bulk--surface observation sets
$\mathcal{O} \subset G$ and $\mathcal{O}_\Gamma^1, \mathcal{O}_\Gamma^2 \subset \Gamma$:
\begin{align}\label{functioPhii1.2}
\begin{aligned}
    \mathcal{E}(y,y_\Gamma)
    :=&\; \frac{1}{2}\mathbb{E}\int_0^T\int_{\mathcal{O}} |y|^2 \,dx \,dt
       + \frac{1}{2}\mathbb{E}\int_0^T\int_{\mathcal{O}^1_\Gamma}|y_\Gamma|^2 \,d\sigma \,dt  \\
    &\;  + \frac{1}{2}\mathbb{E}\int_0^T\int_{\mathcal{O}^2_\Gamma}|\nabla_\Gamma y_\Gamma|^2 \,d\sigma \,dt,
    \end{aligned}
\end{align}
where $(y,y_\Gamma)$ denotes the solution of \eqref{1.4} corresponding to the parameters $\tau_1,\tau_2$ and to the controls $u$, $v_1$, and $v_2$.

Introduced by J.-L.~Lions (see \cite{lions1989quel}), the notion of
\emph{insensitizing control} consists in finding controls $(u, v_1, v_2)$ such that the energy functional $\mathcal{E}$ is locally invariant with respect to small perturbations of the initial data. More precisely, the insensitizing control problem associated with system \eqref{1.4} is defined as follows.
\begin{df}\label{deff1}
Let $\xi \in L^2_\mathcal{F}(0,T;L^2(G))$, $\xi_\Gamma \in L^2_\mathcal{F}(0,T;L^2(\Gamma))$, and $(y_0,y_{\Gamma,0}) \in L^2_{\mathcal{F}_0}(\Omega;\mathbb{L}^2)$.  
A control triple $(u,v_1,v_2)$ is said to \emph{insensitize} the functional $\mathcal{E}$ if
\begin{align}\label{inspb}
    \frac{\partial \mathcal{E}(y,y_\Gamma)}{\partial \tau_1}\bigg|_{\tau_1=\tau_2=0}= \frac{\partial \mathcal{E}(y,y_\Gamma)}{\partial \tau_2}\bigg|_{\tau_1=\tau_2=0} = 0,
    \qquad \forall\,(\widehat{y}_0,\widehat{y}_{\Gamma,0}) \in L^2_{\mathcal{F}_0}(\Omega;\mathbb{L}^2)
    \ \;\textnormal{with}\; \
    \mathbb{E}\|(\widehat{y}_0,\widehat{y}_{\Gamma,0})\|_{\mathbb{L}^2}^2 = 1.
\end{align}
\end{df}

For the sake of clarity, the following notations will be used in certain parts of this paper:

\begin{itemize}
    \item \textbf{Controlled equation:}  
For the null controllability problem, we denote by \textnormal{\textbf{(CE)}} the system
\eqref{1.4} with $\tau_1\equiv\tau_2 \equiv 0$ and 
$\xi \equiv \xi_\Gamma \equiv 0$, namely,
\begin{equation}\tag{\textnormal{\textbf{CE}}}\label{CE}
\begin{cases}
\begin{array}{ll}
dy - \nabla\cdot(\mathcal{A}\nabla y)\,dt
= \bigl(a_1 y + B_1 \cdot \nabla y + \mathbbm{1}_{G_0} u\bigr)\,dt + v_1\,dW(t)
& \textnormal{in } Q, \\[0.3em]
dy_\Gamma - \nabla_\Gamma \cdot (\mathcal{A}_\Gamma \nabla_\Gamma y_\Gamma)\,dt
+ \partial_\nu^\mathcal{A} y\,dt
= \bigl(a_2 y_\Gamma + B_2 \cdot \nabla_\Gamma y_\Gamma\bigr)\,dt + v_2\,dW(t)
& \textnormal{on } \Sigma, \\[0.3em]
y_\Gamma = y|_\Gamma
& \textnormal{on } \Sigma, \\[0.3em]
(y, y_\Gamma)|_{t=0} = (y_0, y_{\Gamma,0})
& \textnormal{in } G \times \Gamma.
\end{array}
\end{cases}
\end{equation}

\item \textbf{State shorthand:} For convenience, we set
$Y_0 := (y_0, y_{\Gamma,0})$ to denote the pair consisting of the bulk and
boundary components of the state at the initial time $t = 0$.

\item \textbf{Control space:} 
    The admissible control space associated with system \eqref{1.4} is defined as
    \[
    \mathcal{U}_T := L^2_{\mathcal{F}}(0,T; L^2(G_0)) 
    \times L^2_{\mathcal{F}}(0,T; L^2(G)) 
    \times L^2_{\mathcal{F}}(0,T; L^2(\Gamma)),
    \]
    so that any control triple $(u,v_1,v_2) \in \mathcal{U}_T$ corresponds respectively to the distributed control in the interior, the stochastic bulk control, and the stochastic boundary control.
  \item \textbf{Control energy:} 
    For a control triple $(u,v_1,v_2) \in \mathcal{U}_T$, we define the associated energy by
    \[
    E(u,v_1,v_2) := \frac{1}{2}\mathbb{E}\int_0^T\Big\{ \|u\|_{L^2(G_0)}^2+\|v_1\|_{L^2(G)}^2+\|v_2\|_{L^2(\Gamma)}^2\Big\} \, dt.
    \]
\end{itemize}

The second main result of this paper concerns the null controllability of the system \eqref{CE}.
\begin{thm}\label{thmm1.2null}
The system \eqref{CE} is null controllable at time \(T\) for any control region \(G_0 \Subset G\). More precisely, for any initial state 
\(Y_0 \in L^2_{\mathcal{F}_0}(\Omega; \mathbb{L}^2)\), there exists a control triple 
\((\widehat{u}, \widehat{v}_1, \widehat{v}_2) \in \mathcal{U}_T\) such that the corresponding solution \((\widehat{y}, \widehat{y}_\Gamma)\) of \eqref{CE} satisfies
\begin{align}\label{nullCont}
(\widehat{y}(T, \cdot), \widehat{y}_\Gamma(T, \cdot)) = (0, 0) \quad \textnormal{in } G \times \Gamma, \quad \textnormal{a.s.}
\end{align}
Moreover, the controls can be chosen to satisfy
\begin{align}\label{1.2201}
\begin{aligned}
E(\widehat{u},\widehat{v}_1,\widehat{v}_2)
\leq \exp(C K_T) \, \mathbb{E}\| Y_0 \|^2_{\mathbb{L}^2},
\end{aligned}
\end{align}
where the constant \(K_T\) is given by
\begin{align}\label{constK}
K_T := 1 + \frac{1}{T} + \|a_1\|_\infty^{2/3} + \|a_2\|_\infty^{2/3} + T (\|a_1\|_\infty + \|a_2\|_\infty) + (1 + T) (\|B_1\|_\infty^2 + \|B_2\|_\infty^2).
\end{align}
\end{thm}

Once Theorem~\ref{thmm1.2null} is established, the set of admissible internal controls
\[
\mathcal{A}_{ad}(Y_0,G_0) := \bigl\{(u,v_1,v_2) \in \mathcal{U}_T \;\big|\; \textnormal{the solution } (y,y_\Gamma) \textnormal{ of } \eqref{CE} \textnormal{ satisfies } \eqref{nullCont} \bigr\}
\]
is nonempty. We then define the cost of null controllability, that is, the minimal $L^2$-norm of the controls required to steer the system to zero at time \(T\), by
\[
\mathcal{K}(Y_0,G_0) := \inf_{(u,v_1,v_2) \in \mathcal{A}_{ad}(Y_0,G_0)} 
E(u,v_1,v_2).
\]

As a direct consequence of Theorem \ref{thmm1.2null}, one can deduce the following explicit bound on the null controllability cost of the system \eqref{CE} .
\begin{cor}
Let \(Y_0 \in L^2_{\mathcal{F}_0}(\Omega; \mathbb{L}^2)\), \(T>0\), and let the coefficients \(a_i \in L_\mathcal{F}^\infty(0,T; L^\infty(G))\) and \(B_i \in L_\mathcal{F}^\infty(0,T; L^\infty(G;\mathbb{R}^N))\), \(i=1,2\). Then, the cost of null controllability of \eqref{CE} satisfies
\begin{align}\label{nulconcost}
\mathcal{K}(Y_0,G_0) \leq \exp(C K_T)\,\mathbb{E}\|Y_0\|^2_{\mathbb{L}^2},
\end{align}
where the constants $C$ and $K_T$ are the same as defined in \eqref{1.2201}-\eqref{constK}.
\end{cor}

The third main result of this paper concerns the existence of insensitizing controls for \eqref{1.4}.
\begin{thm}\label{thmm1.3ins}
Assume that $G_0 \cap \mathcal{O} \neq \emptyset$, and let $\mathcal{O}_\Gamma^1, \mathcal{O}_\Gamma^2 \subset \Gamma$ be arbitrary nonempty open subsets, and set $(y_0, y_{\Gamma,0}) = (0,0)$. Then there exist constants $K > 0$ and $C > 0$, depending only on $G$, $G_0$, $\mathcal{O}$,  $\mathcal{O}^1_\Gamma$, $\mathcal{O}^2_\Gamma$, $T$, $\beta_0$, $M$, $M_\Gamma$, and the coefficients $a_i$, $B_i$ ($i=1,2$), such that for any $\xi \in L^2_\mathcal{F}(0,T;L^2(G))$, and $\xi_\Gamma \in L^2_\mathcal{F}(0,T;L^2(\Gamma))$,
satisfying
\begin{align}
\mathbb{E} \iint_Q \exp(K t^{-1}) |\xi|^2 \, dx\, dt< \infty,\quad 
\mathbb{E} \iint_\Sigma \exp(K t^{-1}) |\xi_\Gamma|^2 \, d\sigma\, dt < \infty,
\end{align}
there exists a control triple $(u, v_1, v_2) \in \mathcal{U}_T$ that insensitizes the functional $\mathcal{E}$ in the sense of Definition~\ref{deff1}. Moreover, the controls satisfy the estimate
\begin{align}\label{estiof controls}
E(u,v_1,v_2) \leq C \bigg[ 
\mathbb{E} \iint_Q \exp(K t^{-1}) |\xi|^2 \, dx\, dt + 
\mathbb{E} \iint_\Sigma \exp(K t^{-1}) |\xi_\Gamma|^2 \, d\sigma\, dt 
\bigg].
\end{align}
\end{thm}

Moreover, as a direct consequence of our Carleman estimate \eqref{1.3}, it is not difficult to deduce the following unique continuation property: The solution $(z,z_\Gamma;Z,\widehat{Z})$ of the adjoint equation of \eqref{CE} (more precisely, the equation \eqref{4.1}) satisfies that
\begin{align*}
(z,Z,\widehat{Z})=(0,0,0)\,\,\,\textnormal{in}\,\,\,Q_0\times Q\times\Sigma,\;\,\textnormal{a.s.}\Longrightarrow\;\,(z_T,z_{\Gamma,T})=(0,0),\,\,\,\textnormal{in}\,\,\, G\times\Gamma,\;\,\textnormal{a.s.}
\end{align*}
Hence, using the above uniqueness together with a duality argument, we  deduce  the following approximate controllability of  \eqref{CE}.
\begin{cor}\label{thmm1.3Appr}
Equation \eqref{CE} is approximately controllable at time \(T\) for any control region \(G_0 \Subset G\). More precisely, for any initial state \(Y_0 = (y_0, y_{\Gamma,0}) \in L^2_{\mathcal{F}_0}(\Omega; \mathbb{L}^2)\), any target state \(Y_T = (y_T, y_{\Gamma,T}) \in L^2_{\mathcal{F}_T}(\Omega; \mathbb{L}^2)\), and any \(\varepsilon > 0\), there exists a control triple \(
(\widehat{u}, \widehat{v}_1, \widehat{v}_2) \in \mathcal{U}_T\)
such that the corresponding solution \((\widehat{y}, \widehat{y}_\Gamma)\) to \eqref{CE} satisfies
\begin{align}\label{apprCont}
\mathbb{E}\| \widehat{y}(T) - y_T \|^2_{L^2(G)} 
+ \mathbb{E}\| \widehat{y}_\Gamma(T) - y_{\Gamma,T} \|^2_{L^2(\Gamma)} \leq \varepsilon.
\end{align}
\end{cor}

This paper investigates the null controllability and an insensitizing control problem for a class of forward anisotropic stochastic parabolic equations with general dynamic boundary conditions. For controllability results concerning \emph{deterministic} parabolic equations with such boundary conditions, we refer the reader to \cite{haschmano}, \cite{khoutaibi2020null}, and
\cite{maniar2017null}. The case of fully discrete parabolic equations is treated
in \cite{LecaMoraPerZam}. For insensitizing control problems of such equations, we refer to \cite{Bouchomanouk25etman,zhaYinGao}, and to \cite{morsancarr2026}, where insensitizing control is studied using a sentinel that also involves the tangential gradient. In addition, the literature on controllability for deterministic partial differential equations is vast, and it is not possible to provide an exhaustive list here; for a comprehensive account of this topic, we refer the reader to the monographs \cite{coron07,Zabczyk} and the references therein.

The controllability of \emph{stochastic} parabolic equations with various static boundary conditions has been extensively studied. In particular, controllability results have been established for Dirichlet, Neumann, and Robin boundary conditions. For the case of Dirichlet boundary conditions, \cite{barbu2003carleman} was among the first works to obtain partial controllability results for forward stochastic parabolic equations with a localized control acting in the drift term, under rather restrictive assumptions on the system coefficients. Subsequently, \cite{tang2009null} substantially extended these results by means of global Carleman estimates, addressing the controllability of more general forward stochastic parabolic equations. In that work, an additional control acting on the diffusion term was introduced, and the analysis was carried out under the assumption that the zero-order coefficients are bounded, adapted, and may depend on space, time, and random variables. We also refer to \cite{FuuLiu}, where a general weighted identity for certain stochastic operators with complex principal parts is derived including stochastic complex Ginzburg--Landau operators, stochastic heat operators, and stochastic Schr\"odinger operators, and subsequently applied to study the controllability of several classes of stochastic partial differential equations. The work \cite{Preprintelgrou23} further investigates the controllability of more general stochastic parabolic equations involving both zero- and first-order terms, under the assumption that all coefficients are bounded and adapted. In \cite{Preprintelgrou23,tang2009null}, controllability results are obtained by introducing a second control acting on the diffusion term over the entire spatial domain. In contrast, in \cite{lu2011some,observineqback}, the Lebeau--Robbiano strategy is employed to study both approximate and null controllability of the forward stochastic heat equation with a single control acting in the drift term, under the assumption of space-independent coefficients. In addition, the controllability of certain classes of stochastic parabolic equations with Robin boundary conditions has been addressed in \cite{RobinBCelg, yan2018carleman}. For null controllability results concerning stochastic semidiscrete parabolic equations, we refer the reader to \cite{Zhaou2025}.

Regarding controllability results for stochastic parabolic equations with dynamic boundary conditions, we refer to \cite{elgrou22SPEwithDBC} for forward and backward stochastic parabolic equations involving only zero-order terms, to \cite{Preprintelgrou231d} for a class of one-dimensional stochastic heat equations, to \cite{elgroubacconvdbc} for backward stochastic parabolic equations incorporating both reaction and convection terms, and to \cite{withouextra} for the forward stochastic heat equation with a time--space localized control acting in the drift term of the bulk equation, without additional forcing terms. For controllability results concerning certain coupled forward--backward stochastic parabolic equations, we refer the reader to \cite{liu2014global,mulobjdynbc25}.

Following the classical duality approach, the controllability of linear stochastic partial differential equations, as studied in the aforementioned works, can be established by analyzing an observability problem for the corresponding adjoint equation. This approach relies on unique continuation properties to achieve approximate controllability and on observability inequalities to study null controllability and insensitizing control problems. A fundamental tool for proving such observability results—both for parabolic PDEs and SPDEs—is the use of Carleman estimates, which provide weighted energy estimates with sufficiently large parameters to absorb lower-order terms. Introduced by Carleman in 1939 to establish uniqueness results for second-order elliptic equations (see \cite{Carl39}), these estimates have played a crucial role in the study of control and inverse problems in the deterministic setting (see, e.g., \cite{Yamam2009invePrb}). More recently, in the stochastic setting, Carleman estimates have also become a central tool for studying control theory (see, e.g., \cite{lu2021mathematical}) as well as for analyzing inverse problems (see, for instance, \cite{luWang2024InvProSPDES,luZhan2024InvProSPDES}).

For insensitizing control problems for certain classes of stochastic parabolic equations, we refer the reader to the following works: the forward stochastic heat equation is studied in \cite{yansun2011}, backward stochastic heat equations in \cite{liu2014global}, and forward stochastic degenerate parabolic equations in \cite{liu2019carleman}. Insensitizing control problems for a class of stochastic forward parabolic equations with dynamic boundary conditions are considered in \cite{barboelouk}. The stochastic Kuramoto--Sivashinsky equation is treated in \cite{luliu2025}. Against this background, the present paper investigates, for the first time, insensitizing control problems for forward anisotropic  stochastic reaction--convection--diffusion equations with general dynamic boundary conditions.

To the best of our knowledge, this work is the first to address the null controllability and an insensitizing control problem for forward anisotropic  stochastic parabolic equations with general dynamic boundary conditions, including both reaction and convection terms. Furthermore, we provide the explicit estimate \eqref{nulconcost} for the null controllability cost of \eqref{CE}, expressed in terms of the final control time \(T\) and the problem data. The cost of approximate controllability of \eqref{CE} can also be considered; however, it is left open in the present paper. For related results on the deterministic heat equation, we refer the reader to \cite{Bouchomanouk21,fernandezzuaua2000secco}.

In equation \eqref{1.4}, the control functions \( u \), \( v_1 \), and \( v_2 \) are introduced to influence the evolution of the system and, in particular, its behavior at the final time \( T \). The conormal derivative \( \partial_\nu^{\mathcal{A}} y \) acts as the coupling term between the bulk equation in \( Q \) and the surface equation on \( \Sigma \). Consequently, the distributed control \( u \) directly affects the drift term of the bulk equation on the interior subdomain \( G_0 \), while the controls \( v_1 \) and \( v_2 \) act on the diffusion terms of the bulk and surface equations, respectively, over the entire spatial domain. Through the coupling term \( \partial_\nu^{\mathcal{A}} y \), the controls \( u \), \( v_1 \), and \( v_2 \) also indirectly influence the drift term of the surface equation. This interaction between bulk and boundary dynamics is a key mechanism that allows us to study the null controllability of \eqref{CE} and the insensitizing control problem \eqref{inspb}. By contrast, investigate these control problems using only the control \( u \), which acts exclusively on the drift term of the bulk equation, remains a challenging open problem; for partial results in this direction under the additional assumption of space-independent coefficients, we refer the reader to \cite{withouextra,lu2011some}.

In equation \eqref{1.4}, the control functions \(u\), \(v_1\), and \(v_2\) are introduced to influence the evolution of the system, and in particular its behavior at the final time \(T\). The conormal derivative \(\partial_\nu^{\mathcal{A}} y\) acts as the coupling term between the bulk equation in \(Q\) and the surface equation on \(\Sigma\). Consequently, the distributed control \(u\) directly affects the drift term of the bulk equation on the interior subdomain \(G_0\), while the controls \(v_1\) and \(v_2\) act on the diffusion terms of the bulk and surface equations, respectively, over the entire spatial domain. Through the coupling term \(\partial_\nu^{\mathcal{A}} y\), the controls \(u\), \(v_1\), and \(v_2\) also indirectly influence the drift term of the surface equation. This interaction between bulk and boundary dynamics is a key mechanism that allows us to study the null controllability of \eqref{CE} and the insensitizing control problem \eqref{inspb}. By contrast, investigating these control problems using only the control \(u\), which acts exclusively on the drift term of the bulk equation, remains a challenging open problem.

The Carleman estimate established in \cite[Theorem~4.1]{elgrou22SPEwithDBC} for backward stochastic parabolic equations, obtained via a weighted identity method for \eqref{1.1gen} with drift terms in $\mathbb{L}^2$ (that is, $F_j = {F_\Gamma}_j = 0$ for all $j = 1,2,\dots,N$), is not sufficient to address the control problems considered in this paper. In particular, it cannot be applied to the null controllability problem studied in Section~\ref{sec3}, nor to the insensitizing control problem analyzed in Section~\ref{sec6}. To overcome this limitation, an observability inequality for the associated adjoint problem with drift terms in negative-order Sobolev spaces is required; such an estimate is not provided by \cite[Theorem~4.1]{elgrou22SPEwithDBC}. For this reason, we derive and employ the Carleman estimate \eqref{1.3} in the present work, which enables us to solve the null controllability problem in Section~\ref{sec3}. Furthermore, in Section~\ref{sec6}, we reduce the insensitizing control problem \eqref{inspb} to a null controllability problem for an optimality cascade forward-backward system. This system involves zeroth-, first-, and second-order coupling terms in the bulk--surface equations. The first- and second-order (weak) coupling terms are of the form
\[
\nabla\cdot(zB_1),\qquad
\nabla_\Gamma \cdot \bigl(z_\Gamma B_2 - \mathbbm{1}_{\mathcal{O}_\Gamma^2}\nabla_\Gamma y_\Gamma \bigr),
\]
and their treatment requires the use of the Carleman estimate \eqref{1.3} to derive a new Carleman estimate for the considered adjoint coupled system (see Proposition~\ref{carlfor insencontro}).
\\

Now, some remarks are in order.

\begin{rmk}\label{remark1.2}
In the derivation of the Carleman estimate \eqref{1.3} for \eqref{1.1gen}, the
boundary vector fields
\[
\mathcal{A}_\Gamma \nabla_\Gamma z_\Gamma
= \Bigl(\sum_{k=1}^N a^{jk}_\Gamma D_k z_\Gamma \Bigr)_{1 \le j \le N}
\quad \textnormal{and} \quad
F_\Gamma = \bigl({F_\Gamma}_j\bigr)_{1 \le j \le N}
\]
are assumed to be tangential to the manifold \( \Gamma \). When this tangentiality condition is not satisfied, one must instead employ the following \emph{generalized tangential divergence formula} (see \cite[Proposition~5.4.9]{herpi18}):
\begin{align}\label{gendiveform}
\int_\Gamma (\nabla_\Gamma \cdot f_\Gamma)\, z_\Gamma \, d\sigma
= -\int_\Gamma f_\Gamma \cdot \nabla_\Gamma z_\Gamma \, d\sigma
+ \int_\Gamma (\nabla_\Gamma \cdot \nu)\, (f_\Gamma \cdot \nu)\, z_\Gamma \, d\sigma.
\end{align}
The additional term in the right-hand side of \eqref{gendiveform}, involving the normal component
\( f_\Gamma \cdot \nu \), generates extra boundary contributions in the Carleman estimate. These terms cannot be controlled by the techniques developed in the
present paper and, in particular, we do not know how to absorb them into the left-hand side of the Carleman inequality. As a consequence, the analysis of non-tangential boundary vector fields within the framework of Carleman estimates for equation \eqref{1.1gen} remains unsolved in this paper.
\end{rmk}

\begin{rmk}
In equation \eqref{CE}, the controls \(v_1\) and \(v_2\) act on the diffusion terms in the whole domain \(G\) and on its boundary \(\Gamma\), respectively. This allows the controllability results for \eqref{CE} to be extended to systems where the diffusion terms depend on the state and its gradients. Specifically, the terms \(v_1\,dW(t)\) and \(v_2\,dW(t)\) can be replaced by 
\[
(b_1 y + \widetilde{B}_1\cdot\nabla y + v_1)\,dW(t) \quad \textnormal{and} \quad 
(b_2 y_\Gamma + \widetilde{B}_2\cdot\nabla_\Gamma y_\Gamma + v_2)\,dW(t),
\]
where \(b_i\) and \(\widetilde{B}_i\), \(i=1,2\), are bounded, adapted stochastic processes. Indeed, if system \eqref{CE} is null controllable, then there exists a control triple $(u, v_1, v_2)$ such that
\[
(y(T, \cdot), y_\Gamma(T, \cdot)) = (0, 0) \quad \textnormal{in } G \times \Gamma, \quad \textnormal{a.s.}
\]
Defining the new control triple
\[
u^\star = u, \qquad
v_1^\star = v_1 - (b_1 y + \widetilde{B}_1 \cdot \nabla y), \qquad
v_2^\star = v_2 - (b_2 y_\Gamma + \widetilde{B}_2 \cdot \nabla_\Gamma y_\Gamma),
\]
one obtains null controllability for the extended system
\begin{equation*}
\begin{cases}
dy - \nabla\cdot(\mathcal{A}\nabla y) \,dt = (a_1 y + B_1 \cdot \nabla y + \mathbbm{1}_{G_0} u^\star) \,dt + (b_1 y + \widetilde{B}_1 \cdot \nabla y + v_1^\star) \,dW(t), & \textnormal{in } Q,\\[1mm]
dy_\Gamma - \nabla_\Gamma \cdot (\mathcal{A}_\Gamma \nabla_\Gamma y_\Gamma) \,dt + \partial_\nu^\mathcal{A} y \,dt = (a_2 y_\Gamma + B_2 \cdot \nabla_\Gamma y_\Gamma) \,dt + (b_2 y_\Gamma + \widetilde{B}_2 \cdot \nabla_\Gamma y_\Gamma + v_2^\star) \,dW(t), & \textnormal{on } \Sigma,\\
y_\Gamma = y|_\Gamma, & \textnormal{on } \Sigma,\\
(y, y_\Gamma)|_{t=0} = (y_0, y_{\Gamma,0}), & \textnormal{in } G \times \Gamma.
\end{cases}
\end{equation*}
\end{rmk}

\begin{rmk}
The surface diffusion term ``\( \displaystyle\sum_{j,k=1}^N 
D_j\Big(a_\Gamma^{jk}(x) D_k z_\Gamma\Big) \, dt \)'' in the boundary equation of \eqref{1.1gen} plays a crucial role in the derivation of the Carleman estimate \eqref{1.3}, in particular in Step~2 of the proof of Lemma~\ref{lm2.2}. This term is essential for absorbing the following problematic boundary integral:
\[
\lambda \mu \, \mathbb{E} \iint_\Sigma \theta^2 \varphi \, \bigl| \partial_\nu^\mathcal{A} \psi \bigr| \, \bigl| \nabla_\Gamma z_\Gamma \bigr|^2 \, d\sigma \, dt.
\]
Consequently, the null and insensitizing control problems for \eqref{1.4} under the dynamic boundary condition
\begin{equation*}
\begin{cases}
dy_\Gamma + \partial_\nu^\mathcal{A} y \, dt = (a_2 y_\Gamma + B_2 \cdot \nabla_\Gamma y_\Gamma) \, dt + v_2 \, dW(t), & \textnormal{on } \Sigma,\\
y_\Gamma = y|_\Gamma, & \textnormal{on } \Sigma,
\end{cases}
\end{equation*}
remains an open problem. In contrast, for \(N = 1\), it is known that the controllability of stochastic heat equations can be addressed under this type of boundary condition; see \cite{Preprintelgrou231d}. This is due to the fact that, when \(N=1\), the boundary \(\Gamma\) is a zero-dimensional manifold, and hence the tangential gradient operator is trivial, i.e., \(\nabla_\Gamma = 0\).
\end{rmk}

\begin{rmk}
Let $\mathcal{O} \subset G$ and $\mathcal{O}_\Gamma^1, \mathcal{O}_\Gamma^2 \subset \Gamma$ be nonempty open subsets. The insensitizing control problem \eqref{inspb} with the energy functional
\begin{align*}
\begin{aligned}
    \widetilde{\mathcal{E}}(y,y_\Gamma)
    :=&\; \frac{1}{2}\,\mathbb{E}\int_0^T\int_{\mathcal{O}} |\nabla y|^2 \,dx\,dt 
       + \frac{1}{2}\,\mathbb{E}\int_0^T\int_{\mathcal{O}^1_\Gamma} |y_\Gamma|^2 \,d\sigma\,dt \\
    &\; + \frac{1}{2}\,\mathbb{E}\int_0^T\int_{\mathcal{O}^2_\Gamma} |\nabla_\Gamma y_\Gamma|^2 \,d\sigma\,dt
\end{aligned}
\end{align*}
poses significant challenges. In particular, the Carleman estimate \eqref{1.3} is not sufficient to handle this functional. The main difficulty lies in the need to control a localized term
$\mathbb{E}\int_0^T\int_{\widetilde{\mathcal{G}}} \kappa\, |z|^2 \,dx\,dt$ using a higher-order term
$\mathbb{E}\int_0^T\int_\mathcal{G} \kappa\, |\Delta z|^2 \,dx\,dt,$
where $\widetilde{\mathcal{G}}\subset \mathcal{G}$ and $\kappa=\kappa(t,x)$ is a suitable weight function. A detailed discussion of the deterministic heat equation can be found in \cite{Guerrero2007}.
\end{rmk}

This paper is organized as follows. Section~\ref{sec2SEC} presents some preliminaries and establishes the well-posedness of equation \eqref{1.1gen}. In Section~\ref{sec3SEC}, we derive an intermediate Carleman estimate for equation \eqref{1.1gen} in the absence of weak divergence terms. Section~\ref{sec2} is devoted to the derivation of the main Carleman estimate stated in Theorem~\ref{thm1.1} for equation \eqref{1.1gen}, where weak divergence terms are taken into account. Section~\ref{sec3} addresses the derivation of the required observability inequality for the adjoint backward equation associated with \eqref{CE}, and the proof of the null controllability result stated in Theorem~\ref{thmm1.2null}. Finally, Section~\ref{sec6} is devoted to the study of the insensitizing control problem \eqref{inspb} and to the proof of Theorem~\ref{thmm1.3ins}.

\section{Preliminaries and Well-Posedness Results}\label{sec2SEC}
In this section, we present some preliminaries that will be used throughout the paper and establish the well-posedness results for equations \eqref{1.1gen} (equivalently \eqref{1.1}).

We first recall the following lemma, which allows us to use integration by parts in space, both in the bulk domain \(G\) and on the boundary \(\Gamma\). For the proofs, we refer, for instance, to \cite{evans,KhmMRH22}.
\begin{lm}\label{lm1.2}
Let \(F \in L^2(G; \mathbb{R}^N)\) and \(F_\Gamma \in L^2(\Gamma; T_x\Gamma)\). The divergence of \(F\) (resp., the tangential divergence of \(F_\Gamma\)) can be extended as a continuous linear functional on \(H^1(G)\) (resp., on \(H^1(\Gamma)\)) defined by
\begin{equation*}
\nabla \cdot F : H^1(G) \longrightarrow \mathbb{R}, 
\quad v \longmapsto -\int_G F \cdot \nabla v \, dx
+ \langle F \cdot \nu, v|_\Gamma \rangle_{H^{-1/2}(\Gamma),\, H^{1/2}(\Gamma)},
\end{equation*}
and
\begin{equation*}
\nabla_\Gamma \cdot F_\Gamma : H^1(\Gamma) \longrightarrow \mathbb{R},
\quad v_\Gamma \longmapsto -\int_\Gamma F_\Gamma \cdot \nabla_\Gamma v_\Gamma \, d\sigma.
\end{equation*}
Moreover, the following estimates hold:
\begin{align*}
\left|
\langle \nabla \cdot F, v \rangle_{(H^1(G))',\, H^1(G)}
\right|
&\le C_1 \|F\|_{L^2(G; \mathbb{R}^N)} \|v\|_{H^1(G)},
\\
\left|
\langle \nabla_\Gamma \cdot F_\Gamma, v_\Gamma \rangle_{H^{-1}(\Gamma),\, H^1(\Gamma)}
\right|
&\le C_2 \|F_\Gamma\|_{L^2(\Gamma; \mathbb{R}^N)} \|v_\Gamma\|_{H^1(\Gamma)},
\end{align*}
for all \(v \in H^1(G)\) and \(v_\Gamma \in H^1(\Gamma)\), and for some positive constants \(C_1\) and \(C_2\).
\end{lm}

We now present the following Itô formula for stochastic processes in weak form (see, e.g., \cite[Chapter~2]{lu2021mathematical}). This result will be useful for deriving energy estimates, as well as for establishing duality relations between forward and backward equations with general dynamic boundary conditions and drift terms in the $\mathbb{H}^{-1}$-Sobolev space.
\begin{lm}\label{lm1.1}
Let $(z,z_\Gamma), (y,y_\Gamma) \in L^2_\mathcal{F}(0,T; \mathbb{H}^1)$,
$(z_T,z_{\Gamma,T}) \in L^2_{\mathcal{F}_T}(\Omega; \mathbb{L}^2)$,
$(y_0,y_{\Gamma,0}) \in L^2_{\mathcal{F}_0}(\Omega; \mathbb{L}^2)$,
$(\phi,\phi_\Gamma), (\widetilde{\phi},\widetilde{\phi}_\Gamma)
\in L^2_\mathcal{F}(0,T; \mathbb{H}^{-1})$, and
$(\Psi,\widehat{\Psi}), (Z,\widehat{Z}) \in L^2_\mathcal{F}(0,T; \mathbb{L}^2)$.
Assume that, for all $t \in [0,T]$, the processes $(z,z_\Gamma; Z,\widehat{Z})$
and $(y,y_\Gamma)$ satisfy the following equations in $G \times \Gamma$,
respectively:
\begin{align*}
(z(t),z_\Gamma(t)) &= (z_T,z_{\Gamma,T})
- \int_t^T (\phi(s),\phi_\Gamma(s)) \, ds
- \int_t^T (Z(s),\widehat{Z}(s)) \, dW(s),
\quad \textnormal{a.s.}, \\
(y(t),y_\Gamma(t)) &= (y_0,y_{\Gamma,0})
+ \int_0^t (\widetilde{\phi}(s),\widetilde{\phi}_\Gamma(s)) \, ds
+ \int_0^t (\Psi(s),\widehat{\Psi}(s)) \, dW(s),
\quad \textnormal{a.s.}
\end{align*}
Then the following assertions hold:
\begin{enumerate}[1.]
\item For any $t \in [0,T]$, it holds almost surely that
\begin{align*}
\| (z(t),z_\Gamma(t)) \|^2_{\mathbb{L}^2}
&= \| (z(0),z_\Gamma(0)) \|^2_{\mathbb{L}^2}
+ 2 \int_0^t
\langle (\phi(s),\phi_\Gamma(s)), (z(s),z_\Gamma(s))
\rangle_{\mathbb{H}^{-1},\mathbb{H}^{1}} \, ds \\
&\quad + 2 \int_0^t
\langle (z(s),z_\Gamma(s)), (Z(s),\widehat{Z}(s))
\rangle_{\mathbb{L}^2} \, dW(s)
+ \int_0^t
\| (Z(s),\widehat{Z}(s)) \|^2_{\mathbb{L}^2} \, ds.
\end{align*}

\item For all $t \in [0,T]$, it holds almost surely that
\begin{align*}
\langle (z(t),z_\Gamma(t)), (y(t),y_\Gamma(t))
\rangle_{\mathbb{L}^2}
&= \langle (z(0),z_\Gamma(0)), (y_0,y_{\Gamma,0})
\rangle_{\mathbb{L}^2} \\
&\quad + \int_0^t
\langle (\phi(s),\phi_\Gamma(s)), (y(s),y_\Gamma(s))
\rangle_{\mathbb{H}^{-1},\mathbb{H}^1} \, ds \\
&\quad + \int_0^t
\langle (\widetilde{\phi}(s),\widetilde{\phi}_\Gamma(s)),
(z(s),z_\Gamma(s))
\rangle_{\mathbb{H}^{-1},\mathbb{H}^1} \, ds \\
&\quad + \int_0^t
\langle (y(s),y_\Gamma(s)), (Z(s),\widehat{Z}(s))
\rangle_{\mathbb{L}^2} \, dW(s) \\
&\quad + \int_0^t
\langle (z(s),z_\Gamma(s)), (\Psi(s),\widehat{\Psi}(s))
\rangle_{\mathbb{L}^2} \, dW(s) \\
&\quad + \int_0^t
\langle Z(s), \Psi(s) \rangle_{\mathbb{L}^2} \, ds.
\end{align*}
\end{enumerate}
\end{lm}
In the above lemma, $\langle \cdot, \cdot \rangle_{\mathbb{L}^2}$ denotes the
inner product in the space $\mathbb{L}^2$, while
$\langle \cdot, \cdot \rangle_{\mathbb{H}^{-1},\mathbb{H}^1}$ denotes the duality
pairing between $\mathbb{H}^1$ and its dual $\mathbb{H}^{-1}$, taken with respect
to the pivot space $\mathbb{L}^2$.

We first consider the following unbounded linear operator on the product space $\mathbb{L}^2$:
\[
\textbf{A} =
\begin{pmatrix}
\nabla \cdot (\mathcal{A} \nabla) & 0 \\
- \partial_\nu^{\mathcal{A}} & \nabla_\Gamma \cdot (\mathcal{A}_\Gamma \nabla_\Gamma)
\end{pmatrix},
\qquad
\mathcal{D}(\textbf{A}) = \mathbb{H}^2 .
\]

Following ideas similar to those in the proof of Proposition~2.1 in \cite{maniar2017null}, one can show that $\textbf{A}$ satisfies the following generation property.

\begin{prop}\label{prop2.1}
The operator $\textbf{A}$ is densely defined, self-adjoint, and dissipative on $\mathbb{L}^2$, and it generates an analytic $C_0$-semigroup $(S(t))_{t \ge 0}$ on $\mathbb{L}^2$. Moreover, we have 
$\mathcal{D}((-\textbf{A})^{1/2}) = \mathbb{H}^1$.
\end{prop}

We now give the definition of weak solution of \eqref{1.1} (equivalently \eqref{1.1gen}).
\begin{df}
A process 
$$ (z,z_\Gamma;Z,\widehat{Z}) \in\left( L^2_\mathcal{F}(\Omega; C([0,T]; \mathbb{L}^2)) \bigcap L^2_\mathcal{F}(0,T; \mathbb{H}^1) \right) \times L^2_\mathcal{F}(0,T; \mathbb{L}^2) $$
is said to be a weak solution of \eqref{1.1} if, for any \( t \in [0,T] \) and all \( (\zeta,\zeta_\Gamma) \in \mathbb{H}^1 \), it holds that
\begin{align}\label{2.2backprow}
\begin{aligned}
&\int_G (z_T - z(t)) \zeta \, dx+\int_\Gamma (z_{\Gamma,T} - z_\Gamma(t)) \zeta_\Gamma \, d\sigma \\
&= \int_t^T \int_G \mathcal{A} \nabla z \cdot \nabla \zeta \, dx \, ds+\int_t^T \int_\Gamma \mathcal{A}_\Gamma \nabla_\Gamma z_\Gamma \cdot \nabla_\Gamma \zeta_\Gamma \, d\sigma \, ds \\
&\quad+ \int_t^T \int_G F_1 \zeta \, dx \, ds+ \int_t^T \int_\Gamma F_2 \zeta_\Gamma \, d\sigma \, ds \\
&\quad - \int_t^T \int_G F \cdot \nabla \zeta \, dx \, ds- \int_t^T \int_\Gamma F_\Gamma \cdot \nabla_\Gamma \zeta_\Gamma \, d\sigma \, ds\\
&\quad+ \int_t^T \int_G Z \zeta \, dx \, dW(s)+ \int_t^T \int_\Gamma \widehat{Z} \zeta_\Gamma \, d\sigma \, dW(s), \quad \textnormal{a.s.}
\end{aligned}
\end{align}
\end{df}

We first establish the following well-posedness result for equation \eqref{1.1} in the absence of divergence terms, that is, when $F = F_\Gamma = 0$.

\begin{prop}\label{prop1wsec}
Let \( (z_T,z_{\Gamma,T}) \in L^2_{\mathcal{F}_T}(\Omega; \mathbb{L}^2) \),
\( F_1 \in L^2_\mathcal{F}(0,T; L^2(G)) \), and
\( F_2 \in L^2_\mathcal{F}(0,T; L^2(\Gamma)) \).
Then equation \eqref{1.1} (with $F = F_\Gamma = 0$) admits a unique weak solution
\[
(z,z_\Gamma;Z,\widehat{Z})
\in \Big( L^2_\mathcal{F}(\Omega; C([0,T]; \mathbb{L}^2))
\bigcap L^2_\mathcal{F}(0,T; \mathbb{H}^1) \Big)
\times L^2_\mathcal{F}(0,T; \mathbb{L}^2).
\]
Moreover, the following estimate holds:
\begin{align}\label{202.3estimsecwdt}
\begin{aligned}
& \max_{0\leq t\leq T} \mathbb{E}\|(z(t),z_\Gamma(t))\|^2_{\mathbb{L}^2}
+ \|(z,z_\Gamma)\|^2_{L^2_\mathcal{F}(0,T; \mathbb{H}^1)}
+ \|(Z,\widehat{Z})\|^2_{L^2_\mathcal{F}(0,T; \mathbb{L}^2)} \\
& \leq C \Big(
\mathbb{E}\|(z_T,z_{\Gamma,T})\|^2_{\mathbb{L}^2}
+ \|F_1\|^2_{L^2_{\mathcal{F}}(0,T; L^2(G))}
+ \|F_2\|^2_{L^2_{\mathcal{F}}(0,T; L^2(\Gamma))}
\Big),
\end{aligned}
\end{align}
where \( C>0 \) is a constant depending only on \( G \), \( T \), $\beta_0$, $M$, and $M_\Gamma$.
\end{prop}

\begin{proof}
Observe that equation \eqref{1.1} can be rewritten in the following abstract backward stochastic evolution form:
\begin{equation*}
\begin{cases}
d\mathbf{Z} = (-\textbf{A}\mathbf{Z} + \mathbf{F})\, dt + \widehat{\mathbf{Z}}\, dW(t),
& t \in [0,T),\\
\mathbf{Z}(T) = (z_T, z_{\Gamma,T}),
\end{cases}
\end{equation*}
where \(\mathbf{Z} = (z,z_\Gamma)\),
\(\widehat{\mathbf{Z}} = (Z,\widehat{Z})\), and
\(\mathbf{F} = (F_1,F_2) \in L^2_\mathcal{F}(0,T; \mathbb{L}^2)\).
Applying Proposition~\ref{prop2.1} together with Theorem~4.11 in \cite{lu2021mathematical}, we obtain the existence, uniqueness, and the estimate \eqref{202.3estimsecwdt}.
\end{proof}

We next establish the well-posedness of the full system \eqref{1.1}, or equivalently, \eqref{1.1gen}.

\begin{prop}\label{prop1w}
Let \( (z_T,z_{\Gamma,T}) \in L^2_{\mathcal{F}_T}(\Omega; \mathbb{L}^2) \),
\( F_1 \in L^2_\mathcal{F}(0,T; L^2(G)) \),
\( F_2 \in L^2_\mathcal{F}(0,T; L^2(\Gamma)) \),
\( F \in L^2_\mathcal{F}(0,T; L^2(G; \mathbb{R}^N)) \),
and
\( F_\Gamma \in L^2_\mathcal{F}(0,T; L^2(\Gamma; \mathbb{R}^N)) \).
Then equation \eqref{1.1} admits a unique weak solution
\[
(z,z_\Gamma;Z,\widehat{Z})
\in \Big( L^2_\mathcal{F}(\Omega; C([0,T]; \mathbb{L}^2))
\bigcap L^2_\mathcal{F}(0,T; \mathbb{H}^1) \Big)
\times L^2_\mathcal{F}(0,T; \mathbb{L}^2).
\]
Moreover, the following estimate holds:
\begin{align}\label{202.3estimbackw}
\begin{aligned}
& \max_{0\leq t\leq T} \mathbb{E}\|(z(t),z_\Gamma(t))\|^2_{\mathbb{L}^2}
+ \|(z,z_\Gamma)\|^2_{L^2_\mathcal{F}(0,T; \mathbb{H}^1)}
+ \|(Z,\widehat{Z})\|^2_{L^2_\mathcal{F}(0,T; \mathbb{L}^2)} \\
& \leq C \Big(
\mathbb{E}\|(z_T,z_{\Gamma,T})\|^2_{\mathbb{L}^2}
+ \|F_1\|^2_{L^2_{\mathcal{F}}(0,T; L^2(G))}
+ \|F_2\|^2_{L^2_{\mathcal{F}}(0,T; L^2(\Gamma))} \\
&\qquad\qquad
+ \|F\|^2_{L^2_{\mathcal{F}}(0,T; L^2(G; \mathbb{R}^N))}
+ \|F_\Gamma\|^2_{L^2_{\mathcal{F}}(0,T; L^2(\Gamma; \mathbb{R}^N))}
\Big),
\end{aligned}
\end{align}
where \( C>0 \) depends only on \( G \), \( T \), $\beta_0$, $M$, and $M_\Gamma$.
\end{prop}

\begin{proof}
For the uniqueness, since the system is linear, we assume that
\[
(z,z_\Gamma;Z,\widehat{Z}) \in 
\Big(L^2_{\mathcal{F}}(\Omega; C([0,T]; \mathbb{L}^2))
\bigcap L^2_{\mathcal{F}}(0,T; \mathbb{H}^1)\Big)
\times L^2_{\mathcal{F}}(0,T; \mathbb{L}^2)
\]
is a weak solution to \eqref{1.1} with vanishing data, that is, with terminal condition $(z_T,z_{\Gamma,T})=(0,0)$ and source terms $F_1=F_2=F=F_\Gamma=0$. Fix \(t\in[0,T]\), applying Itô's formula as stated in Lemma~\ref{lm1.1} to
\(\|(z,z_\Gamma)\|_{\mathbb{L}^2}^2\), computing \(d_s\|(z,z_\Gamma)\|_{\mathbb{L}^2}^2\),
and integrating over the interval \(s\in(t,T]\), we obtain
\begin{align*}
&\mathbb{E}\int_G |z(t)|^2\,dx
+\mathbb{E}\int_\Gamma |z_\Gamma(t)|^2\,d\sigma \\
&= -2\mathbb{E}\int_t^T\int_G \mathcal{A}\nabla z\cdot\nabla z \,dx\,ds
   -2\mathbb{E}\int_t^T\int_\Gamma \mathcal{A}_\Gamma
   \nabla_\Gamma z_\Gamma\cdot\nabla_\Gamma z_\Gamma \,d\sigma\,ds \\
&\quad -\mathbb{E}\int_t^T\int_G |Z|^2 \,dx\,ds
   -\mathbb{E}\int_t^T\int_\Gamma |\widehat{Z}|^2 \,d\sigma\,ds .
\end{align*}
Recalling assumption~\eqref{asummAandAgama}, it follows that
\begin{align*}
&\mathbb{E}\int_G |z(t)|^2\,dx
+\mathbb{E}\int_\Gamma |z_\Gamma(t)|^2\,d\sigma
+\mathbb{E}\int_t^T\int_G |Z|^2 \,dx\,ds
+\mathbb{E}\int_t^T\int_\Gamma |\widehat{Z}|^2 \,d\sigma\,ds \\
&\leq -2\beta_0 \mathbb{E}\int_t^T\int_G |\nabla z|^2 \,dx\,ds
      -2\beta_0 \mathbb{E}\int_t^T\int_\Gamma
      |\nabla_\Gamma z_\Gamma|^2 \,d\sigma\,ds \\
&\leq 0.
\end{align*}
Consequently, we conclude that
\[
(z,z_\Gamma;Z,\widehat{Z}) \equiv (0,0;0,0),
\quad \textnormal{a.s.}
\]
This completes the proof of the uniqueness of solutions to system~\eqref{1.1}.

For the existence, the approach consists of decomposing the solution into two parts: 
the solution of an elliptic problem and the solution of a backward stochastic parabolic equation without the divergence terms, for which well-posedness is already known from Proposition~\ref{prop1wsec}. 
Subsequently, by employing a sequential argument, we establish the existence of a solution to \eqref{1.1gen}.  To this end, let 
\( (z_T,z_{\Gamma,T}) \in L^2_{\mathcal{F}_T}(\Omega; \mathbb{L}^2) \), 
\( F_1 \in L^2_\mathcal{F}(0,T; L^2(G)) \), 
\( F_2 \in L^2_\mathcal{F}(0,T; L^2(\Gamma)) \), 
and assume that \( (F,F_\Gamma) \in \mathscr{F} \), where
\begin{align*}
\mathscr{F} = \Big\{& (F,F_\Gamma)\in L^2_\mathcal{F}(\Omega; C^1([0,T]; L^2(G; \mathbb{R}^N)))\times L^2_\mathcal{F}(\Omega; C^1([0,T]; L^2(\Gamma; \mathbb{R}^N))): \\
&\quad(F(T,\cdot),F_\Gamma(T,\cdot)) = (0,0) \;\;\textnormal{in}\; G\times\Gamma, \;\textnormal{a.s.} \Big\},
\end{align*}
which is dense in 
\( L^2_\mathcal{F}(0,T; L^2(G; \mathbb{R}^N)) \times L^2_\mathcal{F}(0,T; L^2(\Gamma; \mathbb{R}^N)) \). Hence, the elliptic system
\begin{equation}\label{ellipteqq}
\begin{cases}
\nabla \cdot (\mathcal{A} \nabla h) = \nabla \cdot F & \text{in } Q,\\[1mm]
\nabla_\Gamma \cdot (\mathcal{A}_\Gamma \nabla_\Gamma h_\Gamma) - \partial_\nu^\mathcal{A} h = \nabla_\Gamma \cdot F_\Gamma - F\cdot\nu & \text{on } \Sigma,\\[1mm]
h_\Gamma = h|_\Gamma & \text{on } \Sigma,
\end{cases}
\end{equation}
admits a unique solution $(h,h_\Gamma) \in L^2_\mathcal{F}(\Omega; C^1([0,T]; \mathbb{H}^1))$ such that the terminal condition
\[
(h(T,\cdot),h_\Gamma(T,\cdot)) = (0,0) \quad \text{in } G\times \Gamma, \;\text{a.s.}
\]

Consider the following backward stochastic parabolic system with dynamic boundary conditions:
\begin{equation}\label{202.3withh}
\begin{cases}
\begin{array}{ll}
dy + \nabla\cdot(\mathcal{A}\nabla y) \, dt = (F_1 + \partial_t h) \, dt + Y \, dW(t) & \text{in }  Q,\\
dy_\Gamma + \nabla_\Gamma\cdot(\mathcal{A}_\Gamma\nabla_\Gamma y_\Gamma) \, dt - \partial_\nu^\mathcal{A}y \, dt = (F_2 + \partial_t h_\Gamma) \, dt + \widehat{Y} \, dW(t) & \text{on }  \Sigma, \\
y_\Gamma = y|_\Gamma & \text{on }  \Sigma, \\
(y(T),y_\Gamma(T)) = (y_T,y_{\Gamma,T}) & \text{in }  G\times\Gamma.
\end{array}
\end{cases}
\end{equation}
By Proposition \ref{prop1wsec}, the backward system \eqref{202.3withh} admits a unique weak solution
\[
(y,y_\Gamma;Y,\widehat{Y}) \in \mathcal{Z}_T \equiv \Big(L^2_\mathcal{F}(\Omega; C([0,T]; \mathbb{L}^2)) \bigcap L^2_\mathcal{F}(0,T; \mathbb{H}^1)\Big) \times L^2_\mathcal{F}(0,T; \mathbb{L}^2),
\]
which satisfies, for any \( t \in [0,T] \) and all \( (\zeta,\zeta_\Gamma) \in \mathbb{H}^1 \),
\begin{align}\label{ineq1qwel}
\begin{aligned}
&\int_G (y(T) - y(t)) \zeta \, dx + \int_\Gamma (y_{\Gamma,T} - y_\Gamma(t)) \zeta_\Gamma \, d\sigma\\
&=  \int_t^T \int_G \mathcal{A} \nabla y \cdot \nabla \zeta \, dx \, ds + \int_t^T \int_\Gamma \mathcal{A}_\Gamma \nabla_\Gamma y_\Gamma \cdot \nabla_\Gamma \zeta_\Gamma \, d\sigma \, ds \\
&\quad + \int_t^T \int_G (F_1 + \partial_s h) \zeta \, dx \, ds + \int_t^T \int_\Gamma (F_2 + \partial_s h_\Gamma) \zeta_\Gamma \, d\sigma \, ds \\
&\quad + \int_t^T \int_G Y \zeta \, dx \, dW(s) + \int_t^T \int_\Gamma \widehat{Y} \zeta_\Gamma \, d\sigma \, dW(s), \quad \text{a.s.}
\end{aligned}
\end{align}
In particular, integrating by parts in time the third and fourth terms on the right-hand side of \eqref{ineq1qwel}, we obtain the equivalent variational formulation
\begin{align}\label{equaverfbyh}
\begin{aligned}
&\int_G (y(T) - y(t)) \zeta \, dx + \int_\Gamma (y_{\Gamma,T} - y_\Gamma(t)) \zeta_\Gamma \, d\sigma\\
&=  \int_t^T \int_G \mathcal{A} \nabla y \cdot \nabla \zeta \, dx \, ds + \int_t^T \int_\Gamma \mathcal{A}_\Gamma \nabla_\Gamma y_\Gamma \cdot \nabla_\Gamma \zeta_\Gamma \, d\sigma \, ds \\
&\quad + \int_t^T \int_G F_1 \zeta \, dx \, ds - \int_G h(t) \zeta \, dx + \int_t^T \int_\Gamma F_2 \zeta_\Gamma \, d\sigma \, ds - \int_\Gamma h_\Gamma(t) \zeta_\Gamma \, d\sigma\\
&\quad + \int_t^T \int_G Y \zeta \, dx \, dW(s) + \int_t^T \int_\Gamma \widehat{Y} \zeta_\Gamma \, d\sigma \, dW(s), \quad \text{a.s.}
\end{aligned}
\end{align}
Set 
\[
(z,z_\Gamma;Z,\widehat{Z}) = (y,y_\Gamma;Y,\widehat{Y}) + (h,h_\Gamma;0,0) \quad \text{in } (Q\times\Sigma)^2, \;\text{a.s.}
\] 
It is straightforward to verify that $(z,z_\Gamma;Z,\widehat{Z}) \in \mathcal{Z}_T.$ Moreover, using \eqref{equaverfbyh}, we see that \( (z,z_\Gamma;Z,\widehat{Z}) \) satisfies
\begin{align}\label{equasaparticur1}
\begin{aligned}
&\int_G (z_T - z(t)) \zeta \, dx + \int_\Gamma (z_{\Gamma,T} - z_\Gamma(t)) \zeta_\Gamma \, d\sigma\\
&= \int_t^T \int_G \mathcal{A} \nabla z \cdot \nabla \zeta \, dx \, ds + \int_t^T \int_\Gamma \mathcal{A}_\Gamma \nabla_\Gamma z_\Gamma \cdot \nabla_\Gamma \zeta_\Gamma \, d\sigma \, ds \\
&\quad - \int_t^T \int_G \mathcal{A} \nabla h \cdot \nabla \zeta \, dx \, ds - \int_t^T \int_\Gamma \mathcal{A}_\Gamma \nabla_\Gamma h_\Gamma \cdot \nabla_\Gamma \zeta_\Gamma \, d\sigma \, ds\\
&\quad + \int_t^T \int_G F_1 \zeta \, dx \, ds + \int_t^T \int_\Gamma F_2 \zeta_\Gamma \, d\sigma \, ds\\
&\quad + \int_t^T \int_G Z \zeta \, dx \, dW(s) + \int_t^T \int_\Gamma \widehat{Z} \zeta_\Gamma \, d\sigma \, dW(s), \quad \text{a.s.}
\end{aligned}
\end{align}
Next, applying \eqref{ellipteqq} to the third and fourth terms on the right-hand side of \eqref{equasaparticur1}, we obtain
\begin{align}\label{equasaparticur}
\begin{aligned}
&\int_G (z_T - z(t)) \zeta \, dx + \int_\Gamma (z_{\Gamma,T} - z_\Gamma(t)) \zeta_\Gamma \, d\sigma\\
&= \int_t^T \int_G \mathcal{A} \nabla z \cdot \nabla \zeta \, dx \, ds + \int_t^T \int_\Gamma \mathcal{A}_\Gamma \nabla_\Gamma z_\Gamma \cdot \nabla_\Gamma \zeta_\Gamma \, d\sigma \, ds \\
&\quad + \int_t^T \int_G F_1 \zeta \, dx \, ds + \int_t^T \int_\Gamma F_2 \zeta_\Gamma \, d\sigma \, ds\\
&\quad - \int_t^T \int_G F \cdot \nabla \zeta \, dx \, ds - \int_t^T \int_\Gamma F_\Gamma \cdot \nabla_\Gamma \zeta_\Gamma \, d\sigma \, ds\\
&\quad + \int_t^T \int_G Z \zeta \, dx \, dW(s) + \int_t^T \int_\Gamma \widehat{Z} \zeta_\Gamma \, d\sigma \, dW(s), \quad \text{a.s.}
\end{aligned}
\end{align}
Therefore, \((z,z_\Gamma;Z,\widehat{Z})\) is a weak solution of \eqref{1.1} corresponding to \((F,F_\Gamma) \in \mathscr{F}\).

Fix \( t \in (0,T) \). Applying Itô's formula from Lemma \ref{lm1.1} to \( \|(z,z_\Gamma)\|^2_{\mathbb{L}^2} \) and integrating over \( s \in (t,T) \) followed by taking expectations, we obtain
\begin{align*}
\begin{aligned}
& \mathbb{E}\|(z(t),z_\Gamma(t))\|^2_{\mathbb{L}^2} - \mathbb{E}\|(z_T,z_{\Gamma,T})\|^2_{\mathbb{L}^2} \\
&= -2 \mathbb{E}\int_t^T \int_G \mathcal{A} \nabla z \cdot \nabla z \, dx \, ds
   -2 \mathbb{E}\int_t^T \int_\Gamma \mathcal{A}_\Gamma \nabla_\Gamma z_\Gamma \cdot \nabla_\Gamma z_\Gamma \, d\sigma \, ds  \\
&\quad -2 \mathbb{E}\int_t^T \int_G F_1 z \, dx \, ds
   -2 \mathbb{E}\int_t^T \int_\Gamma F_2 z_\Gamma \, d\sigma \, ds \\
&\quad +2 \mathbb{E}\int_t^T \int_G F \cdot \nabla z \, dx \, ds
   +2 \mathbb{E}\int_t^T \int_\Gamma F_\Gamma \cdot \nabla_\Gamma z_\Gamma \, d\sigma \, ds \\
&\quad - \mathbb{E}\int_t^T \int_G Z^2 \, dx \, ds
   - \mathbb{E}\int_t^T \int_\Gamma \widehat{Z}^2 \, d\sigma \, ds.
\end{aligned}
\end{align*}
Using the coercivity condition \eqref{asummAandAgama} and Young's inequality, for any \(\varepsilon>0\), we get
\begin{align}\label{ineqqforwe12}
\begin{aligned}
& \mathbb{E}\|(z(t),z_\Gamma(t))\|^2_{\mathbb{L}^2} - \mathbb{E}\|(z_T,z_{\Gamma,T})\|^2_{\mathbb{L}^2} \\
&\leq -2 \beta_0 \mathbb{E}\int_t^T \int_G |\nabla z|^2 \, dx \, ds
      -2 \beta_0 \mathbb{E}\int_t^T \int_\Gamma |\nabla_\Gamma z_\Gamma|^2 \, d\sigma \, ds \\
&\quad + \mathbb{E}\int_t^T \int_G z^2 \, dx \, ds + \mathbb{E}\int_t^T \int_G F_1^2 \, dx \, ds \\
&\quad + \mathbb{E}\int_t^T \int_\Gamma z_\Gamma^2 \, d\sigma \, ds + \mathbb{E}\int_t^T \int_\Gamma F_2^2 \, d\sigma \, ds \\
&\quad + \varepsilon \mathbb{E}\int_t^T \int_G |\nabla z|^2 \, dx \, ds + \frac{1}{\varepsilon} \mathbb{E}\int_t^T \int_G |F|^2 \, dx \, ds \\
&\quad + \varepsilon \mathbb{E}\int_t^T \int_\Gamma |\nabla_\Gamma z_\Gamma|^2 \, d\sigma \, ds + \frac{1}{\varepsilon} \mathbb{E}\int_t^T \int_\Gamma |F_\Gamma|^2 \, d\sigma \, ds \\
&\quad - \mathbb{E}\int_t^T \int_G Z^2 \, dx \, ds - \mathbb{E}\int_t^T \int_\Gamma \widehat{Z}^2 \, d\sigma \, ds.
\end{aligned}
\end{align}
Adding ``\( -2 \mathbb{E} \int_t^T \|(z,z_\Gamma)\|^2_{\mathbb{L}^2} \, ds \)'' to both sides of \eqref{ineqqforwe12} yields
\begin{align*}
\begin{aligned}
& \mathbb{E}\|(z(t),z_\Gamma(t))\|^2_{\mathbb{L}^2} - \mathbb{E}\|(z_T,z_{\Gamma,T})\|^2_{\mathbb{L}^2} - 2 \mathbb{E}\int_t^T \|(z,z_\Gamma)\|^2_{\mathbb{L}^2} \, ds \\
&\quad + \mathbb{E}\int_t^T \int_G Z^2 \, dx \, ds + \mathbb{E}\int_t^T \int_\Gamma \widehat{Z}^2 \, d\sigma \, ds \\
&\leq \max(-2\beta_0,-1) \, \mathbb{E}\int_t^T \|(z,z_\Gamma)\|^2_{\mathbb{H}^1} \, ds \\
&\quad + \mathbb{E}\int_t^T \int_G F_1^2 \, dx \, ds + \mathbb{E}\int_t^T \int_\Gamma F_2^2 \, d\sigma \, ds \\
&\quad + 2\varepsilon \, \mathbb{E}\int_t^T \|(z,z_\Gamma)\|^2_{\mathbb{H}^1} \, ds
       + \frac{1}{\varepsilon} \mathbb{E}\int_t^T \int_G |F|^2 \, dx \, ds
       + \frac{1}{\varepsilon} \mathbb{E}\int_t^T \int_\Gamma |F_\Gamma|^2 \, d\sigma \, ds.
\end{aligned}
\end{align*}
Choosing \( 0 < \varepsilon < \min(2\beta_0,1)/2 \), we deduce
\begin{align}\label{ineqq1w}
\begin{aligned}
& \mathbb{E}\|(z(t),z_\Gamma(t))\|^2_{\mathbb{L}^2} + \mathbb{E}\int_t^T \|(z,z_\Gamma)\|^2_{\mathbb{H}^1} \, ds
  + \mathbb{E}\int_t^T \|Z\|^2_{L^2(G)} \, ds + \mathbb{E}\int_t^T \|\widehat{Z}\|^2_{L^2(\Gamma)} \, ds \\
&\leq C \bigg[ \mathbb{E}\|(z_T,z_{\Gamma,T})\|^2_{\mathbb{L}^2} 
  + \mathbb{E}\int_t^T \int_G F_1^2 \, dx \, ds + \mathbb{E}\int_t^T \int_\Gamma F_2^2 \, d\sigma \, ds \\
&\qquad\quad + \mathbb{E}\int_t^T \int_G |F|^2 \, dx \, ds + \mathbb{E}\int_t^T \int_\Gamma |F_\Gamma|^2 \, d\sigma \, ds 
  + \mathbb{E}\int_t^T \|(z,z_\Gamma)\|^2_{\mathbb{L}^2} \, ds \bigg].
\end{aligned}
\end{align}
Applying Gronwall's inequality to the last term in \eqref{ineqq1w} gives
\begin{align*}
\begin{aligned}
\mathbb{E}\|(z(t),z_\Gamma(t))\|^2_{\mathbb{L}^2} 
&\leq C \bigg[ \mathbb{E}\|(z_T,z_{\Gamma,T})\|^2_{\mathbb{L}^2} 
 + \mathbb{E} \int_0^T \|F_1\|^2_{L^2(G)} \, ds + \mathbb{E} \int_0^T \|F_2\|^2_{L^2(\Gamma)} \, ds \\
&\qquad\quad + \mathbb{E} \int_0^T \|F\|^2_{L^2(G;\mathbb{R}^N)} \, ds 
       + \mathbb{E} \int_0^T \|F_\Gamma\|^2_{L^2(\Gamma;\mathbb{R}^N)} \, ds \bigg],
\end{aligned}
\end{align*}
which implies 
\begin{align}\label{ineqq2w}
\begin{aligned}
\mathbb{E}\int_t^T \|(z,z_\Gamma)\|^2_{\mathbb{L}^2} \, ds 
&\leq C \bigg[ \mathbb{E}\|(z_T,z_{\Gamma,T})\|^2_{\mathbb{L}^2} 
 + \mathbb{E} \int_0^T \|F_1\|^2_{L^2(G)} \, ds + \mathbb{E} \int_0^T \|F_2\|^2_{L^2(\Gamma)} \, ds \\
&\qquad\quad + \mathbb{E} \int_0^T \|F\|^2_{L^2(G;\mathbb{R}^N)} \, ds 
       + \mathbb{E} \int_0^T \|F_\Gamma\|^2_{L^2(\Gamma;\mathbb{R}^N)} \, ds \bigg].
\end{aligned}
\end{align}
Combining \eqref{ineqq1w} and \eqref{ineqq2w}, we finally obtain
\begin{align}\label{ineqq1wse}
\begin{aligned}
& \max_{0\le t\le T} \mathbb{E} \|(z(t),z_\Gamma(t))\|^2_{\mathbb{L}^2} 
 + \mathbb{E} \int_0^T \|(z,z_\Gamma)\|^2_{\mathbb{H}^1} \, ds
 + \mathbb{E} \int_0^T \|Z\|^2_{L^2(G)} \, ds + \mathbb{E} \int_0^T \|\widehat{Z}\|^2_{L^2(\Gamma)} \, ds \\
&\leq C \bigg[ \mathbb{E}\|(z_T,z_{\Gamma,T})\|^2_{\mathbb{L}^2} 
 + \mathbb{E} \int_0^T \|F_1\|^2_{L^2(G)} \, ds + \mathbb{E} \int_0^T \|F_2\|^2_{L^2(\Gamma)} \, ds \\
&\qquad\quad  + \mathbb{E} \int_0^T \|F\|^2_{L^2(G;\mathbb{R}^N)} \, ds 
       + \mathbb{E} \int_0^T \|F_\Gamma\|^2_{L^2(\Gamma;\mathbb{R}^N)} \, ds \bigg].
\end{aligned}
\end{align}

Let us now take \( (F,F_\Gamma) \in L^2_\mathcal{F}(0,T;L^2(G;\mathbb{R}^N)) \times L^2_\mathcal{F}(0,T;L^2(\Gamma;\mathbb{R}^N)) \). Then, there exists a sequence \( (F_n,F_{\Gamma,n}) \in \mathscr{F} \) such that
\begin{align}\label{converesultofdata}
    (F_n,F_{\Gamma,n}) \longrightarrow (F,F_\Gamma) \quad \text{in} \quad L^2_\mathcal{F}(0,T;L^2(G;\mathbb{R}^N)) \times L^2_\mathcal{F}(0,T;L^2(\Gamma;\mathbb{R}^N)) \quad \text{as } n \to \infty.
\end{align}
Since \( (F_n,F_{\Gamma,n}) \in \mathscr{F} \), for the given data \( ((z_T,z_{\Gamma,T}), F_1, F_2, F_n, F_{\Gamma,n}) \), the identity \eqref{equasaparticur} ensures that \eqref{1.1} admits a unique weak solution \( (z_n,z_{\Gamma,n};Z_n,\widehat{Z}_n) \) satisfying, for any \( t \in [0,T] \) and all \( (\zeta,\zeta_\Gamma) \in \mathbb{H}^1 \),
\begin{align}\label{equasafibyzsec}
\begin{aligned}
&\int_G (z_T - z_n(t)) \zeta \, dx + \int_\Gamma (z_{\Gamma,T}-z_{\Gamma,n}(t)) \zeta_\Gamma \, d\sigma \\
&= \int_t^T \int_G \mathcal{A} \nabla z_n \cdot \nabla \zeta \, dx \, ds + \int_t^T \int_\Gamma \mathcal{A}_\Gamma \nabla_\Gamma z_{\Gamma,n} \cdot \nabla_\Gamma \zeta_\Gamma \, d\sigma \, ds \\
&\quad + \int_t^T \int_G F_1 \zeta \, dx \, ds + \int_t^T \int_\Gamma F_2 \zeta_\Gamma \, d\sigma \, ds \\
&\quad - \int_t^T \int_G F_n \cdot \nabla \zeta \, dx \, ds - \int_t^T \int_\Gamma F_{\Gamma,n} \cdot \nabla_\Gamma \zeta_\Gamma \, d\sigma \, ds \\
&\quad + \int_t^T \int_G Z_n \zeta \, dx \, dW(s) + \int_t^T \int_\Gamma \widehat{Z}_n \zeta_\Gamma \, d\sigma \, dW(s), \quad \text{a.s.}
\end{aligned}
\end{align}
Applying estimate \eqref{ineqq1wse}, we deduce
\begin{align}\label{interesec2sec}
\begin{aligned}
&\max_{0 \le t \le T} \mathbb{E}  \|(z_n(t),z_{\Gamma,n}(t))\|^2_{\mathbb{L}^2} + \mathbb{E} \int_0^T \|(z_n,z_{\Gamma,n})\|^2_{\mathbb{H}^1} \, ds \\
&\quad + \mathbb{E} \int_0^T \|Z_n\|_{L^2(G)}^2 \, ds + \mathbb{E} \int_0^T \|\widehat{Z}_n\|_{L^2(\Gamma)}^2 \, ds \\
&\le C \bigg[ \mathbb{E} \|(z_T,z_{\Gamma,T})\|^2_{\mathbb{L}^2} + \mathbb{E} \int_0^T \|F_1\|^2_{L^2(G)} \, ds + \mathbb{E} \int_0^T \|F_2\|^2_{L^2(\Gamma)} \, ds \\
&\qquad + \mathbb{E} \int_0^T \|F_n\|^2_{L^2(G;\mathbb{R}^N)} \, ds + \mathbb{E} \int_0^T \|F_{\Gamma,n}\|^2_{L^2(\Gamma;\mathbb{R}^N)} \, ds \bigg].
\end{aligned}
\end{align}
For another sequence \( (F_m,F_{\Gamma,m}) \in \mathscr{F} \) with associated solution \( (z_m,z_{\Gamma,m};Z_m,\widehat{Z}_m) \), the linearity of \eqref{1.1} combined with \eqref{interesec2sec} gives
\begin{align}\label{secinteresec2sec}
\begin{aligned}
&\max_{0 \le t \le T} \mathbb{E} \|(z_n - z_m, z_{\Gamma,n} - z_{\Gamma,m})\|^2_{\mathbb{L}^2} + \mathbb{E} \int_0^T \|(z_n - z_m, z_{\Gamma,n} - z_{\Gamma,m})\|^2_{\mathbb{H}^1} \, ds \\
&\quad + \mathbb{E} \int_0^T \|Z_n - Z_m\|^2_{L^2(G)} \, ds + \mathbb{E} \int_0^T \|\widehat{Z}_n - \widehat{Z}_m\|^2_{L^2(\Gamma)} \, ds \\
&\le C \bigg[ \mathbb{E} \int_0^T \|F_n - F_m\|^2_{L^2(G;\mathbb{R}^N)} \, ds + \mathbb{E} \int_0^T \|F_{\Gamma,n} - F_{\Gamma,m}\|^2_{L^2(\Gamma;\mathbb{R}^N)} \, ds \bigg].
\end{aligned}
\end{align}
Therefore, by using \eqref{secinteresec2sec} and \eqref{converesultofdata}, we have that the sequence 
\(\{(z_n,z_{\Gamma,n};Z_n,\widehat{Z}_n)\}_{n \ge 1}\) is Cauchy in \(\mathcal{Z}_T\). Hence, there exists a unique limit $(z,z_{\Gamma};Z,\widehat{Z})\in \mathcal{Z}_T$ such that
\begin{align}\label{conveestim2.12}
(z_n,z_{\Gamma,n};Z_n,\widehat{Z}_n) \longrightarrow (z,z_{\Gamma};Z,\widehat{Z}) \quad \text{in} \quad \Big(L^2_\mathcal{F}(\Omega; C([0,T];\mathbb{L}^2)) \bigcap L^2_\mathcal{F}(0,T;\mathbb{H}^1)\Big) \times L^2_\mathcal{F}(0,T;\mathbb{L}^2), 
\end{align}
as \( n \to \infty \). Finally, by combining \eqref{conveestim2.12}, \eqref{equasafibyzsec}, and \eqref{converesultofdata}, we deduce that 
\((z,z_{\Gamma};Z,\widehat{Z})\) satisfies \eqref{2.2backprow}, and therefore it is a weak solution of \eqref{1.1}. 
Furthermore, the estimate \eqref{202.3estimbackw} is a direct consequence of \eqref{conveestim2.12}, 
\eqref{interesec2sec}, and \eqref{converesultofdata}. This completes the proof of Proposition \ref{prop1w}.

\end{proof}

\section{Intermediate Carleman Estimate without Divergence Terms}\label{sec3SEC}

This section is devoted to the derivation of a Carleman estimate for \eqref{1.1gen}
(respectively, \eqref{1.1}) in the absence of divergence source terms.
More precisely, we consider the following backward stochastic parabolic system
with dynamic boundary conditions:
\begin{equation}\label{1.1sec}
\begin{cases}
	dz + \nabla \cdot (\mathcal{A} \nabla z)\, dt
	= F_1\, dt + Z\, dW(t),
	& \text{in } Q, \\[0.3em]
	dz_\Gamma + \nabla_\Gamma \cdot (\mathcal{A}_\Gamma \nabla_\Gamma z_\Gamma)\, dt
	- \partial_\nu^{\mathcal{A}} z\, dt
	= F_2\, dt + \widehat{Z}\, dW(t),
	& \text{on } \Sigma, \\[0.3em]
	z_\Gamma = z|_\Gamma,
	& \text{on } \Sigma, \\[0.3em]
	(z, z_\Gamma)\big|_{t=T} = (z_T, z_{\Gamma,T}),
	& \text{in } G \times \Gamma.
\end{cases}
\end{equation}
To derive the desired Carleman estimate, we employ a weighted identity method;
see, for instance, \cite{Preprintelgrou23,tang2009null} for related developments.
Compared with the classical Laplacian case, several additional difficulties arise
in the present setting, mainly due to the dynamic boundary condition
\eqref{1.1sec}\(_2\)-\eqref{1.1sec}\(_3\). First, it is necessary to carefully track the dependence of the Carleman parameters
on the final time $T$. Second, the diffusion operators are described by general, possibly space-dependent, symmetric matrices
\[
\mathcal{A} = (a^{jk})_{1\le j,k\le N} \in C^2(\overline{G}; \mathbb{R}^{N\times N})
\quad \text{and} \quad
\mathcal{A}_\Gamma = (a_\Gamma^{jk})_{1\le j,k\le N} \in C^1(\Gamma; \mathbb{R}^{N\times N}),
\]
which satisfy the assumptions stated in \eqref{asummAandAgama}.
The presence of such matrix-valued diffusion operators significantly complicates
the computations.
In particular, they give rise to additional lower-order terms that require a
refined and delicate analysis.
Moreover, due to the dynamic nature of the boundary condition,
several boundary integral terms must be treated in a modified manner, with all constants made explicit in terms of the final time $T$.

Let $\ell \in C^{1,3}\bigl((0,T)\times\overline{G}\bigr)$ be a given weight function,
and let $\Psi \in C^{1,2}\bigl((0,T)\times\overline{G}\bigr)$ and
$\Phi \in C^{1,0}((0,T)\times\Gamma)$ be two auxiliary functions.
For simplicity of notation, given a function $y = y(x)$ with
$x = (x_1,\ldots,x_N) \in \mathbb{R}^N$, we denote by
\[
y_{x_j} := \frac{\partial y}{\partial x_j}
\quad \text{and} \quad
y_{x_j x_k} := \frac{\partial^2 y}{\partial x_j\, \partial x_k},
\qquad 1 \le j,k \le N,
\]
the first- and second-order partial derivatives of $y$ with respect to the
spatial variables.
Higher-order derivatives are understood in the same manner.
Put
\begin{equation}\label{5.1111012}
    \begin{cases}
        A= \displaystyle\sum_{j,k=1}^N (a^{jk}\ell_{x_j}\ell_{x_k}-a^{jk}_{x_k}\ell_{x_j}-a^{jk}\ell_{x_jx_k})-\Psi-\ell_t,\\
        B= 2\bigg[A\Psi+\displaystyle\sum_{j,k=1}^N (Aa^{jk}\ell_{x_j})_{x_k}\bigg]-A_t+\sum_{j,k=1}^N (a^{jk}\Psi_{x_k})_{x_j},\\
        c^{jk}=\displaystyle\sum_{j',k'=1}^N\Big[2a^{jk'}(a^{j'k}\ell_{x_{j'}})_{x_{k'}}-(a^{jk}a^{j'k'}\ell_{x_{j'}})_{x_{k'}}\Big]+\frac{a^{jk}_t}{2}-\Psi a^{jk},\quad j,k=1,2,\cdots,N,\\
        \widetilde{A}=\Phi-\ell_t,\\
        \widetilde{B}=-2\widetilde{A}\Phi-\widetilde{A}_t.
    \end{cases}
\end{equation}

According to \cite[Theorem~9.26]{lu2021mathematical}, the following fundamental
weighted identity holds for the stochastic parabolic-type operator of the form
\[
dz + \sum_{j,k=1}^N
\frac{\partial}{\partial x_j}
\left(a^{jk}(x)\frac{\partial z}{\partial x_k}\right)\, dt
= dz + \nabla \cdot (\mathcal{A}\nabla z)\, dt,
\]
where the differential $dz$ is understood in the It\^o sense and the divergence
term represents the spatial diffusion operator associated with the matrix
$\mathcal{A}$.
\begin{thm}\label{thm5.1}
Let $z$ be an $H^2(G)$-valued Itô process. Set $\theta=e^\ell$ and $h=\theta z$. Then, for any $t\in[0,T]$ and a.e. $(x,\omega)\in \overline{G}\times\Omega$, the following identity holds:
\begin{align}
\begin{aligned}
&\,2\theta\left[\nabla\cdot(\mathcal{A}\nabla h)+Ah\right]\left[dz+\nabla\cdot(\mathcal{A}\nabla z)dt\right]-2\nabla\cdot(\mathcal{A}\nabla h \,dh)\\
&\quad+2\sum_{j,k=1}^N\bigg[\sum_{j',k'=1}^N \big(2a^{jk}a^{j'k'}\ell_{x_{j'}}h_{x_{j}}h_{x_{k'}}-a^{jk}a^{j'k'}\ell_{x_j}h_{x_{j'}}h_{x_{k'}}\big)\\
&\qquad\qquad\qquad\quad\;\;\;-\Psi a^{jk}h_{x_j}h+a^{jk}\Big(A\ell_{x_j}+\frac{\Psi_{x_j}}{2}\Big)h^2\bigg]_{x_k}dt\\
&=2\sum_{j,k=1}^N c^{jk}h_{x_j}h_{x_k}dt+Bh^2dt+d\bigg(-\sum_{j,k=1}^N a^{jk}h_{x_j}h_{x_k}+Ah^2\bigg)\\
\label{5.121301}& \;\;\;\,+2\big[\nabla\cdot(\mathcal{A}\nabla h)+Ah\big]^2dt+\theta^2\sum_{j,k=1}^N a^{jk}(dz_{xj}+\ell_{x_j}dz)(dz_{x_k}+\ell_{x_k}dz)\\
& \;\;\;\,-\theta^2A(dz)^2.
\end{aligned}
\end{align}
\end{thm}

For the boundary operator, we have the following fundamental weighted identity
for the stochastic parabolic-type operator on $\Gamma$:
\[
dz_\Gamma + \sum_{j,k=1}^N 
D_j\Bigl(a_\Gamma^{jk}(x) D_k z_\Gamma\Bigr)\, dt
- \sum_{j,k=1}^N a^{jk}(x) \frac{\partial z}{\partial x_j} \nu^k\, dt
= dz_\Gamma + \nabla_\Gamma \cdot (\mathcal{A}_\Gamma \nabla_\Gamma z_\Gamma)\, dt
- \partial_\nu^\mathcal{A} z\, dt,
\]
where $dz_\Gamma$ is understood in the It\^o sense, 
$\nabla_\Gamma \cdot (\mathcal{A}_\Gamma \nabla_\Gamma z_\Gamma)$ denotes
the tangential diffusion on the boundary, and 
$\partial_\nu^\mathcal{A} z = \sum_{j,k=1}^N a^{jk} \frac{\partial z}{\partial x_j} \nu^k$
represents the coupling conormal derivative along $\Gamma$.
\begin{thm}\label{thm5.1sec}
Let $z_\Gamma$ be an $H^2(\Gamma)$-valued It\^o process, and let
$\theta = e^\ell$, $h_\Gamma = \theta z_\Gamma$ with $z_\Gamma = z|_\Gamma$. 
Then, for any $t \in [0,T]$ and a.e. $(x,\omega) \in \Gamma \times \Omega$, 
the following weighted identity holds:
\begin{align}\label{5.121301secc}
\begin{aligned}
& 2\theta \Big[\nabla_\Gamma \cdot (\mathcal{A}_\Gamma \nabla_\Gamma h_\Gamma) + \widetilde{A} h_\Gamma \Big] 
\Big[ dz_\Gamma + \nabla_\Gamma \cdot (\mathcal{A}_\Gamma \nabla_\Gamma z_\Gamma)\, dt - \partial_\nu^\mathcal{A} z \, dt \Big]
- 2 \nabla_\Gamma \cdot (\mathcal{A}_\Gamma \nabla_\Gamma h_\Gamma \, dh_\Gamma) \\[0.2em]
&= \widetilde{B} h_\Gamma^2\, dt 
+ d\Big( - \mathcal{A}_\Gamma \nabla_\Gamma h_\Gamma \cdot \nabla_\Gamma h_\Gamma + \widetilde{A} h_\Gamma^2 \Big)
+ 2 \Big[ \nabla_\Gamma \cdot (\mathcal{A}_\Gamma \nabla_\Gamma h_\Gamma) + \widetilde{A} h_\Gamma \Big]^2 dt \\[0.2em]
&\quad + (\mathcal{A}_\Gamma)_t \nabla_\Gamma h_\Gamma \cdot \nabla_\Gamma h_\Gamma\, dt
+ \mathcal{A}_\Gamma \, d\nabla_\Gamma h_\Gamma \cdot d\nabla_\Gamma h_\Gamma
- \theta^2 \widetilde{A} |dz_\Gamma|^2 \\[0.2em]
&\quad - 2 \nabla_\Gamma \cdot (\Phi h_\Gamma \mathcal{A}_\Gamma \nabla_\Gamma h_\Gamma)\, dt
+ 2 \Phi \, \mathcal{A}_\Gamma \nabla_\Gamma h_\Gamma \cdot \nabla_\Gamma h_\Gamma \, dt
+ 2 h_\Gamma \, \mathcal{A}_\Gamma \nabla_\Gamma h_\Gamma \cdot \nabla_\Gamma \Phi \, dt \\[0.2em]
&\quad - 2 \theta^2 \partial_\nu^\mathcal{A} z \, \nabla_\Gamma \cdot (\mathcal{A}_\Gamma \nabla_\Gamma z_\Gamma)\, dt
- 2 \theta^2 \widetilde{A} z_\Gamma \, \partial_\nu^\mathcal{A} z \, dt.
\end{aligned}
\end{align}
\end{thm}

\begin{proof}
We first compute the leading term on the left-hand side of \eqref{5.121301secc}. 
By the definition of $\theta$, we have 
\begin{align}\label{iridenfobouop}
\begin{aligned}
&\,2\theta\left[\nabla_\Gamma\cdot(\mathcal{A}_\Gamma\nabla_\Gamma h_\Gamma)+\widetilde{A}h_\Gamma\right]\left[dz_\Gamma + \nabla_\Gamma \cdot (\mathcal{A}_\Gamma \nabla_\Gamma z_\Gamma)\,dt-\partial_\nu^\mathcal{A}z \, dt\right]\\
&=2\left[\nabla_\Gamma\cdot(\mathcal{A}_\Gamma\nabla_\Gamma h_\Gamma)+\widetilde{A}h_\Gamma\right]\left[\nabla_\Gamma\cdot(\mathcal{A}_\Gamma\nabla_\Gamma h_\Gamma)dt+\widetilde{A}h_\Gamma dt+dh_\Gamma-\Phi h_\Gamma dt-\theta\partial_\nu^\mathcal{A}z \, dt\right]\\
&=2I(I_1+I_2+I_3+I_4),
\end{aligned}
\end{align}
where 
\begin{align*}
&I=\nabla_\Gamma\cdot(\mathcal{A}_\Gamma\nabla_\Gamma h_\Gamma)+\widetilde{A}h_\Gamma ,\quad I_1=\nabla_\Gamma\cdot(\mathcal{A}_\Gamma\nabla_\Gamma h_\Gamma)dt+\widetilde{A}h_\Gamma dt,\\
& I_2=dh_\Gamma,\quad I_3=-\Phi h_\Gamma dt,\quad I_4=-\theta\partial_\nu^\mathcal{A}z \, dt.
\end{align*}
Firstly, we have
\begin{align}\label{ini1}
2II_1=2\left[\nabla_\Gamma\cdot(\mathcal{A}_\Gamma\nabla_\Gamma h_\Gamma)+\widetilde{A}h_\Gamma\right]^2dt.
\end{align}
By Itô's formula, we obtain
\begin{align*}
\begin{aligned}
2II_2=&2\nabla_\Gamma\cdot(\mathcal{A}_\Gamma\nabla_\Gamma h_\Gamma dh_\Gamma)-d(\mathcal{A}_\Gamma\nabla_\Gamma h_\Gamma\cdot\nabla_\Gamma h_\Gamma)+(\mathcal{A}_\Gamma)_t\nabla_\Gamma h_\Gamma\cdot\nabla_\Gamma h_\Gamma dt\\
&+\mathcal{A}_\Gamma d\nabla_\Gamma h_\Gamma\cdot d\nabla_\Gamma h_\Gamma+d(\widetilde{A}h_\Gamma^2)-\widetilde{A}_th_\Gamma^2dt-\widetilde{A}|dh_\Gamma|^2,
\end{aligned}
\end{align*}
which can be equivalently written as
\begin{align}\label{ini2}
\begin{aligned}
2II_2=&2\nabla_\Gamma\cdot(\mathcal{A}_\Gamma\nabla_\Gamma h_\Gamma dh_\Gamma)-d(\mathcal{A}_\Gamma\nabla_\Gamma h_\Gamma\cdot\nabla_\Gamma h_\Gamma)+(\mathcal{A}_\Gamma)_t\nabla_\Gamma h_\Gamma\cdot\nabla_\Gamma h_\Gamma dt\\
&+\mathcal{A}_\Gamma d\nabla_\Gamma h_\Gamma\cdot d\nabla_\Gamma h_\Gamma+d(\widetilde{A}h_\Gamma^2)-\widetilde{A}_th_\Gamma^2dt-\theta^2\widetilde{A}|dz_\Gamma|^2.
\end{aligned}
\end{align}
Similarly, we have
\begin{align}\label{ini3}
2II_3=-2\nabla_\Gamma\cdot(\Phi h_\Gamma\mathcal{A}_\Gamma\nabla_\Gamma h_\Gamma)dt+2\Phi\mathcal{A}_\Gamma\nabla_\Gamma h_\Gamma\cdot\nabla_\Gamma h_\Gamma dt+2h_\Gamma\mathcal{A}_\Gamma\nabla_\Gamma h_\Gamma\cdot\nabla_\Gamma\Phi-2\widetilde{A}\Phi h_\Gamma^2dt.
\end{align}
Finally, for the boundary coupling term, we have
\begin{align}\label{ini4}
2II_4=-2\theta^2\partial_\nu^\mathcal{A}z\nabla_\Gamma\cdot(\mathcal{A}_\Gamma\nabla_\Gamma z_\Gamma)dt-2\theta^2\widetilde{A}z_\Gamma\partial_\nu^\mathcal{A}z dt.
\end{align}
Combining \eqref{iridenfobouop}, \eqref{ini1}, \eqref{ini2}, \eqref{ini3}, and \eqref{ini4},  we deduce the desired identity \eqref{5.121301secc}.
\end{proof}

In what follows, we choose the auxiliary functions
\[
\Psi = -2 \sum_{j,k=1}^N a^{jk} \ell_{x_j x_k},
\qquad
\Phi = \frac{2}{\beta_0}\, M \, (\mathcal{A}\nu \cdot \nu)\, |\nabla \ell|,
\]
where the function $\ell$ is chosen to be the one defined in \eqref{3.2}. Moreover, in Lemma \ref{lm3.1}, the set $G_1$ is taken to be any fixed nonempty open subset satisfying $G_1 \Subset \mathcal{B}$, where $\mathcal{B}$ is any nonempty open subset of $G$. For a positive integer $n$, we denote by $O(\mu^n)$ a function of order $\mu^n$ for sufficiently large $\mu$. Similarly, the notation $O\!\left(e^{\mu \|\psi\|_\infty}\right)$ stands for a generic
quantity of order $e^{\mu \|\psi\|_\infty}$.
By \eqref{3.2} and \eqref{3.3}, it is straightforward to verify that for any
parameters $\lambda, \mu \ge 1$ and for all $j,k = 1,2,\cdots,N$, one has
\begin{equation}\label{5.1201}
\ell_t = \lambda \alpha_t,
\qquad
\ell_{x_j} = \lambda \mu \varphi \psi_{x_j},
\qquad
\ell_{x_j x_k}
= \lambda \mu^2 \varphi \psi_{x_j} \psi_{x_k}
+ \lambda \mu \varphi \psi_{x_j x_k}.
\end{equation}
Moreover, the following estimates hold:
\begin{align}\label{5.1301}
\begin{aligned}
& \alpha_t = T \varphi^2\, O\!\left(e^{2\mu \|\psi\|_\infty}\right),
\qquad
\varphi_t = T \varphi^2\, O(1), \\[0.2em]
& \alpha_{tt} = T^2 \varphi^3\, O\!\left(e^{2\mu \|\psi\|_\infty}\right),
\qquad
\varphi_{tt} = T^2 \varphi^3\, O(1).
\end{aligned}
\end{align}

We next give some useful estimates for the quantities \( A \), \( B \), and \( c^{jk} \) defined in \eqref{5.1111012} (see \cite[Theorem A.2]{Preprintelgrou23}). 
\begin{lm}\label{thm5.2}
There exists a constant
$C = C(G, \mathcal{B},\beta_0,M,M_\Gamma) > 0$ such that, for all $\mu \ge C$ and $\lambda \ge C T^2$, the following estimates hold:
\begin{align*}
A &= \lambda^2 \mu^2 \varphi^2 \sum_{j,k=1}^N a^{jk} \psi_{x_j} \psi_{x_k} + \lambda \varphi O(\mu^2) + \lambda T \varphi^2 O(e^{2\mu \|\psi\|_\infty}), \\
B \geq &\, 2 \beta_0^2 \lambda^3 \mu^4 \varphi^3 |\nabla \psi|^4 + \lambda^3 \varphi^3 O(\mu^4)  + \lambda^2 T \varphi^3 O(\mu^2 e^{2\mu \|\psi\|_\infty}), \\
\sum_{j,k=1}^N c^{jk} \eta_j \eta_k \geq &\, \left[\beta_0^2 \lambda \mu^2 \varphi |\nabla \psi|^2 + \lambda \varphi O(\mu) \right] |\eta|^2, \quad \forall \eta = (\eta_1, \cdots, \eta_N) \in \mathbb{R}^N,
\end{align*}
for almost every $(x,\omega)\in \overline{G\setminus G_1}\times\Omega$
(where $G_1$ is defined in Lemma~\ref{lm3.1})
and for all $t \in [0,T]$.
\end{lm} 
Similarly, for the boundary quantities $\widetilde{A}$ and $\widetilde{B}$
defined in \eqref{5.1111012}, we have the following estimates.
\begin{lm}\label{thm5.2sec}
The following identities hold:
\begin{align*}
&\widetilde{A}=-\frac{2}{\beta_0}M\lambda\mu\varphi\partial_\nu^\mathcal{A}\psi+\lambda T\varphi^2O(e^{2\mu\|\psi\|_\infty}),\\
&\widetilde{B}=-\frac{8}{\beta_0^2}M\lambda^2\mu^2\varphi^2|\partial_\nu^\mathcal{A} \psi|^2+\lambda^2T\varphi^3O(\mu e^{2\mu\|\psi\|_\infty})\\
&\hspace{0.8cm}+\lambda T^2\varphi^3O( e^{4\mu\|\psi\|_\infty}),
\end{align*}
for all \( x \in\Gamma\ \), and for any \( t \in [0,T] \), \,\( \textnormal{a.s.} \)
\end{lm} 

\begin{proof}
We begin by recalling the definition $\widetilde{A} = \Phi - \ell_t$.
By the choice of $\Phi$ and the properties of $\ell$, we have
\begin{align*}
\widetilde{A}
&= \frac{2}{\beta_0}M
(\mathcal{A}\nu \cdot \nu)\, |\nabla \ell|
- \ell_t \\
&= \frac{2}{\beta_0}M
\lambda \mu \varphi (\mathcal{A}\nu \cdot \nu)\, |\nabla \psi|
+ \lambda T \varphi^2\, O\!\left(e^{2\mu\|\psi\|_\infty}\right).
\end{align*}
Since $\nabla \psi = (\partial_\nu \psi)\nu$ on $\Gamma$, it follows that
\[
(\mathcal{A}\nu \cdot \nu)\, |\nabla \psi|
= (\mathcal{A}\nu \cdot \nu)\, |\partial_\nu \psi|.
\]
Moreover, using the identity
$\partial_\nu^\mathcal{A} \psi = (\mathcal{A}\nu \cdot \nu)\partial_\nu \psi$
and the fact that $\partial_\nu \psi < 0$ on $\Gamma$, we obtain
\[
(\mathcal{A}\nu \cdot \nu)\, |\partial_\nu \psi|
= - \partial_\nu^\mathcal{A} \psi.
\]
Consequently,
\begin{equation*}
\widetilde{A}
= -\frac{2}{\beta_0}M
\lambda \mu \varphi\, \partial_\nu^\mathcal{A} \psi
+ \lambda T \varphi^2\, O\!\left(e^{2\mu\|\psi\|_\infty}\right).
\end{equation*}
We next estimate $\widetilde{B}$. By definition,
\[
\widetilde{B} = -2\widetilde{A}\Phi - \widetilde{A}_t.
\]
Substituting the above expression for $\widetilde{A}$ and recalling that
$\Phi = -\frac{2}{\beta_0}M
\lambda \mu \varphi\, \partial_\nu^\mathcal{A} \psi$, we compute
\begin{align*}
\widetilde{B}
&= -2\left(
-\frac{2}{\beta_0}M
\lambda \mu \varphi\, \partial_\nu^\mathcal{A} \psi
+ \lambda T \varphi^2\, O\!\left(e^{2\mu\|\psi\|_\infty}\right)
\right)
\left(
-\frac{2}{\beta_0}M
\lambda \mu \varphi\, \partial_\nu^\mathcal{A} \psi
\right) \\
&\quad
- \partial_t\!\left(
-\frac{2}{\beta_0}M
\lambda \mu \varphi\, \partial_\nu^\mathcal{A} \psi
+ \lambda T \varphi^2\, O\!\left(e^{2\mu\|\psi\|_\infty}\right)
\right).
\end{align*}
Using the bounds for $\varphi_t$ in \eqref{5.1301}, we deduce
\begin{align*}
\widetilde{B}
&= -\frac{8}{\beta_0^2}M^2
\lambda^2 \mu^2 \varphi^2
|\partial_\nu^\mathcal{A} \psi|^2
+ \lambda^2 T \varphi^3\, O\!\left(\mu e^{2\mu\|\psi\|_\infty}\right) \\
&\quad
+ \lambda T^2 \varphi^3\, O\!\left(e^{4\mu\|\psi\|_\infty}\right),
\end{align*}
which completes the proof.
\end{proof}

As a consequence of the weighted identities and estimates established above, we obtain the following Carleman estimate for system \eqref{1.1sec}.
\begin{lm}\label{lm2.2}
There exists a constant \( C = C(G, \mathcal{B}, \beta_0,M,M_\Gamma) > 0 \) such that for all \( F_1 \in L^2_\mathcal{F}(0,T;L^2(G)) \), \( F_2 \in L^2_\mathcal{F}(0,T;L^2(\Gamma)) \), and the final state \( (z_T, z_{\Gamma,T}) \in L^2_{\mathcal{F}_T}(\Omega; \mathbb{L}^2) \), the weak solution $(z,z_\Gamma;Z,\widehat{Z})$ of \eqref{1.1sec} satisfies the estimate
\begin{align}\label{3.5}
 \begin{aligned}
&\,\lambda^3 \mu^4 \mathbb{E}\iint_Q \theta^2 \varphi^3 |z|^2 \, dx  dt + \lambda^3 \mu^3 \mathbb{E}\iint_\Sigma \theta^2 \varphi^3 |z_\Gamma|^2 \, d\sigma  dt \\
&+ \lambda \mu^2 \mathbb{E}\iint_Q \theta^2 \varphi |\nabla z|^2 \, dx  dt + \lambda \mu \mathbb{E}\iint_\Sigma \theta^2 \varphi |\nabla_\Gamma z_\Gamma|^2 \, d\sigma  dt \\
&\leq C \bigg[ \lambda^3 \mu^4 \mathbb{E}\int_0^T\int_\mathcal{B} \theta^2 \varphi^3 |z|^2 \, dx  dt + \mathbb{E}\iint_Q \theta^2 |F_1|^2 \, dx  dt + \mathbb{E}\iint_\Sigma \theta^2 |F_2|^2 \, d\sigma  dt \\
&\qquad\quad + \lambda^2 \mu^2 \mathbb{E}\iint_Q \theta^2 \varphi^2 |Z|^2 \, dx  dt + \lambda^2 \mu \mathbb{E}\iint_\Sigma \theta^2 \varphi^2 |\widehat{Z}|^2 \, d\sigma  dt \bigg],
\end{aligned}
\end{align}
for large enough \( \mu \geq C \) and \( \lambda \geq C(e^{2 \mu \|\psi\|_\infty} T + T^2) \).
\end{lm}

\begin{proof} For clarity, we divide the proof into three steps.\\
\textbf{Step 1. Preliminary estimates.}\\
Integrating \eqref{5.121301} on $Q$, taking expectation on both sides and using the fact that $\theta(0,\cdot)=\theta(T,\cdot)=0$ and recalling \eqref{asummAandAgama}, we conclude that
\begin{align}\label{5.121301sec}
\begin{aligned}
&2\mathbb{E}\iint_Q\sum_{j,k=1}^N c^{jk}h_{x_j}h_{x_k} dxdt+\mathbb{E}\iint_Q Bh^2 dxdt\\
&\leq \mathbb{E}\iint_Q\theta^2 F_1^2 dxdt-2\mathbb{E}\iint_Q\nabla\cdot(\mathcal{A}\nabla h \,dh) dx+\mathbb{E}\iint_Q\theta^2AZ^2 dxdt\\
&\quad\;+2\mathbb{E}\iint_Q\sum_{j,k=1}^N\bigg[\sum_{j',k'=1}^N \big(2a^{jk}a^{j'k'}\ell_{x_{j'}}h_{x_{j}}h_{x_{k'}}-a^{jk}a^{j'k'}\ell_{x_j}h_{x_{j'}}h_{x_{k'}}\big)\\
&\qquad\qquad\qquad\quad\;\;\;-\Psi a^{jk}h_{x_j}h+a^{jk}\Big(A\ell_{x_j}+\frac{\Psi_{x_j}}{2}\Big)h^2\bigg]_{x_k} dxdt.
\end{aligned}
\end{align}
Integrating \eqref{5.121301secc} on $\Sigma$, taking expectation on both sides and using the fact that
$\theta(0,\cdot)=\theta(T,\cdot)=0$ and recalling \eqref{asummAandAgama}, we conclude that
\begin{align}\label{5.A71sec}
\begin{aligned}
&\,\mathbb{E}\iint_\Sigma\widetilde{B}h_\Gamma^2d\sigma dt+2\beta_0\mathbb{E}\iint_\Sigma\Phi|\nabla_\Gamma h_\Gamma|^2 d\sigma dt\\
&\leq \mathbb{E}\iint_\Sigma\theta^2 F_2^2d\sigma dt+\mathbb{E}\iint_\Sigma\theta^2\widetilde{A}\widehat{Z}^2d\sigma dt-2\mathbb{E}\iint_\Sigma h_\Gamma\mathcal{A}_\Gamma\nabla_\Gamma h_\Gamma\cdot\nabla_\Gamma\Phi d\sigma dt\\
&\quad-\mathbb{E}\iint_\Sigma \theta^2(\mathcal{A}_\Gamma)_t\nabla_\Gamma z_\Gamma\cdot\nabla_\Gamma z_\Gamma d\sigma dt+2\mathbb{E}\iint_\Sigma \theta^2\partial_\nu^\mathcal{A}z\nabla_\Gamma\cdot(\mathcal{A}_\Gamma\nabla_\Gamma z_\Gamma) d\sigma dt\\
&\quad+2\mathbb{E}\iint_\Sigma \theta^2\widetilde{A}z_\Gamma\partial_\nu^\mathcal{A}z d\sigma dt.
\end{aligned}
\end{align}
Combining \eqref{5.121301sec} and \eqref{5.A71sec}, we get that
\begin{align*}
\begin{aligned}
&2\mathbb{E}\iint_Q\sum_{j,k=1}^N c^{jk}h_{x_j}h_{x_k} dxdt+\mathbb{E}\iint_Q Bh^2 dxdt\\
&+\mathbb{E}\iint_\Sigma\widetilde{B}h_\Gamma^2d\sigma dt+2\beta_0\mathbb{E}\iint_\Sigma\Phi|\nabla_\Gamma h_\Gamma|^2 d\sigma dt\\
&\leq \mathbb{E}\iint_Q\theta^2 F_1^2 dxdt+\mathbb{E}\iint_\Sigma\theta^2 F_2^2d\sigma dt+\mathbb{E}\iint_Q\theta^2AZ^2 dxdt\\
&\quad\;+2\mathbb{E}\iint_Q\sum_{j,k=1}^N\bigg[\sum_{j',k'=1}^N \big(2a^{jk}a^{j'k'}\ell_{x_{j'}}h_{x_{j}}h_{x_{k'}}-a^{jk}a^{j'k'}\ell_{x_j}h_{x_{j'}}h_{x_{k'}}\big)\\
&\qquad\qquad\qquad\quad\;\;\;-\Psi a^{jk}h_{x_j}h+a^{jk}\Big(A\ell_{x_j}+\frac{\Psi_{x_j}}{2}\Big)h^2\bigg]_{x_k} dxdt\\
&\quad\;-2\mathbb{E}\iint_Q\nabla\cdot(\mathcal{A}\nabla h \,dh) dx+\mathbb{E}\iint_\Sigma\theta^2\widetilde{A}\widehat{Z}^2d\sigma dt-2\mathbb{E}\iint_\Sigma h_\Gamma\mathcal{A}_\Gamma\nabla_\Gamma h_\Gamma\cdot\nabla_\Gamma\Phi d\sigma dt\\
&\quad\;-\mathbb{E}\iint_\Sigma \theta^2(\mathcal{A}_\Gamma)_t\nabla_\Gamma z_\Gamma\cdot\nabla_\Gamma z_\Gamma d\sigma dt+2\mathbb{E}\iint_\Sigma \theta^2\partial_\nu^\mathcal{A}z\nabla_\Gamma\cdot(\mathcal{A}_\Gamma\nabla_\Gamma z_\Gamma) d\sigma dt\\
&\quad\;+2\mathbb{E}\iint_\Sigma \theta^2\widetilde{A}z_\Gamma\partial_\nu^\mathcal{A}z d\sigma dt.
\end{aligned}
\end{align*}
Hence, it follows that
\begin{align}\label{5.121301seces1}
\begin{aligned}
&2\mathbb{E}\iint_Q\sum_{j,k=1}^N c^{jk}h_{x_j}h_{x_k} dxdt+\mathbb{E}\iint_Q Bh^2 dxdt+\textbf{I}_1+\textbf{I}_2\\
&\leq \mathbb{E}\iint_Q\theta^2 F_1^2 dxdt+\mathbb{E}\iint_\Sigma\theta^2 F_2^2d\sigma dt+\mathbb{E}\iint_Q\theta^2AZ^2 dxdt+\sum_{i=3}^9 \textbf{I}_i.
\end{aligned}
\end{align}
\textbf{Step 2. Estimating the boundary terms \( \textbf{I}_i \), \( i = 1, 2, \cdots, 9 \).}\\
From Lemma \ref{thm5.2sec}, we first have that
\begin{align*}
    \begin{aligned}
\textbf{I}_1=\mathbb{E}\iint_\Sigma\widetilde{B}h_\Gamma^2d\sigma dt=\mathbb{E}\iint_\Sigma\theta^2\bigg[&-\frac{8}{\beta_0^2}M^2\lambda^2\mu^2\varphi^2|\partial_\nu^\mathcal{A} \psi|^2+\lambda^2T\varphi^3O(\mu e^{2\mu\|\psi\|_\infty})\\
&+\lambda T^2\varphi^3O( e^{4\mu\|\psi\|_\infty})\bigg] z_\Gamma^2d\sigma dt,
    \end{aligned}
\end{align*}
which provides that
\begin{align}\label{ineq8}
    \begin{aligned}
\textbf{I}_1\geq &-C\lambda^2\mu^2\mathbb{E}\iint_\Sigma\theta^2\varphi^2z_\Gamma^2 d\sigma dt-C\lambda^2\mu T e^{2\mu\|\psi\|_\infty}\mathbb{E}\iint_\Sigma\theta^2\varphi^3z_\Gamma^2 d\sigma dt\\
&-C\lambda T^2 e^{4\mu\|\psi\|_\infty}\mathbb{E}\iint_\Sigma\theta^2\varphi^3z_\Gamma^2 d\sigma dt.
    \end{aligned}
\end{align}
We also have that
\begin{align*}
\textbf{I}_2=2\beta_0\mathbb{E}\displaystyle\iint_\Sigma\Phi \vert\nabla_\Gamma v_\Gamma\vert^2 d\sigma dt =  4\lambda\mu M\mathbb{E}\displaystyle\iint_\Sigma \theta^2\varphi \vert\partial_\nu^\mathcal{A}\psi\vert \vert\nabla_\Gamma z_\Gamma\vert^2  d\sigma dt,
\end{align*}
which leads to
\begin{align}\label{ineq1}
\textbf{I}_2 =  -4\lambda\mu M\mathbb{E}\displaystyle\iint_\Sigma \theta^2\varphi \partial_\nu^\mathcal{A}\psi \vert\nabla_\Gamma z_\Gamma\vert^2  d\sigma dt.
\end{align}
In what follows, we use the fact that for all $(t,x)\in\Sigma$, we have
\begin{equation*}
		\nabla\ell =\lambda\mu\varphi\partial_\nu\psi \nu\quad\textnormal{and}\quad\nabla h=\theta\nabla z+\lambda\mu\theta\varphi z_\Gamma\partial_\nu\psi \nu.
        \end{equation*}
Notice that
\begin{align*}
\textbf{I}_3&=2\mathbb{E}\iint_Q\sum_{j,k=1}^N\bigg[\sum_{j',k'=1}^N \big(2a^{jk}a^{j'k'}\ell_{x_{j'}}h_{x_{j}}h_{x_{k'}}-a^{jk}a^{j'k'}\ell_{x_j}h_{x_{j'}}h_{x_{k'}}\big)\\
&\hspace{3.6cm}-\Psi a^{jk}h_{x_j}h+a^{jk}\Big(A\ell_{x_j}+\frac{\Psi_{x_j}}{2}\Big)h^2\bigg]_{x_k} dxdt\\
&=\textbf{I}_3^1+\textbf{I}_3^2+\textbf{I}_3^3, 
\end{align*}
where
\begin{align*}
&\textbf{I}_3^1=2\mathbb{E}\iint_Q\sum_{j,k=1}^N\bigg[\sum_{j',k'=1}^N \big(2a^{jk}a^{j'k'}\ell_{x_{j'}}h_{x_{j}}h_{x_{k'}}-a^{jk}a^{j'k'}\ell_{x_j}h_{x_{j'}}h_{x_{k'}}\big)\bigg]_{x_k} dxdt,\\
&\textbf{I}_3^2=-2\mathbb{E}\iint_Q\sum_{j,k=1}^N[\Psi a^{jk}h_{x_j}h]_{x_k} dxdt,\\
&\textbf{I}_3^3=2\mathbb{E}\iint_Q\sum_{j,k=1}^N\bigg[ a^{jk}\Big(A\ell_{x_j}+\frac{\Psi_{x_j}}{2}\Big)h^2\bigg]_{x_k} dxdt.
\end{align*}
By direct computations, we obtain that
\begin{align}\label{ina11h}
\begin{aligned}
\textbf{I}_3^1&=2\lambda^3\mu^3\mathbb{E}\iint_\Sigma \theta^2\varphi^3\partial_\nu\psi|\partial_\nu^\mathcal{A}\psi|^2z_\Gamma^2 d\sigma dt+4\lambda^2\mu^2\mathbb{E}\iint_\Sigma\theta^2\varphi^2\partial_\nu\psi\partial_\nu^\mathcal{A}\psi z_\Gamma\partial_\nu^\mathcal{A} z d\sigma dt\\
&\quad+4\lambda\mu\mathbb{E}\iint_\Sigma\theta^2\varphi\partial_\nu\psi|\partial_\nu^\mathcal{A} z|^2 d\sigma dt-2\lambda\mu\mathbb{E}\iint_\Sigma\theta^2\varphi\partial_\nu\psi(\mathcal{A}\nabla z\cdot\nabla z)(\mathcal{A}\nu\cdot\nu)d\sigma dt.
\end{aligned}
\end{align}
Using the inequity $4ab\leq\frac{8}{3}a^2+\frac{3}{2}b^2$ for the second term in the right-hand side of \eqref{ina11h}, we obtain that
\begin{align*}
\begin{aligned}
\textbf{I}_3^1&\leq-\frac{2}{3}\lambda^3\mu^3\mathbb{E}\iint_\Sigma \theta^2\varphi^3\partial_\nu\psi|\partial_\nu^\mathcal{A}\psi|^2z_\Gamma^2 d\sigma dt+\frac{5}{2}\lambda\mu\mathbb{E}\iint_\Sigma\theta^2\varphi\partial_\nu\psi|\partial_\nu^\mathcal{A} z|^2 d\sigma dt\\
&\quad-2\lambda\mu\mathbb{E}\iint_\Sigma\theta^2\varphi\partial_\nu\psi(\mathcal{A}\nabla z\cdot\nabla z)(\mathcal{A}\nu\cdot\nu)d\sigma dt.
\end{aligned}
\end{align*}
It follows that
\begin{align}\label{ineqq1}
\begin{aligned}
\textbf{I}_3^1&\leq-\frac{2}{3}\lambda^3\mu^3\mathbb{E}\iint_\Sigma \theta^2\varphi^3\partial_\nu\psi|\partial_\nu^\mathcal{A}\psi|^2z_\Gamma^2 d\sigma dt+\frac{1}{2}\lambda\mu\mathbb{E}\iint_\Sigma\theta^2\varphi\partial_\nu\psi|\partial_\nu^\mathcal{A} z|^2 d\sigma dt\\
&\quad+2\lambda\mu\mathbb{E}\iint_\Sigma\theta^2\varphi\partial_\nu\psi|\partial_\nu^\mathcal{A} z|^2 d\sigma dt-2\lambda\mu\mathbb{E}\iint_\Sigma\theta^2\varphi\partial_\nu\psi(\mathcal{A}\nabla z\cdot\nabla z)(\mathcal{A}\nu\cdot\nu)d\sigma dt.
\end{aligned}
\end{align}
To estimate the last two terms on the right-hand side of \eqref{ineqq1}, we use the identity (see \cite[Lemma 2.3]{haschmano})
\begin{align*}
    |\partial_\nu^\mathcal{A} z|^2-(\mathcal{A}\nabla_\Gamma z_\Gamma\cdot\nu)^2=(\mathcal{A}\nu\cdot\nu)(\mathcal{A}\nabla z\cdot\nabla z - \mathcal{A}\nabla_\Gamma z_\Gamma\cdot\nabla_\Gamma z_\Gamma).
\end{align*}
This implies that
\begin{align*}
\begin{aligned}
&2\lambda\mu\mathbb{E}\iint_\Sigma\theta^2\varphi\partial_\nu\psi|\partial_\nu^\mathcal{A} z|^2 d\sigma dt-2\lambda\mu\mathbb{E}\iint_\Sigma\theta^2\varphi\partial_\nu\psi(\mathcal{A}\nabla z\cdot\nabla z)(\mathcal{A}\nu\cdot\nu)d\sigma dt\\
&
\leq -2\lambda\mu\mathbb{E}\iint_\Sigma\theta^2\varphi\partial_\nu^\mathcal{A}\psi (\mathcal{A}\nabla_\Gamma z_\Gamma\cdot\nabla_\Gamma z_\Gamma) d\sigma dt,
\end{aligned}
\end{align*}
which provides that
\begin{align}\label{ineqq2}
\begin{aligned}
&2\lambda\mu\mathbb{E}\iint_\Sigma\theta^2\varphi\partial_\nu\psi|\partial_\nu^\mathcal{A} z|^2 d\sigma dt-2\lambda\mu\mathbb{E}\iint_\Sigma\theta^2\varphi\partial_\nu\psi(\mathcal{A}\nabla z\cdot\nabla z)(\mathcal{A}\nu\cdot\nu)d\sigma dt\\
&
\leq -2\lambda\mu M\mathbb{E}\iint_\Sigma\theta^2\varphi\partial_\nu^\mathcal{A}\psi |\nabla_\Gamma z_\Gamma|^2 d\sigma dt.
\end{aligned}
\end{align}
From \eqref{ineqq1} and \eqref{ineqq2}, we deduce that
\begin{align}\label{ineq3}
\begin{aligned}
\textbf{I}_3^1&\leq-\frac{2}{3}\lambda^3\mu^3\mathbb{E}\iint_\Sigma \theta^2\varphi^3\partial_\nu\psi|\partial_\nu^\mathcal{A}\psi|^2z_\Gamma^2 d\sigma dt+\frac{1}{2}\lambda\mu\mathbb{E}\iint_\Sigma\theta^2\varphi\partial_\nu\psi|\partial_\nu^\mathcal{A} z|^2 d\sigma dt\\
&\quad-2\lambda\mu M\mathbb{E}\iint_\Sigma\theta^2\varphi\partial_\nu^\mathcal{A}\psi |\nabla_\Gamma z_\Gamma|^2 d\sigma dt.
\end{aligned}
\end{align}
For large enough $\mu\geq C$, we have that
\begin{align*}
\begin{aligned}
\textbf{I}_3^2&=-2\mathbb{E}\iint_\Sigma\sum_{j,k=1}^N \Psi a^{jk}(\lambda\mu\theta\varphi z_\Gamma\partial_\nu\psi\nu_j+\theta z_{x_j})\theta z_\Gamma\nu_k d\sigma dt\\
&=-2\mathbb{E}\iint_\Sigma\Psi (\lambda\mu\theta^2\varphi z_\Gamma^2\partial_\nu\psi\mathcal{A}\nu\cdot\nu+\theta^2 z_\Gamma\partial_\nu^\mathcal{A}z) d\sigma dt\\
&=-2\mathbb{E}\iint_\Sigma \bigg(-2\lambda\mu^2\varphi\mathcal{A}\nabla\psi\cdot\nabla\psi-2\lambda\mu\varphi\sum_{j,k=1}^N a^{jk}\psi_{x_jx_k}\bigg)\\
&\hspace{2cm}(\lambda\mu\theta^2\varphi z_\Gamma^2\partial_\nu\psi\mathcal{A}\nu\cdot\nu+\theta^2 z_\Gamma\partial_\nu^\mathcal{A}z) d\sigma dt,
\end{aligned}
\end{align*}
which leads to 
\begin{align}\label{ineqqa15}
\begin{aligned}
\textbf{I}_3^2\leq C\lambda\mu^2\mathbb{E}\iint_\Sigma\theta^2\varphi |z_\Gamma| |\partial_\nu^\mathcal{A}z|d\sigma dt+C\lambda^2\mu^2\mathbb{E}\iint_\Sigma\theta^2\varphi^2 z_\Gamma^2 d\sigma dt.
\end{aligned}
\end{align}
Applying Young's inequality for the first term on the right-hand side of \eqref{ineqqa15}, we have that
\begin{align}\label{ineq4}
\begin{aligned}
\textbf{I}_3^2\leq C\mu\mathbb{E}\iint_\Sigma\theta^2\varphi |\partial_\nu^\mathcal{A}z|^2d\sigma dt+C\lambda^2\mu^3\mathbb{E}\iint_\Sigma\theta^2\varphi z_\Gamma^2 d\sigma dt+C\lambda^2\mu^2\mathbb{E}\iint_\Sigma\theta^2\varphi^2 z_\Gamma^2 d\sigma dt.
\end{aligned}
\end{align}
On the other hand, it is easy to see that
\begin{align*}
\begin{aligned}
\textbf{I}_3^3&=2\mathbb{E}\iint_\Sigma (A\lambda\mu\varphi\partial_\nu\psi\mathcal{A}\nu\cdot\nu+\frac{1}{2}\partial_\nu^\mathcal{A}\Psi)\theta^2 z_\Gamma^2 d\sigma dt\\
&=2\mathbb{E}\iint_\Sigma \left[\lambda^3\mu^3\theta^2\varphi^3 z_\Gamma^2\partial_\nu\psi(\mathcal{A}\nu\cdot\nu)(\mathcal{A}\nabla\psi\cdot\nabla\psi)+\lambda^2\mu T\varphi^3O(e^{2\mu\|\psi\|_\infty})\right] d\sigma dt,
\end{aligned}
\end{align*}
which implies for large enough $\lambda\geq CTe^{2\mu\|\psi\|_\infty}$,
\begin{align}\label{ineq5}
\begin{aligned}
\textbf{I}_3^3\leq \frac{3}{2}\lambda^3\mu^3\mathbb{E}\iint_\Sigma \theta^2\varphi^3 \partial_\nu\psi|\partial_\nu^{\mathcal{A}}\psi|^2z_\Gamma^2 d\sigma dt.
\end{aligned}
\end{align}
We also have that
\begin{align*}
    \begin{aligned}
\textbf{I}_{4}&=-2\mathbb{E}\iint_Q\nabla\cdot(\mathcal{A}\nabla h \,dh) dx\\
&=-2\mathbb{E}\iint_\Sigma \lambda\mu\varphi\partial_\nu\psi\theta^2z_\Gamma(\mathcal{A}\nu\cdot\nu)(-\nabla_\Gamma\cdot(\mathcal{A}_\Gamma\nabla_\Gamma z_\Gamma)+\partial_\nu^\mathcal{A}z+F_2+\ell_t z_\Gamma)d\sigma dt\\
&\quad\;-2\mathbb{E}\iint_\Sigma \theta^2\partial_\nu^\mathcal{A}z(-\nabla_\Gamma\cdot(\mathcal{A}_\Gamma\nabla_\Gamma z_\Gamma)+\partial_\nu^\mathcal{A}z+F_2+\ell_t z_\Gamma)d\sigma dt.
    \end{aligned}
\end{align*}
Then, it follows that
\begin{align*}
    \begin{aligned}
\textbf{I}_{4}&=2\mathbb{E}\iint_\Sigma \lambda\mu\varphi\partial_\nu^\mathcal{A}\psi\theta^2z_\Gamma\nabla_\Gamma\cdot(\mathcal{A}_\Gamma\nabla_\Gamma z_\Gamma)d\sigma dt-2\mathbb{E}\iint_\Sigma \lambda\mu\varphi\partial^\mathcal{A}_\nu\psi\theta^2z_\Gamma\partial_\nu^\mathcal{A}zd\sigma dt\\
&\quad\;-2\mathbb{E}\iint_\Sigma \lambda\mu\varphi\partial^\mathcal{A}_\nu\psi\theta^2z_\Gamma F_2d\sigma dt-2\mathbb{E}\iint_\Sigma \lambda\mu\varphi\partial^\mathcal{A}_\nu\psi\theta^2z_\Gamma^2\ell_t  d\sigma dt
\\
&\quad\;+2\mathbb{E}\iint_\Sigma \theta^2\partial_\nu^\mathcal{A}z\nabla_\Gamma\cdot(\mathcal{A}_\Gamma\nabla_\Gamma z_\Gamma) d\sigma dt -2\mathbb{E}\iint_\Sigma \theta^2|\partial_\nu^\mathcal{A}z|^2 d\sigma dt \\
&\quad\;-2\mathbb{E}\iint_\Sigma \theta^2\partial_\nu^\mathcal{A}z F_2 d\sigma dt -2\mathbb{E}\iint_\Sigma \theta^2z_\Gamma\partial_\nu^\mathcal{A}z \ell_t d\sigma dt,
    \end{aligned}
\end{align*}
which gives that
\begin{align*}
    \begin{aligned}
\textbf{I}_{4}&=\mathbb{E}\iint_\Sigma (2\lambda\mu\theta^2\varphi\partial_\nu^\mathcal{A}\psi z_\Gamma+2\theta^2\partial_\nu^\mathcal{A}z)\nabla_\Gamma\cdot(\mathcal{A}_\Gamma\nabla_\Gamma z_\Gamma)d\sigma dt
\\
&\quad\;-2\mathbb{E}\iint_\Sigma \lambda\mu\varphi\partial^\mathcal{A}_\nu\psi\theta^2z_\Gamma\partial_\nu^\mathcal{A}zd\sigma dt-2\mathbb{E}\iint_\Sigma \lambda\mu\varphi\partial^\mathcal{A}_\nu\psi\theta^2z_\Gamma F_2d\sigma dt\\
&\quad\;-2\mathbb{E}\iint_\Sigma \lambda\mu\varphi\partial^\mathcal{A}_\nu\psi\theta^2z_\Gamma^2\ell_t  d\sigma dt-2\mathbb{E}\iint_\Sigma \theta^2|\partial_\nu^\mathcal{A}z|^2 d\sigma dt \\
&\quad\;-2\mathbb{E}\iint_\Sigma \theta^2\partial_\nu^\mathcal{A}z F_2 d\sigma dt -2\mathbb{E}\iint_\Sigma \theta^2z_\Gamma\partial_\nu^\mathcal{A}z \ell_t d\sigma dt.
    \end{aligned}
\end{align*}
Hence, it follows that
\begin{align}\label{ineqqa178}
    \begin{aligned}
\textbf{I}_{4}&\leq C\lambda\mu\mathbb{E}\iint_\Sigma \theta^2\varphi|\partial_\nu^\mathcal{A}\psi| |z_\Gamma||\nabla_\Gamma z_\Gamma|d\sigma dt +C\mathbb{E}\iint_\Sigma \theta^2|\partial_\nu^\mathcal{A}z||\nabla_\Gamma z_\Gamma|d\sigma dt
\\
&\quad\;+2\lambda\mu\mathbb{E}\iint_\Sigma \theta^2\varphi|\partial^\mathcal{A}_\nu\psi||z_\Gamma| |\partial_\nu^\mathcal{A}z|d\sigma dt+2\lambda\mu\mathbb{E}\iint_\Sigma \theta^2\varphi|\partial^\mathcal{A}_\nu\psi||z_\Gamma| |F_2|d\sigma dt\\
&\quad\;+2\lambda\mu\mathbb{E}\iint_\Sigma \theta^2\varphi|\partial^\mathcal{A}_\nu\psi|z_\Gamma^2 |\ell_t|  d\sigma dt+2\mathbb{E}\iint_\Sigma \theta^2|\partial_\nu^\mathcal{A}z|^2 d\sigma dt \\
&\quad\;+2\mathbb{E}\iint_\Sigma \theta^2|\partial_\nu^\mathcal{A}z| |F_2| d\sigma dt +2\mathbb{E}\iint_\Sigma \theta^2|z_\Gamma| |\partial_\nu^\mathcal{A}z| |\ell_t| d\sigma dt.
    \end{aligned}
\end{align}
Notice that $|\ell_t|\leq C\lambda Te^{2\mu\|\psi\|_\infty}\varphi^2$, and using Young's inequality in the right-hand side of \eqref{ineqqa178}, we find that
\begin{align}\label{ineq6}
    \begin{aligned}
\textbf{I}_{4}&\leq C\lambda^2\mu^2\mathbb{E}\iint_\Sigma \theta^2\varphi^2 z_\Gamma^2 d\sigma dt +C\mathbb{E}\iint_\Sigma \theta^2|\nabla_\Gamma z_\Gamma|^2 d\sigma dt+C\mathbb{E}\iint_\Sigma \theta^2 |\partial_\nu^\mathcal{A}z|^2 d\sigma dt
\\
&\quad\;+C\mathbb{E}\iint_\Sigma \theta^2 F_2^2 d\sigma dt+C\lambda^2\mu Te^{2\mu\|\psi\|_\infty}\mathbb{E}\iint_\Sigma \theta^2\varphi^3z_\Gamma^2   d\sigma dt\\
&\quad\;+CTe^{2\mu\|\psi\|_\infty}\mathbb{E}\iint_\Sigma \theta^2\varphi |\partial_\nu^\mathcal{A} z_\Gamma|^2   d\sigma dt.
    \end{aligned}
\end{align}
We also have that
\begin{align*}
\textbf{I}_5=\begin{aligned}
\mathbb{E}\iint_\Sigma\theta^2\widetilde{A}\widehat{Z}^2d\sigma dt=\mathbb{E}\iint_\Sigma\theta^2\left[-\frac{2}{\beta_0}M\lambda\mu\varphi\partial_\nu^\mathcal{A}\psi+\lambda T\varphi^2O(e^{2\mu\|\psi\|_\infty})\right]\widehat{Z}^2d\sigma dt,
    \end{aligned}
\end{align*}
which leads to
\begin{align}\label{ineq7sec}
    \begin{aligned}
\textbf{I}_5\leq C\lambda\mu T^2\mathbb{E}\iint_\Sigma\theta^2\varphi^2\widehat{Z}^2 d\sigma dt+C\lambda Te^{2\mu\|\psi\|_\infty}\mathbb{E}\iint_\Sigma\theta^2\varphi^2\widehat{Z}^2 d\sigma dt.
    \end{aligned}
\end{align}
Taking a large $\lambda\geq C(e^{2\mu\|\psi\|_\infty}T+T^2)$, the inequality \eqref{ineq7sec} implies hat
\begin{align}\label{ineq7}
    \begin{aligned}
\textbf{I}_5\leq C\lambda^2\mu \mathbb{E}\iint_\Sigma\theta^2\varphi^2\widehat{Z}^2 d\sigma dt.
    \end{aligned}
\end{align}
By integration by parts, it is easy to see that
\begin{align*}
\textbf{I}_6=-2\mathbb{E}\iint_\Sigma h_\Gamma\mathcal{A}_\Gamma\nabla_\Gamma h_\Gamma\cdot\nabla_\Gamma\Phi d\sigma dt&=\mathbb{E}\iint_\Sigma h_\Gamma^2\nabla_\Gamma\cdot(\mathcal{A}_\Gamma\nabla_\Gamma\Phi) d\sigma dt,
\end{align*}
which provides that
\begin{align}\label{ineq2}
\textbf{I}_6\leq C\lambda\mu\mathbb{E}\iint_\Sigma\theta^2\varphi z_\Gamma^2 dxdt.
\end{align}
We also have that
\begin{align}\label{ineq2sec}
\textbf{I}_7=-\mathbb{E}\iint_\Sigma \theta^2(\mathcal{A}_\Gamma)_t\nabla_\Gamma z_\Gamma\cdot\nabla_\Gamma z_\Gamma d\sigma dt\leq C\mathbb{E}\iint_\Sigma\theta^2|\nabla_\Gamma z_\Gamma|^2 d\sigma dt.
\end{align}
It is not difficult to see that
\begin{align}\label{ineq28}
\textbf{I}_8=2\mathbb{E}\iint_\Sigma \theta^2\partial_\nu^\mathcal{A}z\nabla_\Gamma\cdot(\mathcal{A}_\Gamma\nabla_\Gamma z_\Gamma) d\sigma dt\leq \mathbb{E}\iint_\Sigma \theta^2|\partial_\nu^\mathcal{A}z|^2 d\sigma dt+\mathbb{E}\iint_\Sigma \theta^2|\nabla_\Gamma z_\Gamma|^2 d\sigma dt.
\end{align}
On the other hand, we have that
\begin{align*}
\textbf{I}_9=2\mathbb{E}\iint_\Sigma \theta^2\widetilde{A}z_\Gamma\partial_\nu^\mathcal{A}z d\sigma dt&=2\mathbb{E}\iint_\Sigma \theta^2\left[-\frac{2}{\beta_0}M\lambda\mu\varphi\partial_\nu^\mathcal{A}\psi+\lambda T\varphi^2O(e^{2\mu\|\psi\|_\infty})\right]z_\Gamma\partial_\nu^\mathcal{A}z d\sigma dt.
\end{align*}
It follows that
\begin{align*}
\textbf{I}_9\leq C\lambda\mu\mathbb{E}\iint_\Sigma \theta^2\varphi|z_\Gamma||\partial_\nu^\mathcal{A}z|d\sigma dt+C\lambda Te^{2\mu\|\psi\|_\infty}\mathbb{E}\iint_\Sigma \theta^2\varphi^2|z_\Gamma||\partial_\nu^\mathcal{A}z|d\sigma dt,
\end{align*}
which implies that for large $\lambda\geq C(e^{2\mu\|\psi\|_\infty}T+T^2)$,
\begin{align}\label{ineqfisec}
\textbf{I}_9\leq C\lambda^2\mu\mathbb{E}\iint_\Sigma\theta^2\varphi^2|z_\Gamma||\partial_\nu^\mathcal{A}z|d\sigma dt.
\end{align}
Applying Young's inequality for \eqref{ineqfisec}, we derive that
\begin{align}\label{ineqfi}
\textbf{I}_9\leq C\lambda^3\mu^2\mathbb{E}\iint_\Sigma \theta^2\varphi^3 z_\Gamma^2 d\sigma dt+C\lambda\mathbb{E}\iint_\Sigma \theta^2\varphi |\partial_\nu^\mathcal{A}z|^2 d\sigma dt.
\end{align}
Combining \eqref{5.121301seces1}, \eqref{ineq1}, \eqref{ineq2}, \eqref{ineq2sec}, \eqref{ineq3}, \eqref{ineq4}, \eqref{ineq5}, \eqref{ineq6}, \eqref{ineq7}, \eqref{ineq8}, \eqref{ineq28}, \eqref{ineqfi}, and taking a large enough $\mu\geq C$ and $\lambda\geq C(e^{2\mu\|\psi\|_\infty}T+T^2)$, we conclude that
\begin{align}\label{5.121301seces1sec}
\begin{aligned}
&2\mathbb{E}\iint_Q\sum_{j,k=1}^N c^{jk}h_{x_j}h_{x_k} dxdt+\mathbb{E}\iint_Q Bh^2 dxdt\\
&+\lambda^3\mu^3\mathbb{E}\iint_\Sigma\theta^2\varphi^3 z_\Gamma^2d\sigma dt+\lambda\mu\mathbb{E}\iint_\Sigma\theta^2\varphi|\nabla_\Gamma z_\Gamma|^2 d\sigma dt\\
&\leq C\mathbb{E}\iint_Q\theta^2 F_1^2 dxdt+C\mathbb{E}\iint_\Sigma\theta^2 F_2^2d\sigma dt\\
&\quad\;+C\mathbb{E}\iint_Q\theta^2AZ^2 dxdt+C\lambda^2\mu\mathbb{E}\iint_\Sigma\theta^2\varphi^2\widehat{Z}^2d\sigma dt.
\end{aligned}
\end{align}
\textbf{Step 3. Estimating the non-boundary terms.}\\
From Lemma \ref{thm5.2}, we have that for large $\mu\geq C$ and $\lambda\geq CT^2$,
\begin{align*}
C\mathbb{E}\iint_Q\theta^2AZ^2 dxdt=C\mathbb{E}\iint_Q\theta^2\left[\lambda^2 \mu^2 \varphi^2 \sum_{j,k=1}^N a^{jk} \psi_{x_j} \psi_{x_k} + \lambda \varphi O(\mu^2) + \lambda T \varphi^2 O(e^{2\mu \|\psi\|_\infty})\right]Z^2 dxdt.
\end{align*}
Taking a large $\lambda\geq C(e^{2\mu\|\psi\|_\infty}T+T^2)$, it follows that
\begin{align}\label{ineqj1}
C\mathbb{E}\iint_Q\theta^2AZ^2 dxdt\leq C\lambda^2 \mu^2\mathbb{E}\iint_Q\theta^2 \varphi^2 Z^2 dxdt.
\end{align}
Since $\displaystyle\min_{x \in \overline{G \setminus G_1}} | \nabla \psi(x) |^2>0$, by Lemma \ref{thm5.2}, we obtain for a large \( \mu \geq C \) and \( \lambda \geq C T^2 \),
\begin{align}\label{5.212201}
\begin{aligned}
&\, 2\mathbb{E} \int_0^T \int_{G \setminus G_1} \sum_{j,k=1}^N c^{jk} h_{x_j} h_{x_k} \, dx  dt + \mathbb{E} \int_0^T \int_{G \setminus G_1} B h^2 \, dx  dt \\
&\geq 2\beta_0^2 \mathbb{E} \int_0^T \int_{G \setminus G_1} \bigg[ \left( \lambda \mu^2 \varphi \min_{x \in \overline{G \setminus G_1}} | \nabla \psi |^2 + \lambda \varphi O(\mu) \right) | \nabla h |^2 \\
&\quad + \left( \lambda^3 \mu^4 \varphi^3 \min_{x \in \overline{G \setminus G_1}} | \nabla \psi |^4 + \lambda^3 \varphi^3 O(\mu^4) + \lambda^2 T \varphi^3 O(\mu^2 e^{2 \mu \|\psi\|_\infty}) \right) h^2 \bigg] \, dx  dt.
\end{aligned}
\end{align}
From \eqref{5.212201}, we conclude for a large \( \mu \geq C \) and \( \lambda \geq C ( e^{2 \mu \|\psi\|_\infty} T + T^2 ) \),
\begin{align}
\begin{aligned}
&\, 2 \mathbb{E} \int_0^T \int_{G \setminus G_1} \sum_{j,k=1}^N c^{jk} h_{x_j} h_{x_k} \, dx  dt + \mathbb{E} \int_0^T \int_{G \setminus G_1} B h^2 \, dx  dt \\
\label{5.232404} &\geq C \lambda \mu^2 \mathbb{E} \int_0^T \int_{G \setminus G_1} \varphi \left( | \nabla h |^2 + \lambda^2 \mu^2 \varphi^2 h^2 \right) \, dx  dt.
\end{aligned}
\end{align}
Notice that
\begin{equation}\label{5.242505}
\frac{1}{C} \theta^2 \left( | \nabla z |^2 + \lambda^2 \mu^2 \varphi^2 z^2 \right) \leq | \nabla h |^2 + \lambda^2 \mu^2 \varphi^2 h^2 \leq C \theta^2 \left( | \nabla z |^2 + \lambda^2 \mu^2 \varphi^2 z^2 \right).
\end{equation}
Combining \eqref{5.121301seces1sec}, \eqref{ineqj1}, \eqref{5.232404} and \eqref{5.242505}, we conclude for large enough \( \mu \geq C \) and \( \lambda \geq C \left( e^{2 \mu \|\psi\|_\infty} T + T^2 \right) \),
\begin{align}\label{5.212101}
\begin{aligned}
&\, \lambda \mu^2 \mathbb{E} \iint_Q \theta^2 \varphi \left( | \nabla z |^2 + \lambda^2 \mu^2 \varphi^2 z^2 \right)  dx  dt \\
&+\lambda^3\mu^3\mathbb{E}\iint_\Sigma\theta^2\varphi^3 z_\Gamma^2d\sigma dt+\lambda\mu\mathbb{E}\iint_\Sigma\theta^2\varphi|\nabla_\Gamma z_\Gamma|^2 d\sigma dt\\
&
= \lambda \mu^2 \mathbb{E} \int_0^T \left( \int_{G \setminus G_1} + \int_{G_1} \right) \theta^2 \varphi \left( | \nabla z |^2 + \lambda^2 \mu^2 \varphi^2 z^2 \right)  dx  dt \\
&\quad\;+\lambda^3\mu^3\mathbb{E}\iint_\Sigma\theta^2\varphi^3 z_\Gamma^2d\sigma dt+\lambda\mu\mathbb{E}\iint_\Sigma\theta^2\varphi|\nabla_\Gamma z_\Gamma|^2 d\sigma dt\\
&\leq C\mathbb{E}\iint_Q\theta^2 F_1^2 dxdt+C\mathbb{E}\iint_\Sigma\theta^2 F_2^2d\sigma dt\\
&\quad\;+C\lambda^2 \mu^2\mathbb{E}\iint_Q\theta^2 \varphi^2 Z^2 dxdt+\lambda^2\mu\mathbb{E}\iint_\Sigma\theta^2\varphi^2\widehat{Z}^2d\sigma dt\\
&\quad\; + \lambda \mu^2 \mathbb{E} \int_0^T \int_{G_1} \theta^2 \varphi \left( | \nabla z |^2 + \lambda^2 \mu^2 \varphi^2 z^2 \right)  dx  dt.
\end{aligned}
\end{align}
Finally, to eliminate the gradient term on the right-hand side of \eqref{5.212101}, we proceed exactly as in the proof of \cite[Theorem 9.39]{lu2021mathematical} by choosing a cut-off function \( \zeta \in C^\infty_0(\mathcal{B}; [0, 1]) \) with \( \zeta \equiv 1 \) in \( G_1 \), then computing \( d(\theta^2 \varphi \zeta^2 z^2) \) and integrating the equality over \( Q \), and taking the expectation on both sides. We arrive at
\begin{equation}\label{5.222201}
\mathbb{E} \int_0^T \int_{G_1} \theta^2 \varphi | \nabla z |^2 \, dx  dt \leq C \mathbb{E} \int_0^T\int_\mathcal{B} \theta^2 \left( \frac{1}{\lambda^2 \mu^2} F_1^2 + \lambda^2 \mu^2 \varphi^3 z^2 \right)  dxdt.
\end{equation}
Combining \eqref{5.212101} and \eqref{5.222201}, we deduce the desired Carleman inequality \eqref{3.5}. This concludes the proof of Lemma \ref{lm2.2}.
\end{proof}

\section{Main Carleman Estimate: Proof of Theorem~\ref{thm1.1}}\label{sec2}
This section is devoted to proving the main Carleman estimate in Theorem \ref{thm1.1}. Based on the result of Lemma \ref{lm2.2}, we now proceed to prove Theorem \ref{thm1.1}. First, we establish a controllability result for the following controlled forward stochastic parabolic equation with dynamic boundary conditions and source terms:
\begin{equation}\label{3.6}
\begin{cases}
\begin{array}{ll}
dy - \nabla\cdot(\mathcal{A}\nabla y) \,dt = (\lambda^3\mu^4\theta^2\varphi^3z + \mathbbm{1}_\mathcal{B}u) \,dt + v_1 \,dW(t) & \textnormal{in} \, Q, \\
dy_\Gamma - \nabla_\Gamma\cdot(\mathcal{A}_\Gamma\nabla_\Gamma y_\Gamma) \,dt + \partial_\nu^\mathcal{A} y \, dt = \lambda^3\mu^3\theta^2\varphi^3z_\Gamma \,dt + v_2 \, dW(t) & \textnormal{on} \, \Sigma, \\
y_\Gamma = y|_\Gamma & \textnormal{on} \, \Sigma, \\
(y, y_\Gamma)|_{t=0} = (0, 0) & \textnormal{in} \, G \times \Gamma,
\end{array}
\end{cases}
\end{equation}
where \((u, v_1, v_2) \in L^2_\mathcal{F}(0,T; L^2(\mathcal{B})) \times L^2_\mathcal{F}(0,T; L^2(G)) \times L^2_\mathcal{F}(0,T; L^2(\Gamma))\) are the controls of the system, and \((z, z_\Gamma; Z, \widehat{Z})\) represents the state of equation \eqref{1.1gen}. The functions \(\theta\) and \(\varphi\) are the weight functions defined in \eqref{3.2}. We now state the following controllability result.
\begin{prop}\label{prop3.1}
There exists a triple of controls \((\widehat{u}, \widehat{v}_1, \widehat{v}_2) \in L^2_\mathcal{F}(0,T; L^2(\mathcal{B})) \times L^2_\mathcal{F}(0,T; L^2(G)) \times L^2_\mathcal{F}(0,T; L^2(\Gamma))\) such that the corresponding solution \((\widehat{y}, \widehat{y}_\Gamma)\) of \eqref{3.6} satisfies
$$ (\widehat{y}(T, \cdot), \widehat{y}_\Gamma(T, \cdot)) = (0, 0) \quad \textnormal{in} \,\; G \times \Gamma, \,\; \textnormal{a.s.} $$ 
Moreover, there exists a constant \(C > 0\), depending only on \(G\), \(\mathcal{B}\), $\beta_0$, $M$, and $M_\Gamma$, such that the following inequality holds:
\begin{align}\label{3.7}
\begin{aligned}
&\, \lambda^{-3}\mu^{-4} \mathbb{E} \iint_Q \theta^{-2} \varphi^{-3} |\widehat{u}|^2 \, dx  dt + \lambda^{-2} \mu^{-2} \mathbb{E} \iint_Q \theta^{-2} \varphi^{-2} |\widehat{v}_1|^2 \, dx  dt \\
&+ \lambda^{-2} \mu^{-2} \mathbb{E} \iint_\Sigma \theta^{-2} \varphi^{-2} |\widehat{v}_2|^2 \, d\sigma  dt + \mathbb{E} \iint_Q \theta^{-2} |\widehat{y}|^2 \, dx  dt + \mathbb{E} \iint_\Sigma \theta^{-2} |\widehat{y}_\Gamma|^2 \, d\sigma  dt \\
&+ \lambda^{-2} \mu^{-2} \mathbb{E} \iint_Q \theta^{-2} \varphi^{-2} |\nabla \widehat{y}|^2 \, dx  dt + \lambda^{-2} \mu^{-2} \mathbb{E} \iint_\Sigma \theta^{-2} \varphi^{-2} |\nabla_\Gamma \widehat{y}_\Gamma|^2 \, d\sigma  dt \\
& \leq C \bigg[ \lambda^3 \mu^4 \mathbb{E} \iint_Q \theta^2 \varphi^3 |z|^2 \, dx  dt + \lambda^3 \mu^3 \mathbb{E} \iint_\Sigma \theta^2 \varphi^3 |z_\Gamma|^2 \, d\sigma  dt \bigg],
\end{aligned}
\end{align}
for large \(\mu \geq C\) and \(\lambda \geq C(e^{2 \mu \|\psi\|_\infty} T + T^2)\).
\end{prop}
\begin{proof} For clarity, we present the proof in three steps.\\
\textbf{Step 1.} Fix $\varepsilon>0$ and define the functions:
\begin{align}\label{functalphaepsilon}
\theta_\varepsilon=e^{\lambda\alpha_\varepsilon},\qquad\alpha_\varepsilon\equiv\alpha_\varepsilon(t,x)=(e^{\mu\psi(x)}-e^{2\mu\|\psi\|_\infty})((t+\varepsilon)(T-t+\varepsilon))^{-1}.
\end{align}
It is easy to see that $\theta_\varepsilon\geq\theta$. Consider the following optimal control problem
\begin{equation}\label{3.08}
    \inf_{(u,v_1,v_2)\in\mathcal{U}}\,J_\varepsilon(u,v_1,v_2),
\end{equation}
where 
\begin{align*}
J_\varepsilon(u,v_1,v_2)=&\,\frac{1}{2}\mathbb{E}\int_0^T\int_{\mathcal{B}} \lambda^{-3}\mu^{-4}\theta^{-2}\varphi^{-3}|u|^2dxdt+\frac{1}{2}\mathbb{E}\iint_Q \lambda^{-2}\mu^{-2}\theta^{-2}\varphi^{-2}|v_1|^2dxdt\\
 &+\frac{1}{2}\mathbb{E}\iint_\Sigma \lambda^{-2}\mu^{-2}\theta^{-2}\varphi^{-2}|v_2|^2d\sigma dt+\frac{1}{2}\mathbb{E}\iint_Q \theta^{-2}_\varepsilon |y|^2dxdt+\frac{1}{2}\mathbb{E}\iint_\Sigma \theta^{-2}_\varepsilon |y_\Gamma|^2d\sigma dt\\
 &+\frac{1}{2\varepsilon}\mathbb{E}\int_G \vert y(T)\vert^2dx+\frac{1}{2\varepsilon}\mathbb{E}\int_\Gamma \vert y_\Gamma(T)\vert^2d\sigma,
\end{align*}
and
\begin{align*}
\mathcal{U}=\bigg\{(&u,v_1,v_2)\in L^2_\mathcal{F}(0,T;L^2(\mathcal{B}))\times L^2_\mathcal{F}(0,T;L^2(G))\times L^2_\mathcal{F}(0,T;L^2(\Gamma)):\\
&\mathbb{E}\int_0^T\int_{\mathcal{B}} \theta^{-2}\varphi^{-3}|u|^2dxdt<\infty\,,\quad\mathbb{E}\iint_Q \theta^{-2}\varphi^{-2}|v_1|^2dxdt<\infty\,,\quad\mathbb{E}\iint_\Sigma \theta^{-2}\varphi^{-2}|v_2|^2d\sigma dt<\infty\bigg\}.
\end{align*}
It is easy to see that $J_\varepsilon$ is well-defined, continuous, strictly convex, and coercive. Therefore, the problem \eqref{3.08} admits a unique optimal solution $(u_\varepsilon,v_{\varepsilon,1},v_{\varepsilon,2})\in\mathcal{U}$. Moreover, by classical arguments: Euler-Lagrange equation and optimality system (see, e.g., \cite{imanuvilov2003carleman,lions1972some}), the optimal solution $(u_\varepsilon,v_{\varepsilon,1},v_{\varepsilon,2})$ can be characterized by:
\begin{equation}\label{3.09}
 u_\varepsilon=-\mathbbm{1}_{\mathcal{B}}\lambda^3\mu^4\theta^2\varphi^3r_\varepsilon\,\,,\quad\quad v_{\varepsilon,1}=-\lambda^2\mu^2\theta^2\varphi^2R_{\varepsilon,1}\,\,,\quad\quad v_{\varepsilon,2}=-\lambda^2\mu^2\theta^2\varphi^2 R_{\varepsilon,2},
 \end{equation}
where $(r_\varepsilon,r_{\varepsilon,\Gamma};R_{\varepsilon,1},R_{\varepsilon,2})$ is the solution of the backward equation
\begin{equation}\label{3.010}
\begin{cases}
\begin{array}{ll}
dr_\varepsilon + \nabla\cdot(\mathcal{A}\nabla r_\varepsilon) \,dt = -\theta^{-2}_\varepsilon y_\varepsilon \,dt + R_{\varepsilon,1} \,dW(t) &\textnormal{in}\,\,Q,\\
dr_{\varepsilon,\Gamma} + \nabla_\Gamma\cdot(\mathcal{A}_\Gamma\nabla_\Gamma r_{\varepsilon,\Gamma}) \,dt -\partial_\nu^\mathcal{A} r_\varepsilon dt = -\theta^{-2}_\varepsilon y_{\varepsilon,\Gamma} \,dt + R_{\varepsilon,2} \,dW(t) &\textnormal{on}\,\,\Sigma,\\
r_{\varepsilon,\Gamma}(t,x)=r_\varepsilon\vert_\Gamma(t,x)  &\textnormal{on}\,\,\Sigma,\\
\big(r_\varepsilon,r_{\varepsilon,\Gamma}\big)\vert_{{t=T}}=(\frac{1}{\varepsilon}y_\varepsilon(T),\frac{1}{\varepsilon}y_{\varepsilon,\Gamma}(T)) &\textnormal{in}\,\,G\times\Gamma,
			\end{array}
		\end{cases}
\end{equation}
with $(y_\varepsilon,y_{\varepsilon,\Gamma})$ is the solution of \eqref{3.6} associated to the controls $(u_\varepsilon,v_{\varepsilon,1},v_{\varepsilon,2})$.\\
\textbf{Step 2.} Differentiating $d\langle(y_\varepsilon,y_{\varepsilon,\Gamma}),(r_\varepsilon,r_{\varepsilon,\Gamma})\rangle_{\mathbb{L}^2}$ by Itô's formula, integrating the equality on $(0,T)$ and recalling \eqref{3.09}, we obtain that
\begin{align}\label{3.011}
\begin{aligned}
&\,\frac{1}{\varepsilon}\mathbb{E}\int_G \vert y_\varepsilon(T)\vert^2dx+\frac{1}{\varepsilon}\mathbb{E}\int_\Gamma \vert y_{\varepsilon,\Gamma}(T)\vert^2d\sigma+\mathbb{E}\iint_{Q}\theta_\varepsilon^{-2}y_\varepsilon^2dxdt+\mathbb{E}\iint_{\Sigma}\theta_\varepsilon^{-2}y_{\varepsilon,\Gamma}^2d\sigma dt\\
&+\lambda^3\mu^4\mathbb{E}\int_0^T\int_{\mathcal{B}}\theta^2\varphi^3r_\varepsilon^2dxdt+\lambda^2\mu^2\mathbb{E}\iint_{Q}\theta^2\varphi^2R_{\varepsilon,1}^2dxdt+\lambda^2\mu^2\mathbb{E}\iint_{\Sigma}\theta^2\varphi^2 R_{\varepsilon,2}^2d\sigma dt\\
&=\lambda^3\mu^4\mathbb{E}\iint_{Q}\theta^2\varphi^3 zr_\varepsilon dxdt+\lambda^3\mu^3\mathbb{E}\iint_{\Sigma}\theta^2\varphi^3 z_{\Gamma}r_{\varepsilon,\Gamma} d\sigma dt.
\end{aligned}
\end{align}
By Young's inequality, we get from \eqref{3.011} that for all $\rho>0$,
\begin{align}\label{3.012}
\begin{aligned}
&\,\frac{1}{\varepsilon}\mathbb{E}\int_G \vert y_\varepsilon(T)\vert^2dx+\frac{1}{\varepsilon}\mathbb{E}\int_\Gamma \vert y_{\varepsilon,\Gamma}(T)\vert^2d\sigma+\mathbb{E}\iint_{Q}\theta_\varepsilon^{-2}y_\varepsilon^2dxdt+\mathbb{E}\iint_{\Sigma}\theta_\varepsilon^{-2}y_{\varepsilon,\Gamma}^2d\sigma dt\\
&+\lambda^3\mu^4\mathbb{E}\int_0^T\int_{\mathcal{B}}\theta^2\varphi^3r_\varepsilon^2dxdt+\lambda^2\mu^2\mathbb{E}\iint_{Q}\theta^2\varphi^2R_{\varepsilon,1}^2dxdt+\lambda^2\mu^2\mathbb{E}\iint_{\Sigma}\theta^2\varphi^2 R_{\varepsilon,2}^2d\sigma dt\\
&\leq \rho\bigg[\lambda^3\mu^4\mathbb{E}\iint_{Q}\theta^2\varphi^3 r_\varepsilon^2 dxdt+\lambda^3\mu^3\mathbb{E}\iint_{\Sigma}\theta^2\varphi^3 r_{\varepsilon,\Gamma}^2 d\sigma\bigg]\\
& \;\;\;+\frac{1}{4\rho}\bigg[\lambda^3\mu^4\mathbb{E}\iint_{Q}\theta^2\varphi^3 z^2 dxdt+\lambda^3\mu^3\mathbb{E}\iint_{\Sigma}\theta^2\varphi^3 z_{\Gamma}^2 d\sigma dt\bigg].
\end{aligned}
\end{align}
Using the Carleman estimate \eqref{3.5} for solutions of \eqref{3.010} and noting that $\theta\theta_\varepsilon^{-1}\leq1$, then there exists a constant $C=C(G,\mathcal{B},\beta_0, M,M_\Gamma)>0$ such that for all $\mu\geq C$ and $\lambda\geq C(e^{2\mu\|\psi\|_\infty}T+T^2)$, the inequality \eqref{3.012} implies that
\begin{align}\label{3.013}
\begin{aligned}
&\,\frac{1}{\varepsilon}\mathbb{E}\int_G \vert y_\varepsilon(T)\vert^2dx+\frac{1}{\varepsilon}\mathbb{E}\int_\Gamma \vert y_{\varepsilon,\Gamma}(T)\vert^2d\sigma+\mathbb{E}\iint_{Q}\theta_\varepsilon^{-2}y_\varepsilon^2dxdt+\mathbb{E}\iint_{\Sigma}\theta_\varepsilon^{-2}y_{\varepsilon,\Gamma}^2d\sigma dt\\
&+\lambda^3\mu^4\mathbb{E}\int_0^T\int_{\mathcal{B}}\theta^2\varphi^3r_\varepsilon^2dxdt+\lambda^2\mu^2\mathbb{E}\iint_{Q}\theta^2\varphi^2R_{\varepsilon,1}^2dxdt+\lambda^2\mu^2\mathbb{E}\iint_{\Sigma}\theta^2\varphi^2 R_{\varepsilon,2}^2d\sigma dt\\
&\leq C\rho\bigg[\lambda^3\mu^4\mathbb{E}\int_0^T\int_{\mathcal{B}}\theta^2\varphi^3 r_\varepsilon^2 dxdt+\mathbb{E}\iint_{Q}\theta_\varepsilon^{-2}y_\varepsilon^2dxdt+\mathbb{E}\iint_{\Sigma}\theta_\varepsilon^{-2}y_{\varepsilon,\Gamma}^2d\sigma dt\\
&\qquad\;\quad+\lambda^2\mu^2\mathbb{E}\iint_{Q}\theta^2\varphi^2 R_{\varepsilon,1}^2 dxdt+\lambda^2\mu^2\mathbb{E}\iint_{\Sigma}\theta^2\varphi^2 R_{\varepsilon,2}^2 d\sigma dt\bigg]\\
& \;\;\;+\frac{1}{4\rho}\bigg[\lambda^3\mu^4\mathbb{E}\iint_{Q}\theta^2\varphi^3 z^2 dxdt+\lambda^3\mu^3\mathbb{E}\iint_{\Sigma}\theta^2\varphi^3 z_{\Gamma}^2 d\sigma dt\bigg].
\end{aligned}
\end{align}
Choosing a small enough $\rho>0$ in \eqref{3.013}, we obtain that
\begin{align*}
\begin{aligned}
&\,\frac{1}{\varepsilon}\mathbb{E}\int_G \vert y_\varepsilon(T)\vert^2dx+\frac{1}{\varepsilon}\mathbb{E}\int_\Gamma \vert y_{\varepsilon,\Gamma}(T)\vert^2d\sigma+\mathbb{E}\iint_{Q}\theta_\varepsilon^{-2}y_\varepsilon^2dxdt+\mathbb{E}\iint_{\Sigma}\theta_\varepsilon^{-2}y_{\varepsilon,\Gamma}^2d\sigma dt\\
&+\lambda^3\mu^4\mathbb{E}\int_0^T\int_{\mathcal{B}}\theta^2\varphi^3r_\varepsilon^2dxdt+\lambda^2\mu^2\mathbb{E}\iint_{Q}\theta^2\varphi^2R_{\varepsilon,1}^2dxdt+\lambda^2\mu^2\mathbb{E}\iint_{\Sigma}\theta^2\varphi^2 R_{\varepsilon,2}^2d\sigma dt\\
&\leq C\bigg[\lambda^3\mu^4\mathbb{E}\iint_{Q}\theta^2\varphi^3 z^2 dxdt+\lambda^3\mu^3\mathbb{E}\iint_{\Sigma}\theta^2\varphi^3 z_{\Gamma}^2 d\sigma dt\bigg].
\end{aligned}
\end{align*}
Recalling \eqref{3.09}, it follows that
\begin{align}\label{3.015}
\begin{aligned}
&\,\frac{1}{\varepsilon}\mathbb{E}\int_G \vert y_\varepsilon(T)\vert^2dx+\frac{1}{\varepsilon}\mathbb{E}\int_\Gamma \vert y_{\varepsilon,\Gamma}(T)\vert^2d\sigma+\mathbb{E}\iint_{Q}\theta_\varepsilon^{-2}y_\varepsilon^2dxdt+\mathbb{E}\iint_{\Sigma}\theta_\varepsilon^{-2}y_{\varepsilon,\Gamma}^2d\sigma dt\\
&+\lambda^{-3}\mu^{-4}\mathbb{E}\iint_{Q}\theta^{-2}\varphi^{-3}u_\varepsilon^{2}dxdt+\lambda^{-2}\mu^{-2}\mathbb{E}\iint_{Q}\theta^{-2}\varphi^{-2}v_{\varepsilon,1}^2dxdt\\
&+\lambda^{-2}\mu^{-2}\mathbb{E}\iint_{\Sigma}\theta^{-2}\varphi^{-2} v_{\varepsilon,2}^2d\sigma dt\\
&\leq C\bigg[\lambda^3\mu^4\mathbb{E}\iint_{Q}\theta^2\varphi^3 z^2 dxdt+\lambda^3\mu^3\mathbb{E}\iint_{\Sigma}\theta^2\varphi^3 z_{\Gamma}^2 d\sigma dt\bigg].
\end{aligned}
\end{align}
On the other hand, by computing $d(\theta^{-2}_\varepsilon\varphi^{-2}y_\varepsilon^2)$ with Itô's formula, integrating the equality on $Q$ and taking the expectation on both sides, we find that
\begin{align}\label{3.016}
\begin{aligned}
&\,2 \mathbb{E}\iint_Q \theta^{-2}_\varepsilon\varphi^{-2}\mathcal{A}\nabla y_\varepsilon\cdot\nabla y_\varepsilon dxdt\\
&=-2 \mathbb{E}\iint_Q y_\varepsilon\mathcal{A}\nabla y_\varepsilon\cdot\nabla(\theta_\varepsilon^{-2}\varphi^{-2}) dxdt+2\mathbb{E}\int_0^T \langle \partial_\nu^\mathcal{A} y_\varepsilon,\theta_\varepsilon^{-2}\varphi^{-2}y_{\varepsilon,\Gamma}\rangle_{H^{-1/2}(\Gamma),H^{1/2}(\Gamma)}dt
\\&\quad\,+\mathbb{E}\iint_Q (\theta^{-2}_\varepsilon\varphi^{-2})_ty_\varepsilon^2dxdt+2\lambda^3\mu^4\mathbb{E}\iint_Q \theta^{-2}_\varepsilon\theta^2\varphi y_\varepsilon z dxdt\\&\quad\,+2\mathbb{E}\int_0^T\int_{\mathcal{B}}\theta^{-2}_\varepsilon\varphi^{-2}y_\varepsilon u_\varepsilon dxdt+\mathbb{E}\iint_Q \theta^{-2}_\varepsilon\varphi^{-2}v_{\varepsilon,1}^2dxdt.
\end{aligned}
\end{align}
Similarly to \eqref{3.016}, by computing $d(\theta^{-2}_\varepsilon\varphi^{-2}y_{\varepsilon,\Gamma}^2)$, we get
\begin{align}\label{3.017}
\begin{aligned}
&\,2 \mathbb{E}\iint_\Sigma \theta^{-2}_\varepsilon\varphi^{-2}\mathcal{A}_\Gamma\nabla_\Gamma y_{\varepsilon,\Gamma}\cdot\nabla_\Gamma y_{\varepsilon,\Gamma} d\sigma dt\\
&=-2\mathbb{E}\int_0^T \langle \partial_\nu^\mathcal{A} y_\varepsilon,\theta_\varepsilon^{-2}\varphi^{-2}y_{\varepsilon,\Gamma}\rangle_{H^{-1/2}(\Gamma),H^{1/2}(\Gamma)}dt
\\&\quad\,+\mathbb{E}\iint_\Sigma (\theta^{-2}_\varepsilon\varphi^{-2})_ty_{\varepsilon,\Gamma}^2d\sigma dt+2\lambda^3\mu^3\mathbb{E}\iint_\Sigma \theta^{-2}_\varepsilon\theta^2\varphi y_{\varepsilon,\Gamma} z_\Gamma d\sigma dt\\&\quad\,+\mathbb{E}\iint_\Sigma \theta^{-2}_\varepsilon\varphi^{-2}v_{\varepsilon,2}^2d\sigma.
\end{aligned}
\end{align}
By using \eqref{3.3} and \eqref{functalphaepsilon}, it is easy to check that for a large enough \( \lambda \geq C T^2 \), we have for any $(t,x)\in(0,T)\times\overline{G}$,
\begin{equation}\label{2.0140}
|(\theta^{-2}_\varepsilon \varphi^{-2})_t| \leq C T \lambda e^{2\mu \|\psi\|_\infty} \theta_\varepsilon^{-2}, \qquad
|\nabla (\theta^{-2}_\varepsilon \varphi^{-2})| \leq C \lambda \mu \theta_\varepsilon^{-2} \varphi^{-1}.
\end{equation}
Recalling \eqref{asummAandAgama} and combining \eqref{3.016}, \eqref{3.017} and \eqref{2.0140}, we conclude that for a large \( \lambda \geq C T^2 \), we obtain
\begin{align}\label{3.019}
\begin{aligned}
&\,2 \beta_0\mathbb{E}\iint_Q \theta^{-2}_\varepsilon\varphi^{-2}\vert\nabla y_\varepsilon\vert^2 dxdt+2 \beta_0\mathbb{E}\iint_\Sigma \theta^{-2}_\varepsilon\varphi^{-2}\vert\nabla_\Gamma y_{\varepsilon,\Gamma}\vert^2 d\sigma dt\\
&\leq C\lambda\mu\mathbb{E}\iint_Q \theta_\varepsilon^{-2}\varphi^{-1} \vert y_\varepsilon\vert \vert\nabla y_\varepsilon\vert dxdt+CT\lambda e^{2\mu\|\psi\|_\infty}\mathbb{E}\iint_Q \theta^{-2}_\varepsilon y_\varepsilon^2 dxdt+2\lambda^3\mu^4\mathbb{E}\iint_Q \theta^{-1}_\varepsilon\theta\varphi \vert y_\varepsilon\vert  \vert z\vert  dxdt\\&\quad\,+2\mathbb{E}\int_0^T\int_{\mathcal{B}}\theta^{-1}_\varepsilon\theta^{-1}\varphi^{-2}\vert y_\varepsilon\vert  \vert u_\varepsilon\vert  dxdt+\mathbb{E}\iint_Q \theta^{-2}_\varepsilon\varphi^{-2}v_{\varepsilon,1}^2dxdt+CT\lambda e^{2\mu\|\psi\|_\infty}\mathbb{E}\iint_\Sigma \theta^{-2}_\varepsilon y_{\varepsilon,\Gamma}^2d\sigma dt\\
&\quad\,+2\lambda^3\mu^3\mathbb{E}\iint_\Sigma \theta^{-1}_\varepsilon\theta\varphi \vert y_{\varepsilon,\Gamma}\vert  \vert z_\Gamma\vert  d\sigma dt+\mathbb{E}\iint_\Sigma \theta^{-2}_\varepsilon\varphi^{-2}v_{\varepsilon,2}^2d\sigma dt.
    \end{aligned}
\end{align}
Applying Young's inequality on the right-hand side of \eqref{3.019}, we deduce that for any $\rho>0$,
\begin{align}\label{3.020}
\begin{aligned}
&\,2 \beta_0\mathbb{E}\iint_Q \theta^{-2}_\varepsilon\varphi^{-2}\vert\nabla y_\varepsilon\vert^2 dxdt+2 \beta_0\mathbb{E}\iint_\Sigma \theta^{-2}_\varepsilon\varphi^{-2}\vert\nabla_\Gamma y_{\varepsilon,\Gamma}\vert^2 d\sigma dt\\
&\leq \rho \mathbb{E}\iint_Q \theta_\varepsilon^{-2}\varphi^{-2} \vert\nabla y_\varepsilon\vert^2 dxdt+\frac{C}{\rho}\lambda^2\mu^2 \mathbb{E}\iint_Q \theta_\varepsilon^{-2}  y_\varepsilon^2 dxdt+CT\lambda e^{2\mu\|\psi\|_\infty}\mathbb{E}\iint_Q \theta^{-2}_\varepsilon y_\varepsilon^2 dxdt\\&\quad\,+2\lambda^3\mu^4\mathbb{E}\iint_Q \theta^{-1}_\varepsilon\theta\varphi \vert y_\varepsilon\vert  \vert z\vert  dxdt+2\mathbb{E}\int_0^T\int_{\mathcal{B}}\theta^{-1}_\varepsilon\theta^{-1}\varphi^{-2}\vert y_\varepsilon\vert  \vert u_\varepsilon\vert  dxdt+\mathbb{E}\iint_Q \theta^{-2}_\varepsilon\varphi^{-2}v_{\varepsilon,1}^2dxdt\\&\quad\,+CT\lambda e^{2\mu\|\psi\|_\infty}\mathbb{E}\iint_\Sigma \theta^{-2}_\varepsilon y_{\varepsilon,\Gamma}^2d\sigma dt+2\lambda^3\mu^3\mathbb{E}\iint_\Sigma \theta^{-1}_\varepsilon\theta\varphi \vert y_{\varepsilon,\Gamma}\vert  \vert z_\Gamma\vert  d\sigma dt+\mathbb{E}\iint_\Sigma \theta^{-2}_\varepsilon\varphi^{-2}v_{\varepsilon,2}^2d\sigma dt.
    \end{aligned}
\end{align}
Absorbing the first term on the right-hand side of \eqref{3.020} with the first term on the left-hand side by choosing a small $\rho>0$. Thus, multiplying the obtained inequality by $\lambda^{-2}\mu^{-2}$, we find that
\begin{align}\label{3.21021s}
    \begin{aligned}
&\,\lambda^{-2}\mu^{-2}\mathbb{E}\iint_Q \theta^{-2}_\varepsilon\varphi^{-2}\vert\nabla y_\varepsilon\vert^2 dxdt+\lambda^{-2}\mu^{-2}\mathbb{E}\iint_\Sigma \theta^{-2}_\varepsilon\varphi^{-2}\vert\nabla_\Gamma y_{\varepsilon,\Gamma}\vert^2 d\sigma dt\\
&\leq C\mathbb{E}\iint_Q \theta_\varepsilon^{-2}  y_\varepsilon^2 dxdt+CT\lambda^{-1}\mu^{-2} e^{2\mu\|\psi\|_\infty}\mathbb{E}\iint_Q \theta^{-2}_\varepsilon y_\varepsilon^2 dxdt+C\lambda\mu^2\mathbb{E}\iint_Q \theta^{-1}_\varepsilon\theta\varphi \vert y_\varepsilon\vert  \vert z\vert  dxdt\\&\quad\,+C\lambda^{-2}\mu^{-2}\mathbb{E}\int_0^T\int_{\mathcal{B}}\theta^{-1}_\varepsilon\theta^{-1}\varphi^{-2}\vert y_\varepsilon\vert  \vert u_\varepsilon\vert  dxdt+C\lambda^{-2}\mu^{-2}\mathbb{E}\iint_Q \theta^{-2}\varphi^{-2}v_{\varepsilon,1}^2dxdt\\&\quad\,+CT\lambda^{-1}\mu^{-2} e^{2\mu\|\psi\|_\infty}\mathbb{E}\iint_\Sigma \theta^{-2}_\varepsilon y_{\varepsilon,\Gamma}^2d\sigma dt+C\lambda\mu\mathbb{E}\iint_\Sigma \theta^{-1}_\varepsilon\theta\varphi \vert y_{\varepsilon,\Gamma}\vert  \vert z_\Gamma\vert  d\sigma dt\\&\quad\,+C\lambda^{-2}\mu^{-2}\mathbb{E}\iint_\Sigma \theta^{-2}\varphi^{-2}v_{\varepsilon,2}^2d\sigma dt.
    \end{aligned}
\end{align}
Using Young's inequality for \eqref{3.21021s}, we have that
\begin{align*}
    \begin{aligned}
&\,\lambda^{-2}\mu^{-2}\mathbb{E}\iint_Q \theta^{-2}_\varepsilon\varphi^{-2}\vert\nabla y_\varepsilon\vert^2 dxdt+\lambda^{-2}\mu^{-2}\mathbb{E}\iint_\Sigma \theta^{-2}_\varepsilon\varphi^{-2}\vert\nabla_\Gamma y_{\varepsilon,\Gamma}\vert^2 d\sigma dt\\
&\leq C\mathbb{E}\iint_Q \theta_\varepsilon^{-2}  y_\varepsilon^2 dxdt+CT\lambda^{-1}\mu^{-2} e^{2\mu\|\psi\|_\infty}\mathbb{E}\iint_Q \theta^{-2}_\varepsilon y_\varepsilon^2 dxdt+C\lambda^2\mu^4\mathbb{E}\iint_Q \theta^2\varphi^2 z^2  dxdt\\
&\quad\,+C\mathbb{E}\iint_Q \theta^{-2}_\varepsilon y_\varepsilon^2  dxdt +C\lambda^{-4}\mu^{-4}\mathbb{E}\iint_{Q} \theta^{-2}\varphi^{-4} u_\varepsilon^2  dxdt+C\mathbb{E}\iint_{Q}\theta^{-2}_\varepsilon y_\varepsilon^2  dxdt\\
&\quad\,+C\lambda^{-2}\mu^{-2}\mathbb{E}\iint_Q \theta^{-2}\varphi^{-2}v_{\varepsilon,1}^2dxdt+CT\lambda^{-1}\mu^{-2} e^{2\mu\|\psi\|_\infty}\mathbb{E}\iint_\Sigma \theta^{-2}_\varepsilon y_{\varepsilon,\Gamma}^2d\sigma dt\\
&\quad\,+C\lambda^2\mu^2\mathbb{E}\iint_\Sigma \theta^2\varphi^2 z_\Gamma^2 d\sigma dt+C\mathbb{E}\iint_\Sigma \theta^{-2}_\varepsilon y_{\varepsilon,\Gamma}^2  d\sigma dt+C\lambda^{-2}\mu^{-2}\mathbb{E}\iint_\Sigma \theta^{-2}\varphi^{-2}v_{\varepsilon,2}^2d\sigma dt.
    \end{aligned}
\end{align*}
Taking a large enough \( \lambda \geq C (e^{2 \mu \|\psi\|_\infty} T + T^2) \) and $\mu\geq C$, it follows that
\begin{align}\label{ineqforgradtert}
    \begin{aligned}
&\,\lambda^{-2}\mu^{-2}\mathbb{E}\iint_Q \theta^{-2}_\varepsilon\varphi^{-2}\vert\nabla y_\varepsilon\vert^2 dxdt+\lambda^{-2}\mu^{-2}\mathbb{E}\iint_\Sigma \theta^{-2}_\varepsilon\varphi^{-2}\vert\nabla_\Gamma y_{\varepsilon,\Gamma}\vert^2 d\sigma dt\\
&\leq C\mathbb{E}\iint_Q \theta_\varepsilon^{-2}  y_\varepsilon^2 dxdt+C\mathbb{E}\iint_\Sigma \theta^{-2}_\varepsilon y_{\varepsilon,\Gamma}^2  d\sigma dt+C\lambda^3\mu^4\mathbb{E}\iint_Q \theta^2\varphi^3 z^2  dxdt\\
&\quad\, +C\lambda^3\mu^3\mathbb{E}\iint_\Sigma \theta^2\varphi^3 z_\Gamma^2 d\sigma dt+C\lambda^{-3}\mu^{-4}\mathbb{E}\iint_{Q} \theta^{-2}\varphi^{-3} u_\varepsilon^2  dxdt\\
&\quad\,+C\lambda^{-2}\mu^{-2}\mathbb{E}\iint_Q \theta^{-2}\varphi^{-2}v_{\varepsilon,1}^2dxdt+C\lambda^{-2}\mu^{-2}\mathbb{E}\iint_\Sigma \theta^{-2}\varphi^{-2}v_{\varepsilon,2}^2d\sigma dt.
    \end{aligned}
\end{align}
Combining \eqref{ineqforgradtert} and \eqref{3.015}, we end up with
\begin{align}\label{3.021}
\begin{aligned}
&\,\frac{1}{\varepsilon}\mathbb{E}\int_G \vert y_\varepsilon(T)\vert^2dx+\frac{1}{\varepsilon}\mathbb{E}\int_\Gamma \vert y_{\varepsilon,\Gamma}(T)\vert^2d\sigma+\mathbb{E}\iint_{Q}\theta_\varepsilon^{-2}y_\varepsilon^2dxdt+\mathbb{E}\iint_{\Sigma}\theta_\varepsilon^{-2}y_{\varepsilon,\Gamma}^2d\sigma dt\\
&+\lambda^{-2}\mu^{-2}\mathbb{E}\iint_Q \theta^{-2}_\varepsilon\varphi^{-2}\vert\nabla y_\varepsilon\vert^2 dxdt+\lambda^{-2}\mu^{-2}\mathbb{E}\iint_\Sigma \theta^{-2}_\varepsilon\varphi^{-2}\vert\nabla_\Gamma y_{\varepsilon,\Gamma}\vert^2 d\sigma dt\\
&+\lambda^{-3}\mu^{-4}\mathbb{E}\iint_{Q}\theta^{-2}\varphi^{-3}u_\varepsilon^{2}dxdt+\lambda^{-2}\mu^{-2}\mathbb{E}\iint_{Q}\theta^{-2}\varphi^{-2}v_{\varepsilon,1}^2dxdt\\
&+\lambda^{-2}\mu^{-2}\mathbb{E}\iint_{\Sigma}\theta^{-2}\varphi^{-1} v_{\varepsilon,2}^2d\sigma dt\\
&\leq C\bigg[\lambda^3\mu^4\mathbb{E}\iint_{Q}\theta^2\varphi^3 z^2 dxdt+\lambda^3\mu^3\mathbb{E}\iint_{\Sigma}\theta^2\varphi^3 z_{\Gamma}^2 d\sigma dt\bigg].
\end{aligned}
\end{align}
\textbf{Step 3.} Using the uniform estimate \eqref{3.021} and recalling \eqref{3.2} and \eqref{functalphaepsilon}, we have that
\begin{align*}
\begin{aligned}
&\,\mathbb{E}\iint_{Q} (y_\varepsilon^2+\vert\nabla y_\varepsilon\vert^2)dxdt+\mathbb{E}\iint_{\Sigma} (y_{\varepsilon,\Gamma}^2+\vert\nabla_\Gamma y_{\varepsilon,\Gamma}\vert^2)d\sigma dt\\
&+\mathbb{E}\iint_{Q} (u_\varepsilon^{2}+v_{\varepsilon,1}^2)dxdt+\mathbb{E}\iint_{\Sigma} v_{\varepsilon,2}^2d\sigma dt\\
&\leq C\bigg[\lambda^3\mu^4\mathbb{E}\iint_{Q}\theta^2\varphi^3 z^2 dxdt+\lambda^3\mu^3\mathbb{E}\iint_{\Sigma}\theta^2\varphi^3 z_{\Gamma}^2 d\sigma dt\bigg],
\end{aligned}
\end{align*}
which implies that
\begin{align*}
\begin{aligned}
&\,\mathbb{E}\int_0^T \|u_\varepsilon\|_{L^2(G)}^{2} dt+ \mathbb{E}\int_0^T \|v_{\varepsilon,1}\|_{L^2(G)}^{2} dt+\mathbb{E}\int_0^T \|v_{\varepsilon,2}\|_{L^2(\Gamma)}^2 dt\\
&+\mathbb{E}\int_0^T \|y_\varepsilon\|_{H^1(G)}^2 dt+\mathbb{E}\int_0^T \|y_{\varepsilon,\Gamma}\|^2_{H^1(\Gamma)}  dt\\
&\leq C\bigg[\lambda^3\mu^4\mathbb{E}\iint_{Q}\theta^2\varphi^3 z^2 dxdt+\lambda^3\mu^3\mathbb{E}\iint_{\Sigma}\theta^2\varphi^3 z_{\Gamma}^2 d\sigma dt\bigg].
\end{aligned}
\end{align*}
Then, we deduce that there exists  
\begin{align*}
(\widehat{u},\widehat{v}_1,\widehat{v}_2,\widehat{y},\widehat{y}_\Gamma)\in&\,\, L^2_\mathcal{F}(0,T;L^2(\mathcal{B}))\times L^2_\mathcal{F}(0,T;L^2(G))\times L^2_\mathcal{F}(0,T;L^2(\Gamma))\\
&\times L^2_\mathcal{F}(0,T;H^1(G))\times L^2_\mathcal{F}(0,T;H^1(\Gamma)),
\end{align*}
such that as $\varepsilon\rightarrow0$,
$$
u_\varepsilon \longrightarrow \widehat{u},\,\,\,\textnormal{weakly in}\,\,\,L^2((0,T)\times\Omega;L^2(\mathcal{B}));
$$
$$
v_{\varepsilon,1} \longrightarrow \widehat{v}_1,\,\,\,\textnormal{weakly in}\,\,\,L^2((0,T)\times\Omega;L^2(G));
$$
\begin{equation}\label{3.022}
v_{\varepsilon,2} \longrightarrow \widehat{v}_2,\,\,\,\textnormal{weakly in}\,\,\,L^2((0,T)\times\Omega;L^2(\Gamma));
\end{equation}
$$
y_{\varepsilon} \longrightarrow \widehat{y},\,\,\,\textnormal{weakly in}\,\,\,L^2((0,T)\times\Omega;H^1(G));
$$
$$
y_{\varepsilon,\Gamma} \longrightarrow\widehat{y}_\Gamma,\,\,\,\textnormal{weakly in}\,\,\,L^2((0,T)\times\Omega;H^1(\Gamma)).
$$
Now, we claim that $(\widehat{y},\widehat{y}_\Gamma)$ is the solution of \eqref{3.6} associated to the controls $\widehat{u}, \widehat{v}_1$ and $\widehat{v}_2$. To prove this fact, we suppose  $(\widetilde{y},\widetilde{y}_\Gamma))$ is the unique solution in $L^2_\mathcal{F}(\Omega;C([0,T];\mathbb{L}^2))\cap L^2_\mathcal{F}(0,T;\mathbb{H}^1)$ to \eqref{3.6} with controls $\widehat{u}$, $\widehat{v}_1$ and $\widehat{v}_2$. For any processes $(f,g)\in L^2_\mathcal{F}(0,T;L^2(G))\times L^2_\mathcal{F}(0,T;L^2(\Gamma))$, we consider the following backward stochastic parabolic equation
		\begin{equation*}
			\begin{cases}
				\begin{array}{ll}
d\phi+\nabla\cdot(\mathcal{A}\nabla\phi) \,dt = f \,dt+\Phi \,dW(t) &\textnormal{in} \,\,Q,\\
d\phi_\Gamma+\nabla_\Gamma\cdot(\mathcal{A}_\Gamma\nabla_\Gamma\phi_\Gamma) dt-\partial_\nu^\mathcal{A} \phi \,dt = g \,dt+\widehat{\Phi} \,dW(t) &\textnormal{on}\,\,  \Sigma,\\
\phi_\Gamma=\phi\vert_\Gamma&\textnormal{on} \,\,\Sigma,\\
(\phi,\phi_\Gamma)\vert_{t=T}=(0,0)&\textnormal{in} \,\,G\times\Gamma.
				\end{array}
			\end{cases}
		\end{equation*}
By Itô's formula, computing ``$d\langle(\phi,\phi_\Gamma),(\widetilde{y},\widetilde{y}_\Gamma)\rangle_{\mathbb{L}^2}-d\langle(\phi,\phi_\Gamma),(y_\varepsilon,y_{\varepsilon,\Gamma})\rangle_{\mathbb{L}^2}$", integrating the obtained equality on $(0,T)$, then we deduce that
\begin{align}\label{3.023}
    \begin{aligned}
       &\, 0=\mathbb{E}\iint_Q f(y_\varepsilon-\widetilde{y})dxdt+\mathbb{E}\iint_{Q_0} \phi(u_\varepsilon-\widehat{u})dxdt+\mathbb{E}\iint_Q \Phi(v_{\varepsilon,1}-\widehat{v}_1)dxdt\\
        &\quad\;+\mathbb{E}\iint_\Sigma g(y_{\varepsilon,\Gamma}-\widetilde{y}_\Gamma)d\sigma dt+\mathbb{E}\iint_\Sigma \widehat{\Phi}(v_{\varepsilon,2}-\widehat{v}_2)d\sigma dt.
    \end{aligned}
\end{align}
Letting $\varepsilon\rightarrow0$ in \eqref{3.023}, we get
\begin{align}\label{3.024}
    \begin{aligned}
\mathbb{E}\iint_Q f(\widehat{y}-\widetilde{y})dxdt+\mathbb{E}\iint_\Sigma g(\widehat{y}_{\Gamma}-\widetilde{y}_\Gamma)d\sigma dt=0.
    \end{aligned}
\end{align}
From \eqref{3.024}, we conclude that $\widetilde{y}=\widehat{y}$ in $Q$, a.s., and $\widetilde{y}_\Gamma=\widehat{y}_\Gamma$ on $\Sigma$, a.s. Then we deduce that $(\widehat{y},\widehat{y}_\Gamma)$ is the solution of \eqref{3.6} with controls $\widehat{u}, \widehat{v}_1$ and $\widehat{v}_2$. Finally, combining the uniform estimate \eqref{3.021} and the convergence results \eqref{3.022}, we obtain the null controllability of \eqref{3.6}, along with the desired inequality \eqref{3.7}. This concludes the proof of Proposition \ref{prop3.1}.
\end{proof}

Now, we are in a position to prove Theorem \ref{thm1.1}.
\begin{proof}[Proof of Theorem \ref{thm1.1}]
Let $(z,z_\Gamma;Z,\widehat{Z})$ be the solution of \eqref{1.1} and $(\widehat{y},\widehat{y}_\Gamma)$ be the solution of \eqref{3.6} associated to the controls $\widehat{u}, \widehat{v}_1$ and $\widehat{v}_2$ obtained in Proposition \ref{prop3.1}. By applying Itô's formula, we compute $d\langle(\widehat{y},\widehat{y}_\Gamma),(z,z_\Gamma)\rangle_{\mathbb{L}^2}$, integrating the result on $(0,T)$ and taking the expectation on both sides, we obtain
\begin{align}\label{3.9}
\begin{aligned}
\lambda^3\mu^4\mathbb{E}\iint_Q \theta^2\varphi^3z^2dxdt+\lambda^3\mu^3\mathbb{E}\iint_\Sigma \theta^2\varphi^3z_\Gamma^2d\sigma dt =&-\mathbb{E}\iint_Q [\mathbbm{1}_{\mathcal{B}}\widehat{u}z+F_1\widehat{y}-F\cdot\nabla\widehat{y}+Z\widehat{v}_1]dxdt\\
&-\mathbb{E}\iint_\Sigma 
 [F_2\widehat{y}_\Gamma-F_\Gamma\cdot\nabla_\Gamma\widehat{y}_\Gamma +\widehat{Z}\widehat{v}_2] d\sigma dt.
\end{aligned}
\end{align}
Applying Young's inequality on the right-hand side of \eqref{3.9}, we get that for all $\rho>0$,
\begin{align}\label{3.10}
\begin{aligned}
&\,\lambda^3\mu^4\mathbb{E}\iint_Q \theta^2\varphi^3z^2dxdt+\lambda^3\mu^3\mathbb{E}\iint_\Sigma \theta^2\varphi^3z_\Gamma^2d\sigma dt\\
&\leq \rho\bigg[\lambda^{-3}\mu^{-4}\mathbb{E}\iint_{Q} \theta^{-2}\varphi^{-3}\widehat{u}^2dxdt+\lambda^{-2}\mu^{-2}\mathbb{E}\iint_Q \theta^{-2}\varphi^{-2}\widehat{v}_1^2dxdt\\
&\hspace{0.8cm}+\lambda^{-2}\mu^{-2}\mathbb{E}\iint_\Sigma \theta^{-2}\varphi^{-2}\widehat{v}_2^2d\sigma dt+\mathbb{E}\iint_Q \theta^{-2}\widehat{y}^2dxdt+\mathbb{E}\iint_\Sigma \theta^{-2}\widehat{y}_\Gamma^2d\sigma dt\\
&\hspace{0.8cm}+\lambda^{-2}\mu^{-2}\mathbb{E}\iint_Q \theta^{-2}\varphi^{-2}\vert\nabla\widehat{y}\vert^2dxdt+\lambda^{-2}\mu^{-2}\mathbb{E}\iint_\Sigma \theta^{-2}\varphi^{-2}\vert\nabla_\Gamma\widehat{y}_\Gamma\vert^2d\sigma dt\bigg]\\
&\quad +\frac{1}{4\rho}\bigg[\lambda^{3}\mu^{4}\mathbb{E}\int_0^T\int_{\mathcal{B}} \theta^{2}\varphi^{3}z^2dxdt+\lambda^{2}\mu^{2}\mathbb{E}\iint_Q \theta^{2}\varphi^{2}Z^2dxdt\\
&\hspace{1.6cm}+\lambda^{2}\mu^2\mathbb{E}\iint_\Sigma \theta^{2}\varphi^2\widehat{Z}^2d\sigma dt+\mathbb{E}\iint_Q \theta^{2}F_1^2dxdt+\mathbb{E}\iint_\Sigma \theta^{2}F_2^2d\sigma dt\\
&\hspace{1.6cm}+\lambda^{2}\mu^{2}\mathbb{E}\iint_Q \theta^{2}\varphi^{2}\vert F\vert^2dxdt+\lambda^{2}\mu^{2}\mathbb{E}\iint_\Sigma \theta^{2}\varphi^{2}\vert F_\Gamma\vert^2d\sigma dt\bigg].
\end{aligned}
\end{align}
Using Carleman estimate \eqref{3.7} and taking a small enough $\rho>0$, \eqref{3.10} implies that 
\begin{align}\label{3.11firineq}
\begin{aligned}
&\,\lambda^3\mu^4\mathbb{E}\iint_Q \theta^2\varphi^3z^2dxdt+\lambda^3\mu^3\mathbb{E}\iint_\Sigma \theta^2\varphi^3z_\Gamma^2d\sigma dt\\
&\leq C\bigg[\lambda^{3}\mu^{4}\mathbb{E}\int_0^T\int_{\mathcal{B}} \theta^{2}\varphi^{3}z^2dxdt+\lambda^{2}\mu^{2}\mathbb{E}\iint_Q \theta^{2}\varphi^{2}Z^2dxdt\\
&\hspace{1cm}+\lambda^{2}\mu^2\mathbb{E}\iint_\Sigma \theta^{2}\varphi^2\widehat{Z}^2d\sigma dt+\mathbb{E}\iint_Q \theta^{2}F_1^2dxdt+\mathbb{E}\iint_\Sigma \theta^{2}F_2^2d\sigma dt\\
&\hspace{1cm}+\lambda^{2}\mu^{2}\mathbb{E}\iint_Q \theta^{2}\varphi^{2}\vert F\vert^2dxdt+\lambda^{2}\mu^{2}\mathbb{E}\iint_\Sigma \theta^{2}\varphi^{2}\vert F_\Gamma\vert^2d\sigma dt\bigg],
\end{aligned}
\end{align}
for all $\mu\geq C$ and $\lambda\geq C(e^{2\mu\|\psi\|_\infty}T+T^2)$. 

On the other hand, computing $d\|\theta\varphi^{\frac{1}{2}} z\|^2_{L^2(G)}$ by using Itô's formula, integrating the equality with respect to $t\in(0,T)$ and taking the expectation on both sides, we obtain that
\begin{align}\label{3.12}
\begin{aligned}
0=&\,\mathbb{E}\iint_Q (\theta^2\varphi)_tz^2dxdt+2\mathbb{E}\iint_Q \mathcal{A}\nabla z\cdot\nabla(\theta^2\varphi z)dxdt\\
&-2\mathbb{E}\int_0^T\langle\partial_\nu^\mathcal{A} z,\theta^2\varphi z_\Gamma\rangle_{H^{-1/2}(\Gamma),H^{1/2}(\Gamma)} dt+2\mathbb{E}\iint_Q \theta^2\varphi zF_1\,dxdt\\
&-2\mathbb{E}\iint_Q F\cdot\nabla(\theta^2\varphi z)dxdt+2\mathbb{E}\int_0^T \langle F\cdot\nu,\theta^2\varphi z_\Gamma\rangle_{H^{-1/2}(\Gamma),H^{1/2}(\Gamma)} dt\\
&+\mathbb{E}\iint_Q \theta^2\varphi Z^2dxdt.
\end{aligned}
\end{align}
Similarly to \eqref{3.12}, by computing $d\|\theta\varphi^{\frac{1}{2}}z_\Gamma\|^2_{L^2(\Gamma)}$, integrating the result on $(0,T)$ and taking the expectation, we get that
\begin{align}\label{3.13}
    \begin{aligned}
0=&\,\mathbb{E}\iint_\Sigma (\theta^2\varphi)_tz_\Gamma^2d\sigma dt+2\mathbb{E}\iint_\Sigma \mathcal{A}_\Gamma\nabla_\Gamma z_\Gamma\cdot\nabla_\Gamma(\theta^2\varphi z_\Gamma)d\sigma dt\\
&+2\mathbb{E}\int_0^T \langle \partial_\nu^\mathcal{A} z,\theta^2\varphi z_\Gamma\rangle_{H^{-1/2}(\Gamma),H^{1/2}(\Gamma)} dt+2\mathbb{E}\iint_\Sigma\theta^2\varphi z_\Gamma F_2 \, d\sigma dt\\
&-2\mathbb{E}\int_0^T \langle F\cdot\nu,\theta^2\varphi z_\Gamma\rangle_{H^{-1/2}(\Gamma),H^{1/2}(\Gamma)} dt-2\mathbb{E}\iint_\Sigma F_\Gamma\cdot\nabla_\Gamma(\theta^2\varphi z_\Gamma)d\sigma dt\\
&+\mathbb{E}\iint_\Sigma \theta^2\varphi \widehat{Z}^2d\sigma dt.
\end{aligned}
\end{align}
Combining \eqref{3.12} and \eqref{3.13}, it follows that
\begin{align}\label{3.12inte}
\begin{aligned}
0=&\,2\mathbb{E}\iint_Q \mathcal{A}\nabla z\cdot\nabla(\theta^2\varphi z)dxdt+2\mathbb{E}\iint_\Sigma \mathcal{A}_\Gamma\nabla_\Gamma z_\Gamma\cdot\nabla_\Gamma(\theta^2\varphi z_\Gamma)d\sigma dt\\
&+\mathbb{E}\iint_Q (\theta^2\varphi)_tz^2dxdt+\mathbb{E}\iint_\Sigma (\theta^2\varphi)_tz_\Gamma^2d\sigma dt+2\mathbb{E}\iint_Q \theta^2\varphi zF_1\,dxdt\\
&+2\mathbb{E}\iint_\Sigma\theta^2\varphi z_\Gamma F_2 \, d\sigma dt-2\mathbb{E}\iint_Q F\cdot\nabla(\theta^2\varphi z)dxdt\\
&-2\mathbb{E}\iint_\Sigma F_\Gamma\cdot\nabla_\Gamma(\theta^2\varphi z_\Gamma)d\sigma dt+\mathbb{E}\iint_Q \theta^2\varphi Z^2dxdt+\mathbb{E}\iint_\Sigma \theta^2\varphi \widehat{Z}^2d\sigma dt,
\end{aligned}
\end{align}
which implies that
\begin{align}\label{3.12intesec}
\begin{aligned}
&\,2\mathbb{E}\iint_Q \theta^2\varphi\mathcal{A}\nabla z\cdot\nabla zdxdt+2\mathbb{E}\iint_\Sigma \theta^2\varphi\mathcal{A}_\Gamma\nabla_\Gamma z_\Gamma\cdot\nabla_\Gamma z_\Gamma d\sigma dt\\
&=-2\mathbb{E}\iint_Q z\mathcal{A}\nabla z\cdot\nabla(\theta^2\varphi)dxdt-\mathbb{E}\iint_Q (\theta^2\varphi)_tz^2dxdt-\mathbb{E}\iint_\Sigma (\theta^2\varphi)_tz_\Gamma^2d\sigma dt\\
&\quad\,-2\mathbb{E}\iint_Q \theta^2\varphi zF_1\,dxdt-2\mathbb{E}\iint_\Sigma\theta^2\varphi z_\Gamma F_2 \, d\sigma dt+2\mathbb{E}\iint_Q \theta^2\varphi F\cdot\nabla zdxdt\\
&\quad\,+2\mathbb{E}\iint_Q zF\cdot\nabla(\theta^2\varphi)dxdt+2\mathbb{E}\iint_\Sigma \theta^2\varphi F_\Gamma\cdot\nabla_\Gamma z_\Gamma d\sigma dt\\
&\quad\,-\mathbb{E}\iint_Q \theta^2\varphi Z^2dxdt-\mathbb{E}\iint_\Sigma \theta^2\varphi \widehat{Z}^2d\sigma dt.
\end{aligned}
\end{align}
It is easy to see that for large $\lambda\geq CT^2$, we have
\begin{equation}\label{2.32012}
    \vert (\theta^2\varphi)_t\vert\leq CT\lambda  e^{2\mu\|\psi\|_\infty}\theta^2\varphi^3\quad\;\;\textnormal{and} \quad\;\;\vert\nabla(\theta^2\varphi)\vert\leq C\lambda\mu\theta^2\varphi^2.
\end{equation}
Recalling \eqref{asummAandAgama} and using \eqref{2.32012}, we obtain for large $\lambda\geq CT^2$,
\begin{align}\label{3.12intesec2s}
\begin{aligned}
&\,\mathbb{E}\iint_Q \theta^2\varphi|\nabla z|^2dxdt+\mathbb{E}\iint_\Sigma \theta^2\varphi|\nabla_\Gamma z_\Gamma|^2 d\sigma dt\\
&\leq C\lambda\mu\mathbb{E}\iint_Q \theta^2\varphi^2|z||\nabla z| dxdt+CT\lambda  e^{2\mu\|\psi\|_\infty}\mathbb{E}\iint_Q \theta^2\varphi^3z^2dxdt\\
&\quad\,+CT\lambda  e^{2\mu\|\psi\|_\infty}\mathbb{E}\iint_\Sigma \theta^2\varphi^3z_\Gamma^2d\sigma dt+C\mathbb{E}\iint_Q \theta^2\varphi |z| |F_1|\,dxdt\\
&\quad\,+C\mathbb{E}\iint_\Sigma\theta^2\varphi |z_\Gamma| |F_2| \, d\sigma dt+C\mathbb{E}\iint_Q \theta^2\varphi |F||\nabla z| dxdt\\
&\quad\,+C\lambda\mu\mathbb{E}\iint_Q \theta^2\varphi^2 |z| |F|dxdt+C\mathbb{E}\iint_\Sigma \theta^2\varphi |F_\Gamma||\nabla_\Gamma z_\Gamma| d\sigma dt\\
&\quad\,+C\mathbb{E}\iint_Q \theta^2\varphi Z^2dxdt+C\mathbb{E}\iint_\Sigma \theta^2\varphi \widehat{Z}^2d\sigma dt.
\end{aligned}
\end{align}
Applying Young's inequality in the right-hand side of \eqref{3.12intesec2s}, we get for any $\varepsilon>0$,
\begin{align}\label{3.31intesecYgs}
\begin{aligned}
&\,\mathbb{E}\iint_Q \theta^2\varphi|\nabla z|^2dxdt+\mathbb{E}\iint_\Sigma \theta^2\varphi|\nabla_\Gamma z_\Gamma|^2 d\sigma dt\\
&\leq \varepsilon\mathbb{E}\iint_Q \theta^2\varphi|\nabla z|^2 dxdt+\varepsilon\mathbb{E}\iint_\Sigma \theta^2\varphi |\nabla_\Gamma z_\Gamma|^2 d\sigma dt+\frac{C}{\varepsilon}\lambda^2\mu^2\mathbb{E}\iint_Q \theta^2\varphi^3 z^2 dxdt\\
&\quad\,+\frac{C}{\varepsilon}\mathbb{E}\iint_Q \theta^2\varphi |F|^2 dxdt+\frac{C}{\varepsilon}\mathbb{E}\iint_\Sigma \theta^2\varphi |F_\Gamma|^2 d\sigma dt+CT\lambda  e^{2\mu\|\psi\|_\infty}\mathbb{E}\iint_Q \theta^2\varphi^3z^2dxdt\\
&\quad\,+CT\lambda  e^{2\mu\|\psi\|_\infty}\mathbb{E}\iint_\Sigma \theta^2\varphi^3z_\Gamma^2d\sigma dt+C\lambda^{-1}\mu^{-2}\mathbb{E}\iint_Q \theta^2 F_1^2\,dxdt+C\lambda\mu^2\mathbb{E}\iint_Q \theta^2\varphi^2 z^2\,dxdt\\
&\quad\,+C\lambda^{-1}\mu^{-1}\mathbb{E}\iint_\Sigma\theta^2 F_2^2 \, d\sigma dt+C\lambda\mu\mathbb{E}\iint_\Sigma\theta^2\varphi^2 z_\Gamma^2 \, d\sigma dt+C\lambda\mathbb{E}\iint_Q \theta^2\varphi^2 |F|^2 dxdt\\
&\quad\,+C\lambda\mu^2\mathbb{E}\iint_Q \theta^2\varphi^2 z^2 dxdt+C\mathbb{E}\iint_Q \theta^2\varphi Z^2dxdt+C\mathbb{E}\iint_\Sigma \theta^2\varphi \widehat{Z}^2d\sigma dt.
\end{aligned}
\end{align}
Choosing a small enough $\varepsilon$ in \eqref{3.31intesecYgs}, and multiplying the obtained inequality by $\lambda\mu$, we deduce that
\begin{align}\label{3.31intesecYgslst}
\begin{aligned}
&\,\lambda\mu\mathbb{E}\iint_Q \theta^2\varphi|\nabla z|^2dxdt+\lambda\mu\mathbb{E}\iint_\Sigma \theta^2\varphi|\nabla_\Gamma z_\Gamma|^2 d\sigma dt\\
&\leq C\lambda^3\mu^3\mathbb{E}\iint_Q \theta^2\varphi^3 z^2 dxdt+C\lambda\mu\mathbb{E}\iint_Q \theta^2\varphi |F|^2 dxdt+C\lambda\mu\mathbb{E}\iint_\Sigma \theta^2\varphi |F_\Gamma|^2 d\sigma dt\\
&\quad\,+CT\lambda^2\mu  e^{2\mu\|\psi\|_\infty}\mathbb{E}\iint_Q \theta^2\varphi^3z^2dxdt+CT\lambda^2\mu  e^{2\mu\|\psi\|_\infty}\mathbb{E}\iint_\Sigma \theta^2\varphi^3z_\Gamma^2d\sigma dt\\
&\quad\,+C\lambda^2\mu^3\mathbb{E}\iint_Q \theta^2\varphi^2 z^2\,dxdt+C\mu^{-1}\mathbb{E}\iint_Q \theta^2 F_1^2\,dxdt+C\mathbb{E}\iint_\Sigma\theta^2 F_2^2 \, d\sigma dt\\
&\quad\,+C\lambda^2\mu^2\mathbb{E}\iint_\Sigma\theta^2\varphi^2 z_\Gamma^2 \, d\sigma dt+C\lambda^2\mu\mathbb{E}\iint_Q \theta^2\varphi^2 |F|^2 dxdt+C\lambda^2\mu^3\mathbb{E}\iint_Q \theta^2\varphi^2 z^2 dxdt\\
&\quad\,+C\lambda\mu\mathbb{E}\iint_Q \theta^2\varphi Z^2dxdt+C\lambda\mu\mathbb{E}\iint_\Sigma \theta^2\varphi \widehat{Z}^2d\sigma dt.
\end{aligned}
\end{align}
Choosing a sufficiently large \( \lambda \geq C(e^{2\mu\|\psi\|_\infty}T + T^2) \) and $\mu\geq C$ in \eqref{3.31intesecYgslst}, we conclude that
\begin{align}\label{lastoneYgslst}
\begin{aligned}
&\,\lambda\mu\mathbb{E}\iint_Q \theta^2\varphi|\nabla z|^2dxdt+\lambda\mu\mathbb{E}\iint_\Sigma \theta^2\varphi|\nabla_\Gamma z_\Gamma|^2 d\sigma dt\\
&\leq C\lambda^3\mu^3\mathbb{E}\iint_Q \theta^2\varphi^3 z^2 dxdt+C\lambda^3\mu^2\mathbb{E}\iint_\Sigma\theta^2\varphi^3 z_\Gamma^2 \, d\sigma dt+C\lambda^2\mu\mathbb{E}\iint_Q \theta^2\varphi^2 |F|^2 dxdt\\
&\quad\,+C\lambda^2\mu\mathbb{E}\iint_\Sigma \theta^2\varphi^2 |F_\Gamma|^2 d\sigma dt+C\mu^{-1}\mathbb{E}\iint_Q \theta^2 F_1^2\,dxdt+C\mathbb{E}\iint_\Sigma\theta^2 F_2^2 \, d\sigma dt\\
&\quad\,+C\lambda^2\mu\mathbb{E}\iint_Q \theta^2\varphi^2 Z^2dxdt+C\lambda^2\mu\mathbb{E}\iint_\Sigma \theta^2\varphi^2 \widehat{Z}^2d\sigma dt.
\end{aligned}
\end{align}
Finally, by combining \eqref{lastoneYgslst} and \eqref{3.11firineq} and taking a large $\mu\geq C$, we obtain the desired Carleman estimate \eqref{1.3}. This completes the proof of Theorem \ref{thm1.1}.
\end{proof}

\section{ Application to Null Controllability of System \eqref{CE}}\label{sec3}
This section is devoted primarily to establishing the null controllability result, stated in Theorem \ref{thmm1.2null}, of the forward anisotropic  stochastic parabolic system \eqref{CE}, based on the main Carleman estimate \eqref{1.3}. This estimate will be employed to derive the corresponding observability inequality for the adjoint backward equation, in the sense introduced in \cite[Chap.~7, Sect.~7.3]{lu2021mathematical}.

Let us consider the following adjoint backward equation associated with the controlled system \eqref{CE}:
\begin{equation}\label{4.1}
\begin{cases}
			\begin{array}{ll}
		dz + \nabla\cdot(\mathcal{A}\nabla z) \,dt = (-a_1 z+\nabla\cdot(zB_1)) \,dt + Z \,dW(t) &\textnormal{in}\,\,Q,\\
		dz_\Gamma+\nabla_\Gamma\cdot(\mathcal{A}_\Gamma\nabla_\Gamma z_\Gamma) \,dt-\partial_\nu^\mathcal{A} z \,dt = (-a_2z_\Gamma-z B_1\cdot\nu+\nabla_\Gamma\cdot(z_\Gamma B_2))\,dt+\widehat{Z} \,dW(t) &\textnormal{on}\,\,\Sigma,\\
z_\Gamma=z\vert_\Gamma &\textnormal{on}\,\,\Sigma,\\
(z,z_\Gamma)\vert_{t=T}=(z_T,z_{\Gamma,T}) &\textnormal{in}\,\,G\times\Gamma,
			\end{array}
		\end{cases}
\end{equation}
where $(z_T,z_{\Gamma,T})\in L^2_{\mathcal{F}_T}(\Omega;\mathbb{L}^2)$ is the terminal state. Similar to the well-posedness of \eqref{1.1}, we also have that \eqref{4.1} is well-posed i.e., for all $(z_T,z_{\Gamma,T})\in L^2_{\mathcal{F}_T}(\Omega;\mathbb{L}^2)$, the system \eqref{4.1} admits a unique weak solution
$$(z,z_\Gamma;Z,\widehat{Z})\in \Big(L^2_\mathcal{F}(\Omega;C([0,T];\mathbb{L}^2))\bigcap L^2_\mathcal{F}(0,T;\mathbb{H}^1)\Big)\times L^2_\mathcal{F}(0,T;\mathbb{L}^2),$$
such that 
\begin{align}\label{2.011}
\max_{0 \le t \le T} \mathbb{E}\|(z(t),z_\Gamma(t))\|^2_{\mathbb{L}^2}+\|(z,z_\Gamma)\|^2_{L^2_\mathcal{F}(0,T;\mathbb{H}^1)}+\| (Z,\widehat{Z}) \|^2_{L^2_\mathcal{F}(0,T;\mathbb{L}^2)}\leq C \,\mathbb{E}\|(z_T,z_{\Gamma,T})\|^2_{\mathbb{L}^2}.
\end{align}

From Theorem \ref{thm1.1}, we deduce the following Carleman estimate for solutions of the system \eqref{4.1}.
\begin{cor}\label{cor4.1}
For $\mu=\mu_0$ given in Theorem \ref{thm1.1}. There exist a large \( \lambda_0 \), and a constant \( C \) depending only on \( G \), \( G_0 \), $\mu_0$, $\beta_0$, $M$, and $M_\Gamma$ such that, for all final state $(z_T,z_{\Gamma,T})\in L^2_{\mathcal{F}_T}(\Omega;\mathbb{L}^2)$, the weak solution $(z,z_\Gamma;Z,\widehat{Z})$ of \eqref{4.1} satisfies that
\begin{align}\label{4.2}
 \begin{aligned}
&\,\lambda^3\mathbb{E}\iint_Q \theta^2\varphi^3|z|^2dxdt + \lambda^3\mathbb{E}\iint_\Sigma \theta^2\varphi^3|z_\Gamma|^2d\sigma dt\\
&+\lambda\mathbb{E}\iint_Q \theta^2\varphi \vert\nabla z\vert^2dxdt + \lambda\mathbb{E}\iint_\Sigma \theta^2\varphi \vert\nabla_\Gamma z_\Gamma\vert^2 d\sigma dt\\
&\leq C \bigg[ \lambda^3\mathbb{E}\iint_{Q_0} \theta^2\varphi^3 |z|^2 dxdt+\lambda^3\mathbb{E}\iint_Q \theta^2\varphi^3 |Z|^2 dxdt + \lambda^3\mathbb{E}\iint_\Sigma \theta^2\varphi^3 |\widehat{Z}|^2 d\sigma dt\bigg],
  \end{aligned}
	\end{align}
for sufficiently large 
$$ \lambda \geq \lambda_0\left[T + T^2\Big(1+\|a_1\|_\infty^{2/3}+\|a_2\|_\infty^{2/3}+\|B_1\|_\infty^2+\|B_2\|_\infty^2\Big)\right].$$
	\end{cor}
 \begin{proof}
By applying Carleman estimate \eqref{1.3} with the coefficients
     $$F_1=-a_1z,\quad F=zB_1,\quad F_2=-a_2z_\Gamma,\quad F_\Gamma=z_\Gamma B_2,$$
     and the set $\mathcal{B}=G_0$. Then we have that for $\lambda\geq C(T+T^2)$,
\begin{align}\label{1.3froncar}
 \begin{aligned}
&\,\lambda^3\mathbb{E}\iint_Q \theta^2\varphi^3z^2dxdt + \lambda^3\mathbb{E}\iint_\Sigma \theta^2\varphi^3z_\Gamma^2d\sigma dt\\
&+\lambda\mathbb{E}\iint_Q \theta^2\varphi \vert\nabla z\vert^2dxdt + \lambda\mathbb{E}\iint_\Sigma \theta^2\varphi \vert\nabla_\Gamma z_\Gamma\vert^2 d\sigma dt\\
&\leq C \bigg[ \lambda^3\mathbb{E}\iint_{Q_0} \theta^2\varphi^3 z^2 dxdt + \mathbb{E}\iint_Q \theta^2 |a_1z|^2 dxdt + \mathbb{E}\iint_\Sigma \theta^2 |a_2z_\Gamma|^2 d\sigma dt\\
&
\qquad\;+\lambda^2\mathbb{E}\iint_Q \theta^2\varphi^2 \vert zB_1\vert^2 dxdt + \lambda^2\mathbb{E}\iint_\Sigma \theta^2\varphi ^2 \vert z_\Gamma B_2\vert^2 d\sigma dt\\
&
\qquad\;+\lambda^2\mathbb{E}\iint_Q \theta^2\varphi^2 Z^2 dxdt + \lambda^2\mathbb{E}\iint_\Sigma \theta^2\varphi^2 \widehat{Z}\,^2 d\sigma dt\bigg].
\end{aligned}
\end{align}
Firstly, it is easy to see that
\begin{align}\label{inee3.55}
    \begin{aligned}
&C\mathbb{E}\iint_Q \theta^2 |a_1z|^2 dxdt + C\mathbb{E}\iint_\Sigma \theta^2 |a_2z_\Gamma|^2 d\sigma dt\\
&+C\lambda^2\mathbb{E}\iint_Q \theta^2\varphi^2 \vert zB_1\vert^2 dxdt + C\lambda^2\mathbb{E}\iint_\Sigma \theta^2\varphi ^2 \vert z_\Gamma B_2\vert^2 d\sigma dt\\
&\leq \frac{1}{2}\lambda^3\mathbb{E}\iint_Q \theta^2\varphi^3z^2dxdt + \frac{1}{2}\lambda^3\mathbb{E}\iint_\Sigma \theta^2\varphi^3z_\Gamma^2d\sigma dt,
    \end{aligned}
\end{align}
for large $\lambda\geq CT^2(\|a_1\|_\infty^{2/3}+\|a_2\|_\infty^{2/3}+\|B_1\|_\infty^2+\|B_2\|_\infty^2)$. 

We also have that for a large $\lambda\geq CT^2$,
\begin{align}\label{ineZZhat}
\begin{aligned}
&C\lambda^2\mathbb{E}\iint_Q \theta^2\varphi^2 Z^2 dxdt + C\lambda^2\mathbb{E}\iint_\Sigma \theta^2\varphi^2 \widehat{Z}\,^2 d\sigma dt\\
&\leq \lambda^3\mathbb{E}\iint_Q \theta^2\varphi^3 Z^2 dxdt + \lambda^3\mathbb{E}\iint_\Sigma \theta^2\varphi^3 \widehat{Z}\,^2 d\sigma dt.
\end{aligned}
\end{align}
Combining \eqref{1.3froncar}, \eqref{inee3.55} and \eqref{ineZZhat}, we deduce the desired estimate \eqref{4.2}.
 \end{proof}
From Corollary \ref{cor4.1}, we derive the following observability inequality.
\begin{prop}\label{prop3.2}
For any final state $(z_T,z_{\Gamma,T})\in L^2_{\mathcal{F}_T}(\Omega;\mathbb{L}^2)$, the associated solution $(z,z_\Gamma;Z,\widehat{Z})$ of equation \eqref{4.1} satisfies that 
\begin{align}\label{4.4}
\begin{aligned}
&\,\mathbb{E}\| z(0,\cdot)\|^2_{L^2(G)}+\mathbb{E}\| z_\Gamma(0,\cdot)\|^2_{L^2(\Gamma)}\\
&\leq e^{CK} \,\bigg[\mathbb{E}\iint_{Q_0} |z|^2 dx dt+\mathbb{E}\iint_{Q} |Z|^2 dx dt+\mathbb{E}\iint_{\Sigma} |\widehat{Z}|^2 d\sigma dt\bigg],
  \end{aligned}
\end{align}
where \(K\) has the following form:
\begin{align*}
K \equiv 1 + \frac{1}{T} + \| a_1 \|_\infty^{2/3}+ \| a_2 \|_\infty^{2/3} + T (\| a_1 \|_\infty+\| a_2 \|_\infty) + (1+T) (\| B_1 \|_\infty^2+\| B_2 \|_\infty^2).
\end{align*}
\end{prop}
\begin{proof}
From  the Carleman inequality \eqref{4.2}, we deduce
\begin{align}\label{4.2secc}
 \begin{aligned}
&\,\mathbb{E}\int_{T/4}^{3T/4}\int_G \theta^2\varphi^3z^2dxdt +\mathbb{E}\int_{T/4}^{3T/4}\int_\Gamma \theta^2\varphi^3z_\Gamma^2d\sigma dt\\
&\leq C \bigg[\mathbb{E}\iint_{Q_0} \theta^2\varphi^3 z^2 dxdt+\mathbb{E}\iint_Q \theta^2\varphi^3 Z^2 dxdt + \mathbb{E}\iint_\Sigma \theta^2\varphi^3 \widehat{Z}\,^2 d\sigma dt\bigg],
  \end{aligned}
	\end{align}
for large  $\lambda \geq \lambda_1(T)\equiv \lambda_0\left[T + T^2\Big(1+\|a_1\|_\infty^{2/3}+\|a_2\|_\infty^{2/3}+\|B_1\|_\infty^2+\|B_2\|_\infty^2\Big)\right].$ Subsequently, notice that the functions
$$ t\mapsto\min_{x\in\overline{G}}\, [\theta^2(t,x)\varphi^3(t,x)] $$
reach their minimum in \( [T/4, 3T/4] \) at \( t = T/4 \) and 
$$ t\mapsto\max_{x\in\overline{G}}\, [\theta^2(t,x)\varphi^3(t,x)] $$
reach their maximum in \( (0,T) \) at \( t = T/2 \). Then, fixing \( \lambda = \lambda_1(T) \), we deduce that
\begin{align}\label{4.2toOBSIEN}
 \begin{aligned}
&\,\mathbb{E}\int_{T/4}^{3T/4}\int_G z^2dxdt +\mathbb{E}\int_{T/4}^{3T/4}\int_\Gamma z_\Gamma^2d\sigma dt\\
&\leq e^{C\textbf{K}_1} \bigg[\mathbb{E}\iint_{Q_0} z^2 dxdt+\mathbb{E}\iint_Q Z^2 dxdt + \mathbb{E}\iint_\Sigma \widehat{Z}\,^2 d\sigma dt\bigg],
  \end{aligned}
	\end{align}
where $\textbf{K}_1\equiv 1+\frac{1}{T}+\|a_1\|_\infty^{2/3}+\|a_2\|_\infty^{2/3}+\|B_1\|_\infty^2+\|B_2\|_\infty^2$.

On the other hand, let $t\in(0, T)$, computing $d_s \|(z,z_\Gamma)\|^2_{\mathbb{L}^2}$ by using the Itô's formula. Then integrating the equality w.r.t $s\in(0,t)$ and taking the expectation on both sides, we conclude that
\begin{align*}
    \begin{aligned}
        &\,\mathbb{E}\int_G z^2(0)dx + \mathbb{E}\int_\Gamma z_\Gamma^2(0)d\sigma -\left(\mathbb{E}\int_G z^2(t)dx + \mathbb{E}\int_\Gamma z_\Gamma^2(t)d\sigma\right)\\
        &=-2\mathbb{E}\int_0^t\int_G \mathcal{A}\nabla z\cdot\nabla z dxds+2\mathbb{E}\int_0^t\int_G a_1z^2dxds+2\mathbb{E}\int_0^t\int_G zB_1\cdot\nabla z \,dxds\\
        &\quad\,-\mathbb{E}\int_0^t\int_G Z^2 dxds-2\mathbb{E}\int_0^t\int_\Gamma \mathcal{A}_\Gamma\nabla_\Gamma z_\Gamma\cdot\nabla_\Gamma z_\Gamma d\sigma ds+2\mathbb{E}\int_0^t\int_\Gamma a_2z_\Gamma^2d\sigma ds\\
 &\quad\,+2\mathbb{E}\int_0^t\int_\Gamma z_\Gamma B_2\cdot\nabla_\Gamma z_\Gamma \,d\sigma ds-\mathbb{E}\int_0^t\int_\Gamma \widehat{Z}^2d\sigma ds.
    \end{aligned}
\end{align*}
Recalling \eqref{asummAandAgama}, and using Young's inequality, it follows that for any $\varepsilon>0$,
\begin{align}\label{4.6ine}
    \begin{aligned}
        &\,\mathbb{E}\int_G z^2(0)dx + \mathbb{E}\int_\Gamma z_\Gamma^2(0)d\sigma -\left(\mathbb{E}\int_G z^2(t)dx + \mathbb{E}\int_\Gamma z_\Gamma^2(t)d\sigma\right)\\
        &\leq -2\beta_0\mathbb{E}\int_0^t\int_G |\nabla z|^2 dxds+2\|a_1\|_\infty\mathbb{E}\int_0^t\int_G z^2dxds+\varepsilon\mathbb{E}\int_0^t\int_G |\nabla z|^2 \,dxds\\
&\quad\,+\frac{\|B_1\|_\infty^2}{\varepsilon}\mathbb{E}\int_0^t\int_G z^2 \,dxds-\mathbb{E}\int_0^t\int_G Z^2 dxds-2\beta_0\mathbb{E}\int_0^t\int_\Gamma |\nabla_\Gamma z_\Gamma|^2 d\sigma ds\\
&\quad\,+2\|a_2\|_\infty\mathbb{E}\int_0^t\int_\Gamma z_\Gamma^2d\sigma ds+\varepsilon\mathbb{E}\int_0^t\int_\Gamma |\nabla_\Gamma z_\Gamma|^2 \,d\sigma ds+\frac{\|B_2\|_\infty^2}{\varepsilon}\mathbb{E}\int_0^t\int_\Gamma z^2_\Gamma \,d\sigma ds\\
&\quad\,-\mathbb{E}\int_0^t\int_\Gamma \widehat{Z}^2d\sigma ds.
    \end{aligned}
\end{align}
Taking a small $\varepsilon$ in \eqref{4.6ine}, we conclude that
\begin{align*}
    \begin{aligned}
\mathbb{E}\int_G z^2(0)dx + \mathbb{E}\int_\Gamma z_\Gamma^2(0)d\sigma&\leq\,\mathbb{E}\int_G z^2(t)dx + \mathbb{E}\int_\Gamma z_\Gamma^2(t)d\sigma\\
&\hspace{0.4cm}+C\textbf{K}_2\left(\mathbb{E}\int_0^t\int_G z^2dxds+\mathbb{E}\int_0^t\int_\Gamma z_\Gamma^2d\sigma ds\right),
    \end{aligned}
\end{align*}
where $\textbf{K}_2\equiv \|a_1\|_\infty+\|a_2\|_\infty+\|B_1\|_\infty^2+\|B_2\|_\infty^2$. Therefore, by Gronwall's inequality, it follows that
\begin{align}\label{4.6inesec}
    \begin{aligned}
\mathbb{E}\int_G z^2(0)dx + \mathbb{E}\int_\Gamma z_\Gamma^2(0)d\sigma\leq e^{CT\textbf{K}_2}\left(\mathbb{E}\int_G z^2(t)dx + \mathbb{E}\int_\Gamma z_\Gamma^2(t)d\sigma\right).
    \end{aligned}
\end{align}
Finally, integrating \eqref{4.6inesec} on $(T/4,3T/4)$ and combining the obtained inequality with \eqref{4.2toOBSIEN}, we conclude the observability inequality \eqref{4.4}.
\end{proof}

In order to prove Theorem \ref{thmm1.2null}, we first recall the following classical result from functional analysis (see, e.g., \cite{russell73} and \cite[Theorem IV.2.2]{Zabczyk}).
\begin{lm}\label{lm4.1}
Let $\mathbf{H}$, $\mathbf{U}$ and $\mathbf{Y}$ three Hilbert spaces and two bounded linear operators $M\in\mathcal{L}(\mathbf{H};\mathbf{Y})$ and $L\in\mathcal{L}(\mathbf{U};\mathbf{Y})$. Then we have
\begin{equation}\label{condt41}
\| M^*y\|_{\mathbf{H}}\leq\mathcal{C}\| L^*y\|_{\mathbf{U}}\,,\quad\forall y\in \mathbf{Y},\quad\textnormal{for some constant}\;\; \mathcal{C}>0,
\end{equation}
 \begin{center}
     if and only if 
 \end{center}
\begin{equation}\label{condt42}
\mathcal{R}(M)\subset\mathcal{R}(L)\quad
\textnormal{i.e.,}
\quad\forall x_0\in \mathbf{H},\,\,\exists u\in \mathbf{U},\,\,\textnormal{s.t.}\,\,Mx_0=Lu,\,\quad\textnormal{and}\,\quad\| u\|_\mathbf{U}\leq \mathcal{C}\| x_0\|_\mathbf{H}.
\end{equation}
Moreover, the constant $\mathcal{C}$ in \eqref{condt41} and \eqref{condt42} is the same.
\end{lm}
Let us first prove the null controllability of the equation \eqref{CE}.
\begin{proof}[Proof of Theorem \ref{thmm1.2null}]
Let us first define the following linear bounded operator 
$$\mathcal{M}_T:L^2_{\mathcal{F}_0}(\Omega;\mathbb{L}^2)\longrightarrow L^2_{\mathcal{F}_T}(\Omega;\mathbb{L}^2),\,\,\,\quad\mathcal{M}_T(y_0,y_{\Gamma,0})=(y(T,\cdot),y_\Gamma(T,\cdot)),$$
where $(y,y_\Gamma)$ is the solution of \eqref{CE} with controls $(u,v_1,v_2)\equiv(0,0,0)$. Also, we consider the following controllability operator
$$\mathcal{L}_T:L^2_\mathcal{F}(0,T;L^2(G_0))\times L^2_\mathcal{F}(0,T;L^2(G))\times L^2_\mathcal{F}(0,T;L^2(\Gamma))\longrightarrow L^2_{\mathcal{F}_T}(\Omega;\mathbb{L}^2),$$
$$\mathcal{L}_T(u,v_1,v_2)=(y(T,\cdot),y_\Gamma(T,\cdot)),$$
where $(y,y_\Gamma)$ is the solution of \eqref{CE} with initial data $(y_0,y_{\Gamma,0})\equiv(0,0)$.

Multiplying the equations \eqref{4.1} and \eqref{CE} using Itô's formula, we find the following duality relation:
\begin{align}\label{ch26.05}
\begin{aligned}
&\,\langle (y(T,\cdot),y_\Gamma(T,\cdot)),(z_T,z_{\Gamma,T})\rangle_{L^2_{\mathcal{F}_T}(\Omega;\mathbb{L}^2)} - \langle (y_0,y_{\Gamma,0}),(z(0,\cdot),z_\Gamma(0,\cdot))\rangle_{L^2_{\mathcal{F}_0}(\Omega;\mathbb{L}^2)}\\
&= \mathbb{E}\iint_{Q}\mathbbm{1}_{G_0} uz\, dx dt+\mathbb{E}\iint_{Q} v_1Z \,dx dt+\mathbb{E}\iint_{\Sigma} v_2\widehat{Z} \,d\sigma dt.
\end{aligned}
\end{align}
Therefore,  it is easy to see that the adjoint operators of $\mathcal{M}_T$ and $\mathcal{L}_T$ are respectively  defined by
\begin{equation}\label{6.6601}
\mathcal{M}_T^*(z_T,z_{\Gamma,T})=(z(0,\cdot),z_\Gamma(0,\cdot))\,\,\quad\textnormal{and}\quad\,\,\mathcal{L}_T^*(z_T,z_{\Gamma,T})=(\mathbbm{1}_{G_0}z,Z,\widehat{Z}),
\end{equation}
for all $(z_T,z_{\Gamma,T})\in L^2_{\mathcal{F}_T}(\Omega;\mathbb{L}^2)$.

From the observability inequality \eqref{4.4}, it is straightforward to see that
for all $(z_T,z_{\Gamma,T})\in L^2_{\mathcal{F}_T}(\Omega;\mathbb{L}^2)$,
\begin{align}\label{ineeadj4.7}
\|\mathcal{M}_T^*(z_T,z_{\Gamma,T})\|^2_{L^2_{\mathcal{F}_0}(\Omega;\mathbb{L}^2)}\leq e^{CK}\|\mathcal{L}_T^*(z_T,z_{\Gamma,T})\|^2_{L^2_\mathcal{F}(0,T;L^2(G))\times L^2_\mathcal{F}(0,T;L^2(G))\times L^2_\mathcal{F}(0,T;L^2(\Gamma))},
\end{align}
where \begin{align*}
K \equiv 1 + \frac{1}{T} + \| a_1 \|_\infty^{2/3}+ \| a_2 \|_\infty^{2/3} + T (\| a_1 \|_\infty+\| a_2 \|_\infty) + (1+T) (\| B_1 \|_\infty^2+\| B_2 \|_\infty^2).
\end{align*}
Let $(y_0, y_{\Gamma,0}) \in{L^2_{\mathcal{F}_0}(\Omega; \mathbb{L}^2)}$ be the initial state of \eqref{CE}. From Lemma \ref{lm4.1} and the inequality \eqref{ineeadj4.7}, we deduce that \(\mathcal{R}(\mathcal{M}_T) \subset \mathcal{R}(\mathcal{L}_T)\), and hence there exists controls \((\widehat{u}, \widehat{v}_1, \widehat{v}_2) \in L^2_\mathcal{F}(0,T; L^2(G_0)) \times L^2_\mathcal{F}(0,T; L^2(G)) \times L^2_\mathcal{F}(0,T; L^2(\Gamma))\) such that
\begin{align}\label{estmmcos}
\| \widehat{u} \|^2_{L^2_\mathcal{F}(0,T; L^2(G_0))} + \| \widehat{v}_1 \|^2_{L^2_\mathcal{F}(0,T; L^2(G))} + \| \widehat{v}_2 \|^2_{L^2_\mathcal{F}(0,T; L^2(\Gamma))} \leq e^{CK} \| (y_0, y_{\Gamma,0}) \|^2_{L^2_{\mathcal{F}_0}(\Omega; \mathbb{L}^2)}.
\end{align}
On the other hand, it is obvious that \eqref{CE} is null controllable at time \( T \) if and only if \( \mathcal{R}(\mathcal{M}_T) \subset \mathcal{R}(\mathcal{L}_T) \). Therefore, the equation \eqref{CE} is null controllable at time \( T \), and from \eqref{estmmcos}, the desired estimate \eqref{1.2201} holds as well. This concludes the proof of Theorem \ref{thmm1.2null}.
\end{proof}

\section{Application to an Insensitizing Control Problem for System \eqref{1.4}}\label{sec6}
Following a standard approach (see, e.g., the proof of Proposition~3.1 in \cite{barboelouk}), the insensitizing control problem \eqref{inspb} can be reduced to a null controllability problem for the following cascade system of coupled  forward--backward anisotropic stochastic  parabolic equations:
\begin{equation}\label{forr4.1}
\begin{cases}
\begin{array}{ll}
dy - \nabla\cdot(\mathcal{A}\nabla y)\,dt
= \bigl(\xi+a_1 y + B_1 \cdot \nabla y + \mathbbm{1}_{G_0} u\bigr)\,dt + v_1\,dW(t)
& \textnormal{in } Q, \\[0.3em]
dy_\Gamma - \nabla_\Gamma \cdot (\mathcal{A}_\Gamma \nabla_\Gamma y_\Gamma)\,dt
+ \partial_\nu^\mathcal{A} y\,dt
= \bigl(\xi_\Gamma+a_2 y_\Gamma + B_2 \cdot \nabla_\Gamma y_\Gamma\bigr)\,dt + v_2\,dW(t)
& \textnormal{on } \Sigma, \\[0.3em]
dz + \nabla\cdot(\mathcal{A}\nabla z) \,dt = (-a_1 z+\nabla\cdot(zB_1)-\mathbbm{1}_{\mathcal{O}}y) \,dt + Z \,dW(t) &\textnormal{in}\,\,Q,\\[0.3em]

        dz_\Gamma+\nabla_\Gamma\cdot(\mathcal{A}_\Gamma\nabla_\Gamma z_\Gamma) \,dt-\partial^\mathcal{A}_\nu z \,dt = (-a_2z_\Gamma-zB_1\cdot\nu+\nabla_\Gamma\cdot(z_\Gamma B_2-\mathbbm{1}_{\mathcal{O}^2_\Gamma}\nabla_\Gamma y_\Gamma)-\mathbbm{1}_{\mathcal{O}^1_\Gamma}y_\Gamma)\,dt+\widehat{Z} \,dW(t) &\textnormal{on}\,\,\Sigma,\\[0.3em]
y_\Gamma = y|_\Gamma,\quad z_\Gamma=z\vert_\Gamma
& \textnormal{on } \Sigma, \\[0.3em]
(y, y_\Gamma)|_{t=0} = (y_0, y_{\Gamma,0})
& \textnormal{in } G \times \Gamma,\\[0.3em]
(z, z_\Gamma)|_{t=T} = (0, 0)
& \textnormal{in } G \times \Gamma.
\end{array}
\end{cases}
\end{equation}
More precisely, the insensitizing control problem \eqref{inspb} is solvable for a control triple \((u,v_1,v_2)\in \mathcal{U}_T\) if and only if the solution \((y,y_\Gamma; z,z_\Gamma, Z,\widehat{Z})\) of system \eqref{forr4.1} satisfies the following null controllability condition:
\begin{align}\label{null controprop}
(z(0,\cdot), z_\Gamma(0,\cdot)) = (0,0) \quad \textnormal{in } G \times \Gamma, \quad \textnormal{a.s.}
\end{align}

By the classical duality argument, the null controllability problem \eqref{null controprop} for system \eqref{forr4.1} can be reduced to establishing an observability inequality for the following adjoint coupled stochastic parabolic system:
\begin{equation}\label{adback4.77}
\begin{cases}
\begin{array}{ll}
dp + \nabla\cdot(\mathcal{A}\nabla p) \,dt = (-a_1 p+\nabla\cdot(pB_1)-\mathbbm{1}_{\mathcal{O}}q) \,dt + P \,dW(t) &\textnormal{in}\,\,Q,\\[0.3em]
dp_\Gamma+\nabla_\Gamma\cdot(\mathcal{A}_\Gamma\nabla_\Gamma p_\Gamma)  \,dt-\partial^\mathcal{A}_\nu p \,dt = (-a_2p_\Gamma-pB_1\cdot\nu+\nabla_\Gamma\cdot(p_\Gamma B_2-\mathbbm{1}_{\mathcal{O}^2_\Gamma}\nabla_\Gamma q_\Gamma)-\mathbbm{1}_{\mathcal{O}^1_\Gamma}q_\Gamma)\,dt+\widehat{P} \,dW(t) &\textnormal{on}\,\,\Sigma,\\[0.3em]
dq - \nabla\cdot(\mathcal{A}\nabla q) \,dt = (a_1 q+B_1\cdot\nabla q) \,dt &\textnormal{in}\,\,Q,\\[0.3em]
dq_\Gamma-\nabla_\Gamma\cdot(\mathcal{A}_\Gamma\nabla_\Gamma q_\Gamma) \,dt+\partial^\mathcal{A}_\nu q \,dt = (a_2q_\Gamma+B_2\cdot\nabla_\Gamma q_\Gamma)\,dt &\textnormal{on}\,\,\Sigma,\\[0.3em]
p_\Gamma=p\vert_\Gamma,\quad q_\Gamma=q\vert_\Gamma &\textnormal{on}\,\,\Sigma,\\[0.3em]
(p,p_\Gamma)\vert_{t=T}=(0,0) &\textnormal{in}\,\,G\times\Gamma,\\[0.3em]
(q,q_\Gamma)\vert_{t=0}=(q_0,q_{\Gamma,0}) &\textnormal{in}\,\,G\times\Gamma,
			\end{array}
		\end{cases}
\end{equation}
where \((q_0, q_{\Gamma,0}) \in L^2_{\mathcal{F}_0}(\Omega; \mathbb{L}^2)\). 
The corresponding observability inequality is stated as follows.

\begin{prop}\label{propo5.1obseine}
Assume that $G_0 \cap \mathcal{O} \neq \emptyset$, and let $\mathcal{O}_\Gamma^1, \mathcal{O}_\Gamma^2 \subset \Gamma$ be arbitrary nonempty open subsets. Then there exist positive constants \(C\) and \(K\), depending only on \(G\), \(G_0\), \(\mathcal{O}\), \(\mathcal{O}^1_\Gamma\), \(\mathcal{O}^2_\Gamma\), \(T\), $\beta_0$, $M$, $M_\Gamma$, \(a_1\), \(a_2\), \(B_1\), and \(B_2\), such that for any \((q_0, q_{\Gamma,0}) \in L^2_{\mathcal{F}_0}(\Omega; \mathbb{L}^2)\), the solution \((p,p_\Gamma,P,\widehat{P}; q,q_\Gamma)\) of the coupled system \eqref{adback4.77} satisfies
\begin{align}\label{obseine4.9}
\begin{aligned}
&\mathbb{E} \iint_Q \exp\big(-K t^{-1}\big) \, |p|^2 \, dx \, dt
+ \mathbb{E} \iint_\Sigma \exp\big(-K t^{-1}\big) \, |p_\Gamma|^2 \, d\sigma \, dt \\
&\leq C \bigg( \mathbb{E} \iint_{Q_0} |p|^2 \, dx \, dt
+ \mathbb{E} \iint_Q |P|^2 \, dx \, dt
+ \mathbb{E} \iint_\Sigma |\widehat{P}|^2 \, d\sigma \, dt \bigg).
\end{aligned}
\end{align}
\end{prop}
Hence, the insensitizing control problem \eqref{inspb} is reduced to establishing the observability inequality \eqref{obseine4.9}. In other words, the full state \((p, p_\Gamma)\) can be reconstructed from partial observations of \(p\) on \(G_0\) together with complete information on \(P\) and \(\widehat{P}\). Classically, it is straightforward to see that the proof of the observability inequality \eqref{obseine4.9} follows from the following Carleman estimate for the coupled system \eqref{adback4.77}.

\begin{prop}\label{carlfor insencontro}
Assume that $G_0 \cap \mathcal{O} \neq \emptyset$, and let  $\mathcal{O}_\Gamma^1, \mathcal{O}_\Gamma^2 \subset \Gamma$ be arbitrary nonempty open subsets. Then there exist a constant \(C>0\) and sufficiently large parameters \(\overline{\lambda}, \overline{\mu} \geq 1\), depending only on \(G\), \(G_0\), \(\mathcal{O}\),  \(\mathcal{O}^1_\Gamma\), \(\mathcal{O}^2_\Gamma\), \(\beta_0\), $M$, and $M_\Gamma$ such that for \(\mu = \overline{\mu}\) and for any \((q_0,q_{\Gamma,0}) \in L^2_{\mathcal{F}_0}(\Omega; \mathbb{L}^2)\), the associated solution \((p,p_\Gamma,P,\widehat{P};q,q_\Gamma)\) of system \eqref{adback4.77} satisfies 
\begin{align}\label{carlestcasca}
\begin{aligned}
& \lambda^3 \, \mathbb{E} \iint_Q \theta^2 \gamma^3 |p|^2 \, dx \, dt 
+ \lambda^3 \, \mathbb{E} \iint_\Sigma \theta^2 \gamma^3 |p_\Gamma|^2 \, d\sigma \, dt
+ \lambda^4 \, \mathbb{E} \iint_Q \theta^2 \gamma^4 |q|^2 \, dx \, dt \\
& + \lambda^4 \, \mathbb{E} \iint_\Sigma \theta^2 \gamma^4 |q_\Gamma|^2 \, d\sigma \, dt
+ \lambda\mathbb{E} \iint_Q \theta^2 \gamma |\nabla p|^2 \, dx \, dt
+ \lambda\mathbb{E} \iint_\Sigma \theta^2 \gamma |\nabla_\Gamma p_\Gamma|^2 \, d\sigma \, dt \\
& + \lambda^2\mathbb{E} \iint_Q \theta^2 \gamma^2 |\nabla q|^2 \, dx \, dt
+ \lambda^2\mathbb{E} \iint_\Sigma \theta^2 \gamma^2 |\nabla_\Gamma q_\Gamma|^2 \, d\sigma \, dt \\
& \leq C \bigg[ \lambda^{13} \, \mathbb{E} \iint_{Q_0} \theta^2 \gamma^{13} |p|^2 \, dx \, dt
+ \lambda^6 \, \mathbb{E} \iint_Q \theta^2 \gamma^6 |P|^2 \, dx \, dt
+ \lambda^2 \, \mathbb{E} \iint_\Sigma \theta^2 \gamma^2 |\widehat{P}|^2 \, d\sigma \, dt \bigg],
\end{aligned}
\end{align}
for all
\[\lambda \geq \overline{\lambda}\Big[T+T^2\Big(1+\|a_1\|_\infty^{2/3}+\|a_1\|_\infty^{1/2}+\|a_2\|_\infty^{2/3}+\|B_1\|_\infty^{2}+\|B_1\|_\infty+\|B_2\|_\infty^{2} \Big)\Big].\]
\end{prop}

In order to prove the Carleman estimate \eqref{carlestcasca}, we first state the following Carleman estimate for the general forward stochastic parabolic equation
\begin{equation}\label{forq5.1}
\begin{cases}
\begin{array}{ll}
dq - \nabla\cdot(\mathcal{A}\nabla q) \,dt = f \,dt + g \,dW(t) &\textnormal{in}\,\,Q,\\
dq_\Gamma-\nabla_\Gamma\cdot(\mathcal{A}_\Gamma\nabla_\Gamma q_\Gamma) \,dt+\partial_\nu^\mathcal{A} q \,dt = f_\Gamma\,dt+g_\Gamma \,dW(t) &\textnormal{on}\,\,\Sigma,\\
q_\Gamma=q\vert_\Gamma &\textnormal{on}\,\,\Sigma,\\
(q,q_\Gamma)\vert_{t=0}=(q_0,q_{\Gamma,0}) &\textnormal{in}\,\,G\times\Gamma,
\end{array}
\end{cases}
\end{equation}
where \((q_0,q_{\Gamma,0}) \in L^2_{\mathcal{F}_0}(\Omega; \mathbb{L}^2)\) is the initial condition, \(f, g \in L^2_\mathcal{F}(0,T;L^2(G))\), and \(f_\Gamma, g_\Gamma \in L^2_\mathcal{F}(0,T;L^2(\Gamma))\).
\begin{lm}\label{lmm4.2}
Let $\mathcal{B}\Subset G$ be a nonempty open subset of $G$ and $d\in\mathbb{R}$. There exists a large $\mu_1\geq1$ such that for $\mu=\mu_1$, one can find a constant $C>0$ and a large $\lambda_1\geq 1$ depending only on $G$, $\mathcal{B}$, $\mu_1$, $d$, \(\beta_0\), $M$, and $M_\Gamma$ such that for any $f, g\in L^2_\mathcal{F}(0,T;L^2(G))$, $f_\Gamma, g_\Gamma\in L^2_\mathcal{F}(0,T;L^2(\Gamma))$, and $(q_0,q_{\Gamma,0})\in L^2_{\mathcal{F}_0}(\Omega;\mathbb{L}^2)$, the associated solution $(q,q_\Gamma)$ of \eqref{forq5.1} satisfies  that
\begin{align}\label{carfro5.2}
\begin{aligned}
&\;\lambda^d\mathbb{E}\iint_Q \theta^2\gamma^d|q|^2\,dx\,dt +\lambda^d\mathbb{E}\iint_\Sigma \theta^2\gamma^d|q_\Gamma|^2\,d\sigma\,dt\\
 &+\lambda^{d-2}\mathbb{E}\iint_Q \theta^2\gamma^{d-2} \vert\nabla q\vert^2\,dx\,dt+ \lambda^{d-2}\mathbb{E}\iint_\Sigma \theta^2\gamma^{d-2} \vert\nabla_\Gamma q_\Gamma\vert^2 \,d\sigma\,dt
\\&\leq C \bigg[ \lambda^d\mathbb{E}\int_0^T\int_{\mathcal{B}} \theta^2\gamma^d |q|^2 \,dx\,dt+ \lambda^{d-3}\mathbb{E}\iint_Q \theta^2\gamma^{d-3}|f|^2 \,dx\,dt+\lambda^{d-3}\mathbb{E}\iint_\Sigma \theta^2\gamma^{d-3}|f_\Gamma|^2 \,d\sigma\,dt\\
&\hspace{0.8cm}+\lambda^{d-1}\mathbb{E}\iint_Q \theta^2\gamma^{d-1}|g|^2 \,dx\,dt+\lambda^{d-1}\mathbb{E}\iint_\Sigma \theta^2\gamma^2|g_\Gamma|^{d-1} \,d\sigma\,dt\bigg], 
\end{aligned}\end{align}
for all $\lambda\geq \lambda_1(T+T^2)$.
\end{lm}

Notice that, for $d=3$, the Carleman estimate \eqref{carfro5.2} was established in \cite[Theorem~3.1]{elgroubacconvdbc}. For $d \neq 3$, we apply this estimate to the weighted pair $\Big((\lambda\gamma)^{\frac{d-3}{2}} q, (\lambda\gamma)^{\frac{d-3}{2}} q_\Gamma\Big)$ instead of $(q,q_\Gamma)$. By choosing $\lambda$ sufficiently large, the resulting lower-order terms can be absorbed, which yields the desired estimate \eqref{carfro5.2} for any real $d$.

In the sequel, we fix \(\mu = \overline{\mu} := \max(\mu_0,\mu_1)\), where \(\mu_0\) (resp., \(\mu_1\)) is the constant given in Theorem~\ref{thm1.1} (resp., Lemma~\ref{lmm4.2}). We are now ready to prove the key Carleman estimate \eqref{carlestcasca} of this section.

\begin{proof}[Proof of Proposition \ref{carlfor insencontro}] Let \( w_m = \theta^2 (\lambda \gamma)^m \) with \( m \in \mathbb{N} \), and consider the subsets \( \widetilde{G}_i \) (\( i = 1, 2 \)) such that
\begin{equation*}
\widetilde{G}_2 \Subset \widetilde{G}_1 \Subset \widetilde{G}_0 \subset G_0 \cap \mathcal{O}.
\end{equation*}
Moreover, we introduce smooth functions \( \rho_i \in C^{\infty}(\mathbb{R}^N) \) satisfying
\begin{align}\label{assmzeta}
\begin{aligned}
& 0 \leq \rho_i \leq 1, \quad \rho_i = 1 \,\, \text{in} \,\, \widetilde{G}_{3-i}, \quad \text{Supp}(\rho_i) \subset \widetilde{G}_{2-i}, \\ 
& \frac{\nabla \rho_i}{\rho_i^{1/2}} \in L^\infty(G; \mathbb{R}^N), \quad
\frac{\Delta \rho_i}{\rho_i^{1/2}} \in L^\infty(G) , \quad i=1,2.
\end{aligned}
\end{align}
Such functions \(\rho_1\) and \(\rho_2\) indeed exist. To see this, consider any function \(\rho \in C^{\infty}(\mathbb{R}^N)\) satisfying
\begin{align}\label{conctfof}
0 \leq \rho \leq 1, \quad \rho = 1 \,\, \text{in} \,\, S \subset \widetilde{S} \subset \mathbb{R}^N, \quad \text{Supp}(\rho) \subset \widetilde{S}.
\end{align}
One can then choose a standard cut-off function \(\rho_0 \in C_0^\infty(\mathbb{R}^N)\) satisfying \eqref{conctfof} and set \(\rho = \rho_0^4\). This choice ensures
\[
\frac{\nabla \rho}{\rho^{1/2}} \in L^\infty(G; \mathbb{R}^N), \qquad \frac{\Delta \rho}{\rho^{1/2}} \in L^\infty(G).
\]
Using \eqref{3.3} and \eqref{assmzeta}, it follows that, for sufficiently large \(\lambda \geq C(T + T^2)\), one has
\begin{align}\label{esttmforT}
|\partial_t w_m| \leq C \lambda^{m+2} \theta^2 \gamma^{m+2}, \qquad
|\nabla(w_m \rho_i)| \leq C \lambda^{m+1} \theta^2 \gamma^{m+1} \rho_i^{1/2}, \quad i=1,2.
\end{align}

We first apply the Carleman estimate \eqref{1.3}  for the backward equation in \eqref{adback4.77} with \(\mathcal{B} = \widetilde{G}_2\). Then, there exist a constant \(C>0\) and a sufficiently large \(\lambda_0 \geq 1\) such that, for all 
\[\lambda \geq \lambda_0\Big[T+T^2\Big(1+\|a_1\|_\infty^{2/3}+\|a_2\|_\infty^{2/3}+\|B_1\|_\infty^{2}+\|B_2\|_\infty^{2}\Big)\Big],\]
we have
\begin{align}\label{estt4.3f}
 \begin{aligned}
& \lambda^3 \, \mathbb{E}\iint_Q \theta^2 \gamma^3 p^2 \, dx \, dt 
+ \lambda^3 \, \mathbb{E}\iint_\Sigma \theta^2 \gamma^3 p_\Gamma^2 \, d\sigma \, dt \\
& + \lambda \, \mathbb{E}\iint_Q \theta^2 \gamma |\nabla p|^2 \, dx \, dt 
+ \lambda \, \mathbb{E}\iint_\Sigma \theta^2 \gamma |\nabla_\Gamma p_\Gamma|^2 \, d\sigma \, dt \\
& \leq C \bigg[ \lambda^3 \, \mathbb{E} \int_0^T \int_{\widetilde{G}_2} \theta^2 \gamma^3 p^2 \, dx \, dt
+ \mathbb{E} \int_0^T \int_\mathcal{O} \theta^2 q^2 \, dx \, dt+ \mathbb{E} \int_0^T \int_{\mathcal{O}^1_\Gamma} \theta^2 q_\Gamma^2 \, d\sigma \, dt \\
& \qquad + \lambda^2\mathbb{E} \int_0^T \int_{\mathcal{O}^2} \theta^2\gamma^2 |\nabla q|^2 \, dx \, dt+ \lambda^2\mathbb{E} \int_0^T \int_{\mathcal{O}^2_\Gamma} \theta^2\gamma^2 |\nabla_\Gamma q_\Gamma|^2 \, d\sigma \, dt
\\
& \qquad+ \lambda^2 \, \mathbb{E}\iint_Q \theta^2 \gamma^2 P^2 \, dx \, dt
+ \lambda^2 \, \mathbb{E}\iint_\Sigma \theta^2 \gamma^2 \widehat{P}^2 \, d\sigma \, dt \bigg].
\end{aligned}
\end{align}
On the other hand, by applying the Carleman estimate \eqref{carfro5.2} for the forward equation in \eqref{adback4.77} with \(\mathcal{B} = \widetilde{G}_2\) and $d=4$, there exist a constant \(C>0\) and a large \(\lambda_1 \geq 1\) such that, for all 
\[\lambda \geq \lambda_1\Big[T+T^2\Big(1+\|a_1\|_\infty^{2/3}+\|a_2\|_\infty^{2/3}+\|B_1\|_\infty^{2}+\|B_2\|_\infty^{2}\Big)\Big],\]
one has
\begin{align}\label{estimm5.4}
\begin{aligned}
& \lambda^4 \, \mathbb{E}\iint_Q \theta^2 \gamma^4 q^2 \, dx \, dt
+ \lambda^4 \, \mathbb{E}\iint_\Sigma \theta^2 \gamma^4 q_\Gamma^2 \, d\sigma \, dt \\
& + \lambda^2 \, \mathbb{E}\iint_Q \theta^2 \gamma^2 |\nabla q|^2 \, dx \, dt 
+ \lambda^2 \, \mathbb{E}\iint_\Sigma \theta^2 \gamma^2 |\nabla_\Gamma q_\Gamma|^2 \, d\sigma \, dt \\
& \leq C \lambda^4 \, \mathbb{E} \int_0^T \int_{\widetilde{G}_2} \theta^2 \gamma^4 q^2 \, dx \, dt.
\end{aligned}
\end{align}
Combining \eqref{estt4.3f} and \eqref{estimm5.4}, and taking  
\[\lambda \geq \max(\lambda_0,\lambda_1)\Big[T+T^2\Big(1+\|a_1\|_\infty^{2/3}+\|a_2\|_\infty^{2/3}+\|B_1\|_\infty^{2}+\|B_2\|_\infty^{2}\Big)\Big],\]
we obtain
\begin{align}\label{estt4.3fsec}
 \begin{aligned}
& \lambda^3 \, \mathbb{E}\iint_Q \theta^2 \gamma^3 p^2 \, dx \, dt 
+ \lambda^3 \, \mathbb{E}\iint_\Sigma \theta^2 \gamma^3 p_\Gamma^2 \, d\sigma \, dt
+ \lambda \, \mathbb{E}\iint_Q \theta^2 \gamma |\nabla p|^2 \, dx \, dt \\
& + \lambda \, \mathbb{E}\iint_\Sigma \theta^2 \gamma |\nabla_\Gamma p_\Gamma|^2 \, d\sigma \, dt
+ \lambda^4 \, \mathbb{E}\iint_Q \theta^2 \gamma^4 q^2 \, dx \, dt
+ \lambda^4 \, \mathbb{E}\iint_\Sigma \theta^2 \gamma^4 q_\Gamma^2 \, d\sigma \, dt \\
& + \lambda^2 \, \mathbb{E}\iint_Q \theta^2 \gamma^2 |\nabla q|^2 \, dx \, dt
+ \lambda^2 \, \mathbb{E}\iint_\Sigma \theta^2 \gamma^2 |\nabla_\Gamma q_\Gamma|^2 \, d\sigma \, dt \\
& \leq C \bigg[ \lambda^3 \, \mathbb{E} \int_0^T \int_{\widetilde{G}_2} \theta^2 \gamma^3 p^2 \, dx \, dt
+ \lambda^4 \, \mathbb{E} \int_0^T \int_{\widetilde{G}_2} \theta^2 \gamma^4 q^2 \, dx \, dt \\
& \qquad\quad + \lambda^2 \, \mathbb{E}\iint_Q \theta^2 \gamma^2 P^2 \, dx \, dt
+ \lambda^2 \, \mathbb{E}\iint_\Sigma \theta^2 \gamma^2 \widehat{P}^2 \, d\sigma \, dt \bigg].
\end{aligned}
\end{align}
We now absorb the second localized term on the right-hand side of \eqref{estt4.3fsec}. First, note that
\begin{align}\label{firine}
\lambda^4 \, \mathbb{E} \int_0^T \int_{\widetilde{G}_2} \theta^2 \gamma^4 q^2 \, dx \, dt \leq \mathbb{E} \iint_Q \rho_1 w_4 q^2 \, dx \, dt.
\end{align}
Next, applying Itô's formula to \(d(\rho_1 w_4 p q)\), integrating over \(Q\), and taking the expectation, we obtain
\begin{align}\label{ineqq6.15}
\begin{aligned}
\mathbb{E} \iint_Q \rho_1 w_4 q^2 \, dx \, dt 
= & \, \mathbb{E} \iint_Q \rho_1 \, \partial_t(w_4) \, p q \, dx \, dt 
+ \mathbb{E} \iint_Q q \, \mathcal{A} \nabla p \cdot \nabla(\rho_1 w_4) \, dx \, dt \\
& - \mathbb{E} \iint_Q p \, \mathcal{A} \nabla q \cdot \nabla(\rho_1 w_4) \, dx \, dt
- \mathbb{E} \iint_Q p q \, B_1 \cdot \nabla(\rho_1 w_4) \, dx \, dt.
\end{aligned}
\end{align}
In the right-hand side of \eqref{ineqq6.15}, using \eqref{esttmforT} and Young's inequality, we obtain that for any \(\varepsilon > 0\) and sufficiently large \(\lambda\geq C\|B_1\|_\infty T^2\),
\begin{align}\label{lasineaf515}
\begin{aligned}
\mathbb{E}\iint_Q \rho_1 w_4 q^2 \, dx \, dt
\leq &\, 3 \varepsilon \, \lambda^4 \, \mathbb{E}\iint_Q \theta^2 \gamma^4 q^2 \, dx \, dt
+ \varepsilon \, \lambda^2 \, \mathbb{E}\iint_Q \theta^2 \gamma^2 |\nabla q|^2 \, dx \, dt \\
& + \frac{C}{\varepsilon} \, \lambda^8 \, \mathbb{E}\iint_Q \theta^2 \gamma^8 \rho_1 p^2 \, dx \, dt
+ \frac{C}{\varepsilon} \, \lambda^6 \, \mathbb{E}\iint_Q \theta^2 \gamma^6 \rho_1 |\nabla p|^2 \, dx \, dt.
\end{aligned}
\end{align}
Combining \eqref{lasineaf515}, \eqref{firine}, and \eqref{estt4.3fsec}, and choosing \(\varepsilon\) small enough and sufficiently large
\[\lambda \geq C\Big[T+T^2\Big(1+\|a_1\|_\infty^{2/3}+\|a_2\|_\infty^{2/3}+\|B_1\|_\infty^{2}+\|B_1\|_\infty +\|B_2\|_\infty^{2}\Big)\Big],\]
we get
\begin{align}\label{estt4.3fsecfi}
\begin{aligned}
& \lambda^3 \, \mathbb{E}\iint_Q \theta^2 \gamma^3 p^2 \, dx \, dt
+ \lambda^3 \, \mathbb{E}\iint_\Sigma \theta^2 \gamma^3 p_\Gamma^2 \, d\sigma \, dt
+ \lambda \, \mathbb{E}\iint_Q \theta^2 \gamma |\nabla p|^2 \, dx \, dt \\
& + \lambda \, \mathbb{E}\iint_\Sigma \theta^2 \gamma |\nabla_\Gamma p_\Gamma|^2 \, d\sigma \, dt
+ \lambda^4 \, \mathbb{E}\iint_Q \theta^2 \gamma^4 q^2 \, dx \, dt
+ \lambda^4 \, \mathbb{E}\iint_\Sigma \theta^2 \gamma^4 q_\Gamma^2 \, d\sigma \, dt \\
& + \lambda^2 \, \mathbb{E}\iint_Q \theta^2 \gamma^2 |\nabla q|^2 \, dx \, dt
+ \lambda^2 \, \mathbb{E}\iint_\Sigma \theta^2 \gamma^2 |\nabla_\Gamma q_\Gamma|^2 \, d\sigma \, dt \\
& \leq C \bigg[ \lambda^8 \, \mathbb{E} \int_0^T \int_{\widetilde{G}_1} \theta^2 \gamma^8 p^2 \, dx \, dt
+ \lambda^6 \, \mathbb{E} \int_0^T \int_{\widetilde{G}_1} \theta^2 \gamma^6 |\nabla p|^2 \, dx \, dt \\
& \qquad + \lambda^2 \, \mathbb{E}\iint_Q \theta^2 \gamma^2 P^2 \, dx \, dt
+ \lambda^2 \, \mathbb{E}\iint_\Sigma \theta^2 \gamma^2 \widehat{P}^2 \, d\sigma \, dt \bigg].
\end{aligned}
\end{align}
Next, we absorb the second localized integral on the right-hand side of \eqref{estt4.3fsecfi}. First, observe that
\begin{align}\label{ineqqlas6.18}
\lambda^6 \, \mathbb{E} \int_0^T \int_{\widetilde{G}_1} \theta^2 \gamma^6 |\nabla p|^2 \, dx \, dt
\leq \frac{1}{\beta_0} \, \mathbb{E} \iint_Q \rho_2 w_6 \, \mathcal{A} \nabla p \cdot \nabla p \, dx \, dt.
\end{align}
On the other hand, applying Itô's formula to \(d(\rho_2 w_5 p^2)\), we derive
\begin{align*}
\begin{aligned}
2 \, \mathbb{E} \iint_Q \rho_2 w_6 \mathcal{A} \nabla p \cdot \nabla p \, dx \, dt
= & - \mathbb{E} \iint_Q \rho_2 \, \partial_t(w_6) \, p^2 \, dx \, dt
- 2 \, \mathbb{E} \iint_Q p \, \mathcal{A} \nabla p \cdot \nabla (\rho_2 w_6) \, dx \, dt \\
& + 2 \, \mathbb{E} \iint_Q \rho_2 w_6 a_1 p^2 \, dx \, dt
+ 2 \, \mathbb{E} \iint_Q p B_1 \cdot \nabla (\rho_2 w_6 p) \, dx \, dt \\
& + 2 \, \mathbb{E} \iint_Q \rho_2 w_6 p q \, dx \, dt
- \mathbb{E} \iint_Q \rho_2 w_6 P^2 \, dx \, dt.
\end{aligned}
\end{align*}
Using \eqref{esttmforT} and Young's inequality, it follows that for any \(\varepsilon > 0\) and sufficiently large 
\(\lambda\geq C(\|a_1\|_\infty^{1/2}+\|B_1\|_\infty) T^2\), we derive
\begin{align}\label{ineqqafr6.18}
\begin{aligned}
2 \, \mathbb{E} \iint_Q \rho_2 w_5 \mathcal{A} \nabla p \cdot \nabla p \, dx \, dt
\leq & \, 2 \varepsilon \, \lambda \, \mathbb{E} \iint_Q \theta^2 \gamma |\nabla p|^2 \, dx \, dt
+ \varepsilon \, \lambda^4 \, \mathbb{E} \iint_Q \theta^2 \gamma^4 q^2 \, dx \, dt \\
& + C \, \lambda^8 \, \mathbb{E} \iint_Q \theta^2 \gamma^8 \rho_2 p^2 \, dx \, dt
+ \frac{C}{\varepsilon} \, \lambda^{13} \, \mathbb{E} \iint_Q \theta^2 \gamma^{13} \rho_2 p^2 \, dx \, dt \\
& + C \, \lambda^6 \, \mathbb{E} \iint_Q \theta^2 \gamma^6 P^2 \, dx \, dt.
\end{aligned}
\end{align}
Finally, combining \eqref{ineqqafr6.18}, \eqref{ineqqlas6.18}, and \eqref{estt4.3fsecfi}, and taking \(\varepsilon\) sufficiently small and a large 
\[\lambda \geq C\Big[T+T^2\Big(1+\|a_1\|_\infty^{2/3}+\|a_1\|_\infty^{1/2}+\|a_2\|_\infty^{2/3}+\|B_1\|_\infty^{2}+\|B_2\|_\infty^{2}+\|B_1\|_\infty \Big)\Big],\]
we obtain
\begin{align}\label{estt4.3fsecfifinal}
\begin{aligned}
& \lambda^3 \, \mathbb{E}\iint_Q \theta^2 \gamma^3 p^2 \, dx \, dt
+ \lambda^3 \, \mathbb{E}\iint_\Sigma \theta^2 \gamma^3 p_\Gamma^2 \, d\sigma \, dt
+ \lambda \, \mathbb{E}\iint_Q \theta^2 \gamma |\nabla p|^2 \, dx \, dt \\
& + \lambda \, \mathbb{E}\iint_\Sigma \theta^2 \gamma |\nabla_\Gamma p_\Gamma|^2 \, d\sigma \, dt
+ \lambda^4 \, \mathbb{E}\iint_Q \theta^2 \gamma^4 q^2 \, dx \, dt
+ \lambda^4 \, \mathbb{E}\iint_\Sigma \theta^2 \gamma^4 q_\Gamma^2 \, d\sigma \, dt \\
& + \lambda^2 \, \mathbb{E}\iint_Q \theta^2 \gamma^2 |\nabla q|^2 \, dx \, dt
+ \lambda^2 \, \mathbb{E}\iint_\Sigma \theta^2 \gamma^2 |\nabla_\Gamma q_\Gamma|^2 \, d\sigma \, dt \\
& \leq C \bigg[ \lambda^{13} \, \mathbb{E} \iint_{Q_0} \theta^2 \gamma^{13} p^2 \, dx \, dt
+ \lambda^6 \, \mathbb{E} \iint_Q \theta^2 \gamma^6 P^2 \, dx \, dt
+ \lambda^2 \, \mathbb{E} \iint_\Sigma \theta^2 \gamma^2 \widehat{P}^2 \, d\sigma \, dt \bigg].
\end{aligned}
\end{align}
This concludes the proof of Proposition \ref{carlfor insencontro}.
\end{proof}

We are now in position to prove the result on the existence of insensitizing controls for equation \eqref{1.4}, stated in Theorem~\ref{thmm1.3ins}.

\begin{proof}[Proof of Theorem \ref{thmm1.3ins}]
Let us define the linear subspace
\[
\mathcal{H} = \Big\{ (\mathbbm{1}_{G_0}p, P, \widehat{P}) \;\Big|\; (p, p_\Gamma, P, \widehat{P}; q, q_\Gamma) \text{ solves } \eqref{adback4.77} \text{ for some } (q_0, q_{0,\Gamma}) \in L^2_{\mathcal{F}_0}(\Omega; \mathbb{L}^2) \Big\}.
\]
Consider a linear functional \(\mathcal{L}\colon \mathcal{H} \to \mathbb{R}\) by
\[
\mathcal{L}(\mathbbm{1}_{G_0}p, P, \widehat{P}) = -\mathbb{E} \iint_Q p \xi \, dx \, dt - \mathbb{E} \iint_\Sigma p_\Gamma \xi_\Gamma \, d\sigma \, dt.
\]
By the observability inequality \eqref{obseine4.9}, \(\mathcal{L}\) is bounded, and its norm satisfies
\[
\|\mathcal{L}\|_{\mathcal{L}(\mathcal{H}; \mathbb{R})}^2
\le C \bigg( \mathbb{E}\iint_Q \exp(K t^{-1})|\xi|^2 \,dx\,dt + \mathbb{E}\iint_\Sigma \exp(K t^{-1})|\xi_\Gamma|^2 \,d\sigma\,dt \bigg).
\]
Extending \(\mathcal{L}\) to the whole space \(\mathcal{U}_T\) via the Hahn–Banach theorem and applying the Riesz representation theorem, there exist controls \((u, v_1, v_2) \in \mathcal{U}_T\) such that
\begin{align}\label{eq:controls}
-\mathbb{E} \iint_Q p \xi \, dx \, dt - \mathbb{E} \iint_\Sigma p_\Gamma \xi_\Gamma \, d\sigma \, dt
= \mathbb{E} \iint_{Q_0} u p \, dx \, dt + \mathbb{E} \iint_Q v_1 P \, dx \, dt + \mathbb{E} \iint_\Sigma v_2 \widehat{P} \, d\sigma \, dt,
\end{align}
with
\[
E(u, v_1, v_2) \le C \bigg( \mathbb{E}\iint_Q \exp(K t^{-1})|\xi|^2 \,dx\,dt + \mathbb{E}\iint_\Sigma \exp(K t^{-1})|\xi_\Gamma|^2 \,d\sigma\,dt \bigg).
\]
Using Itô's formula for \((y, y_\Gamma)\) and \((p, p_\Gamma)\), and for \((z, z_\Gamma)\) and \((q, q_\Gamma)\), we obtain
\begin{align*}
\mathbb{E} \int_G z(0) q_0 \, dx + \mathbb{E} \int_\Gamma z_\Gamma(0) q_{0,\Gamma} \, d\sigma 
= &\,\mathbb{E} \iint_Q \xi p \, dx \, dt + \mathbb{E} \iint_\Sigma \xi_\Gamma p_\Gamma \, d\sigma \, dt + \mathbb{E} \iint_{Q_0} u p \, dx \, dt \\
&+ \mathbb{E} \iint_Q P v_1 \, dx \, dt + \mathbb{E} \iint_\Sigma \widehat{P} v_2 \, d\sigma \, dt.
\end{align*}
This, together with \eqref{eq:controls}, yields
\[
\mathbb{E} \int_G z(0) q_0 \, dx + \mathbb{E} \int_\Gamma z_\Gamma(0) q_{0,\Gamma} \, d\sigma = 0.
\]
Since \((q_0, q_{\Gamma,0})\) is arbitrary in \(L^2_{\mathcal{F}_0}(\Omega; \mathbb{L}^2)\), it follows that
\[
z(0, \cdot) = 0 \;\; \text{in } G, \quad z_\Gamma(0, \cdot) = 0 \;\; \text{on } \Gamma, \quad \mathbb{P}\text{-a.s.}
\]
This completes the proof of Theorem~\ref{thmm1.3ins}.
\end{proof}

\begin{center}
\end{center}
\vspace{2cm}
\begin{flushleft}
$^\dagger$ Cadi Ayyad University, National School of Applied Sciences,  
Laboratory of Mathematics, Modeling and Automatic Systems,  
B.P. 575, Marrakesh, Morocco.  
\\ \quad Email: \texttt{s.boulite@uca.ma}

\vspace{0.4cm}

$^\ddagger$ Dipartimento di Matematica, Università degli Studi di Salerno,  
Via Giovanni Paolo II, 132, 84084 Fisciano (SA), Italy.  
\\\quad Email: \texttt{aelgrou@unisa.it}, \texttt{arhandi@unisa.it}

\vspace{0.4cm}

$^\S$ Cadi Ayyad University, Faculty of Sciences Semlalia,  
Laboratory of Mathematics, Modeling and Automatic Systems,  
B.P. 2390, Marrakesh, Morocco.  
\\\quad Email: \texttt{maniar@uca.ma}

\vspace{0.4cm}

$^\top$ University Mohammed VI Polytechnic,  Vanguard Center, Rabat, Morocco.  
\end{flushleft}
\end{document}